\documentclass{amsart}
\usepackage{xypic}
\usepackage[all,cmtip]{xy}
\usepackage{tikz}
\usetikzlibrary{calc,decorations.markings,decorations.pathmorphing,arrows}
\usepackage{amsmath, amssymb, amsthm, amsfonts, mathrsfs, amsfonts}
\usepackage[a4paper,left = 2.5cm, right= 2.5cm, top = 2.5cm, bottom = 2.5cm]{geometry}
\usepackage{tikz-cd}
\usepackage{cases}
\usepackage{float} % commenting this moves pictures around
\usepackage[margin=10pt,font=small,labelfont=bf,
labelsep=period,indention=.7cm]{caption}
\usepackage{enumitem}
\usepackage{hyperref}
\newcommand{\Z}{\ensuremath{\mathbb{Z}}}
\newcommand{\C}{\ensuremath{\mathbb{C}}}
\newcommand{\R}{\ensuremath{\mathbb{R}}}
\newtheorem{thm}{Theorem}[section]
\newtheorem*{thm2}{Theorem}

\newtheorem{cor}[thm]{Corollary}
\newtheorem{lemma}[thm]{Lemma}
\newtheorem{prop}[thm]{Proposition}
\theoremstyle{remark}
\newtheorem{rmk}[thm]{Remark}

\theoremstyle{definition}
\newtheorem{defin}[thm]{Definition}

\newtheorem{nota}[thm]{Notation}
\newtheorem{ex}[thm]{Example}% calligraphic
\newcommand{\kk}{k}
\newcommand{\e}{\hat\varepsilon}

\def\cal#1{\mathcal{#1}}

\newcommand{\on}{\operatorname}

\allowdisplaybreaks

\begin{document}

\title{Skew group algebras of Jacobian algebras}
\author{Simone Giovannini}
\address{Simone Giovannini: Dipartimento di Matematica, Universit\`{a} di Padova, Via Trieste 63, 35121 Padova (Italy)}
\email{sgiovann@math.unipd.it}
\author{Andrea Pasquali}
\address{Andrea Pasquali: Dept.~of Mathematics, Uppsala University, P.O.~Box 480, 751 06 Uppsala, Sweden}
\email{andrea.pasquali@math.uu.se}
\thanks{The first author is supported by the PhD programme in Mathematical Sciences, Department of Mathematics, University of Padova, and by grants BIRD163492 and DOR1690814 of University of Padova. The second author is supported by the PhD programme in Mathematics, Uppsala University and by the Swedish Research Council.}

\maketitle

\begin{abstract}
For a quiver with potential $(Q,W)$ with an action of a finite cyclic group $G$, we study the skew group algebra $\Lambda G$ of the Jacobian algebra $\Lambda = \mathcal P(Q, W)$. 
By a result of Reiten and Riedtmann, the quiver $Q_G$ of a basic algebra $\eta( \Lambda G) \eta$ Morita equivalent to $\Lambda G$ is known.
Under some assumptions on the action of $G$, 
we explicitly construct a potential $W_G$ on $Q_G$ such that $\eta(\Lambda G) \eta\cong \mathcal P(Q_G , W_G)$. The original quiver with potential can then be recovered by 
the skew group algebra construction with a natural action of the dual group of $G$. If $\Lambda$ is self-injective, then $\Lambda G$ is as well, and we investigate this case.
Motivated by Herschend and Iyama's characterisation of 2-representation finite algebras, we study how cuts on $(Q,W)$ behave with respect to our construction.
\end{abstract}

\tableofcontents

\section{Introduction}
The aim of this article is to study the skew group algebra of a Jacobian algebra coming from a quiver with potential. 
Skew group algebras were first studied from the point of view of representation theory in \cite{RR85}.
If $\Lambda$ is a finite-dimensional algebra over a field $\kk$ with
a finite group $G$ acting on $\Lambda$ by automorphisms, then the skew group algebra $\Lambda G$ shares many representation-theoretic properties with $\Lambda$, 
often incarnated in properties of functors between $\on{mod} \Lambda$ and $\on{mod}\Lambda G$. 
If $\Lambda$ is the quotient of a path algebra by an admissible ideal, then Reiten and Riedtmann describe the quiver $Q_G$ of (a basic version of) $\Lambda G$. 
This description is complete if $G$ is cyclic, and Demonet extended it to a complete description for arbitrary finite groups if $\Lambda$ is hereditary \cite{Dem10}.
However, describing the relations on this quiver is difficult in general. 

Something can be said for Jacobian algebras of quivers with potential (QPs).
These were introduced in \cite{DWZ08}, and have since found applications in cluster theory via the Amiot cluster category \cite{Ami09}. 
%To a QP $(Q, W)$ one can associate a Ginzburg dg algebra \cite{Gin06}, 
%and extend the action of $G$ to it. 
An action of $G$ on a QP $(Q, W)$ induces an action on the corresponding Ginzburg dg algebra defined in \cite{Gin06}.
In \cite{LM18}, it is shown that the skew group dg algebra is Morita equivalent to the Ginzburg dg algebra associated to another QP. The quiver is $Q_G$, and the potential is the image of $W$ under a natural map.
%Moreover, the potential is expressed as a linear combination of cycles of $Q_G$ in examples \cite[\S4.5]{LM18}.
Moreover, in \cite[\S4.5]{LM18} the potential is expressed as a linear combination of cycles of $Q_G$ in some examples.

In this article we eschew the dg setting and focus on Jacobian algebras of QPs.
The Jacobian algebra $\mathcal P(Q, W)$ is the 0-th homology of the Ginzburg dg algebra, and can be identified with the algebra obtained by imposing on $Q$ the relations coming 
from all cyclic derivatives of $W$. In particular, Le Meur's result implies that $\mathcal P(Q, W)G$ is Morita equivalent to the Jacobian algebra of a QP.
Under some assumptions on the action, we explicitly construct a potential $W_G$ on $Q_G$ such that we have:
\begin{thm2}[Theorem \ref{thm:main}]
  Let $(Q, W)$ be a QP, and let $\Lambda = \mathcal P(Q, W)$. Let $G$ be a finite cyclic group acting on $(Q, W)$ as per the assumptions 
  \ref{ass:field}-\ref{ass:cycles} of \S\ref{subsec:assumptions}. Let $Q_G$ be the quiver constructed in 
  \S\ref{sec:quiver}, $W_G$ the potential on $Q_G$ defined in \S\ref{subsec:cycles}, and $\eta\in \Lambda G$ the idempotent defined in \S\ref{subsec:assumptions}. Then
  \begin{align*}
    \mathcal P(Q_G, W_G) \cong \eta(\mathcal P(Q, W)G)\eta.
  \end{align*}
\end{thm2}
As observed in \cite[\S5]{RR85}, there is a natural action of the dual group $\hat G$ on $\Lambda G$. In our setting, this action restricts to the basic algebra $\eta( \Lambda G) \eta$. 
Reiten and Riedtmann prove that $(\Lambda G)\hat G$ is Morita equivalent to $\Lambda$ if $G$ is abelian, so it is natural to ask whether one gets back the original QP by applying this second skew group algebra construction.
To do so, one needs to find assumptions which guarantee that $W_G$ is fixed by $\hat G$ as an element of $\kk Q_G\cong \eta((\kk Q)G)\eta$, and which are preserved under taking skew group algebras. 
If $G= \Z/2\Z$, it was shown in \cite{AP17} that indeed we get $(Q,W)$ back (and in fact the Ginzburg dg algebra of $(Q,W)$).
We extend Amiot and Plamondon's result to our setting (assumptions \ref{ass:field}-\ref{ass:cycles} of \S\ref{subsec:assumptions}), via a direct check using our formula for $W_G$:
\begin{thm2}[Proposition~\ref{prop:main2b} and Corollary~\ref{cor:main2}]
  There is an isomorphism of quivers $\phi :(Q_G)_{\hat G}\cong Q$ such that, if we extend it to an isomorphism between the corresponding path algebras, we have $\phi( (W_G)_{\hat G}) = W$. This induces an algebra isomorphism
  \begin{align*}
    \theta \left( \left(\eta\left( \Lambda G \right)\eta\right) \hat G\right)\theta \cong \Lambda,
  \end{align*}
  where $\Lambda = \mathcal P(Q, W)$ and $\theta$ is the idempotent defined in Section~\ref{sec:dual}.
\end{thm2}

A simple example of the above construction which is good to have in mind is illustrated in Example~\ref{ex:triangles}, and specifically in the quivers of Figure~\ref{fig:triangle} and Figure~\ref{fig:sga_triangle}. 
Here we take $Q$ to be the QP of the \mbox{3-preprojective} algebra of type ${\rm A}_4$, so the potential is given by the sum of all 3-cycles with alternating signs. The group $G= \Z/3\Z$ acts by rotations in the plane, 
and the quiver $Q_G$ is given in Figure~\ref{fig:sga_triangle}. Here the action of $\hat G$ permutes the vertices $4^i$ and multiplies the arrow $\tilde \delta$ by a third root of unity, and one can check that 
by performing the same construction on $Q_G$ one gets $Q$ back.

We look with special interest at the case where $\Lambda = \mathcal P(Q,W)$ is self-injective.
This is because of the relationship between self-injective QPs and higher (in this case 2-) Auslander-Reiten theory. 
Namely, self-injective Jacobian algebras are precisely the 3-preprojective algebras of 2-representation finite algebras (see \cite{HI11b}).
Any such 2-representation finite algebra can be constructed from $(Q, W)$ together with the combinatorial datum of a so-called cut, as a truncated Jacobian algebra.
In \cite{RR85} it is proved that the skew group algebra construction preserves self-injectivity. We show that it also preserves the property of being Frobenius, and compute a Nakayama automorphism of $\Lambda G$
if the bilinear form on $\Lambda$ is $G$-equivariant. As a consequence, we prove that if $\Lambda$ is the Jacobian algebra of a planar self-injective QP and we take $G$ generated by a Nakayama automorphism,
then $\Lambda  G$ is symmetric. 
We show that \mbox{$G$-invariant} cuts on $(Q,W)$ induce cuts on $(Q_G, W_G)$, and the corresponding truncated Jacobian algebras are obtained from each other by a skew group algebra construction. Thus we have that, under some hypotheses, $2$-representation finiteness is preserved under taking skew group algebras (note that an analogous result was obtained, using different methods, for some $d$-representation infinite algebras in \cite{Gio19}).
Moreover we give some sufficient conditions on $(Q, W)$ which imply that all the truncated Jacobian algebras of $(Q_G, W_G)$ are derived equivalent.
It was recently shown in \cite{LM18}, by different methods, that in fact the property of being $d$-representation (in)finite is always preserved under taking skew group algebras.
An example where the 2-representation finite algebra is constructed from tensor product of Dynkin quivers is illustrated in Example~\ref{ex:tensor}.
We also look at a case where $\Lambda$ is not self-injective in Example~\ref{ex:nonselfinj}. Here we realise an Auslander algebra as a truncated Jacobian algebra, thus checking directly a special case of 
\cite[Theorem 1.3(c)(iv)]{RR85}.

There is a natural class of QPs with a group action satisfying our assumptions, namely rotation-invariant planar QPs.
Planar QPs were introduced in \cite{HI11b} as they behave particularly nicely when they have self-injective Jacobian algebras. It turns out that in all known examples of self-injective planar QPs 
a Nakayama automorphism acts by a rotation, hence they fit nicely in our setting.
Recently it has been shown that Postnikov diagrams have connections with planar self-injective QPs: in \cite{Pas17b} it is proved that the QP coming from an $(a,n)$-Postnikov diagram on a disk (as in 
\cite{BKM16}) is self-injective
if and only if the diagram is rotation invariant. Thus, our construction produces many examples of symmetric Jacobian algebras, one for every such Postnikov diagram.
An example is given as Example~\ref{ex:postnikov}.

The structure of this article is as follows. In Section~\ref{sec:prelim} we recall definitions and some facts about quivers with potential and skew group algebras. Moreover, we prove that skew group algebras 
of Frobenius algebras are again Frobenius. In Section~\ref{sec:setup} we explain our setup and assumptions, fix notation and present our main result. 
Section~\ref{sec:core} is devoted to proving Theorem~\ref{thm:main}.
In Section~\ref{sec:dual} we look at the $\hat G$-action on $\mathcal P(Q_G, W_G)$, and prove that we get back to $(Q, W)$ by taking 
the skew group algebra with respect to this action. 
In Section~\ref{sec:nakayama} we apply our results to planar rotation-invariant QPs. 
In Section~\ref{sec:cuts} we consider how cuts behave with respect to taking skew group algebras, and the consequences for truncated Jacobian algebras.
Section~\ref{sec:ex} consists of some examples which illustrate our construction. 

\paragraph{\textbf{Acknowledgements.}}
%The authors would like to express their gratitude to Pierre-Guy Plamondon and Martin Herschend for their help and encouragement.
%They would also like to thank Patrick Le Meur for pointing them to the reference \cite{LM18}.
We would like to express our gratitude to Pierre-Guy Plamondon and Martin Herschend for their help and encouragement. We thank an anonymous referee for their feedback on the paper.
We also thank Patrick Le Meur for pointing out the reference \cite{LM18}.

\section{Preliminaries}\label{sec:prelim}

\subsection{Conventions}
We denote by $\kk$ a fixed field. 
Algebras are assumed to be associative unital finite dimensional $\kk$-algebras. We denote by $D  = \on{Hom}_{\kk}(-,\kk)$ the $\kk$-dual, in both directions.
Quivers are understood to be finite and connected.
For a quiver $Q$, we denote by $Q_0$ its set of vertices and by $Q_1$ its set of arrows.
We compose quiver arrows from right to left, as functions. For an arrow $\alpha$, we denote by $\mathfrak{s}(\alpha)$ and $\mathfrak{t}(\alpha)$ its start and target respectively.
We compose quiver arrows from right to left, as functions. For an arrow $\alpha$, we denote by $\mathfrak{s}(\alpha)$ and $\mathfrak{t}(\alpha)$ its start and target respectively.
If $p$ is a path in a quiver and $\alpha$ is an arrow, we use the notation $\alpha\in p$ to indicate that $\alpha$ appears as one of the arrows in $p$.
A relation of a quiver is a linear combination of paths with the same start and end.

Let $\Lambda$ be an algebra and $\varphi\colon \Lambda\to\Lambda$ be an algebra endomorphism. For a right $\Lambda$-module $M$, we define $M_\varphi$ to be the right $\Lambda$-module which is equal to $M$ as a vector space but whose action is given by $m\cdot \lambda=m\varphi(\lambda)$, for all $m\in M$ and $\lambda\in\Lambda$.
For a finite group $G$, we denote by $\kk G$ the corresponding group algebra.
If $X$ is a subset of a ring $A$, we denote by $\langle X\rangle$ the two-sided ideal of $A$ generated by $X$.

\subsection{Index of terminology}
Since the statements, constructions and proofs in this article are quite technical and notation-heavy, we collect here the main terminology we use. The definitions given here are not meant to be complete, but we refer to the position in the text where they are explained properly.\\

\begin{tabular}{|c|p{10cm}|l|}
	\hline
	Symbol & Description & Reference\\ \hline\hline
	$a(c)$ & The coefficient of the cycle $c$ in $W$. & Notation~\ref{not:potential}\\ \hline
	
	$*$ & The ``forgetful" action of $G$ on $Q$. & Notation~\ref{not:star}\\ \hline
	
	$b(\alpha)$ & For an arrow $\alpha$ between fixed vertices, we define $b(\alpha)$ by $g(\alpha)= \zeta^{b(\alpha)}\alpha$. & Notation \ref{not:b}\\ \hline

	Types (i)--(iv) & The types of cycles appearing in $W$, as per  assumption~\ref{ass:cycles}. & Notation~\ref{not:cycletypes}\\ \hline

	$e_\mu$ & A choice of idempotents of $kG$. & Notation~\ref{not:emu}\\ \hline
	
	$\mathcal E, \mathcal{E}', \mathcal E''$ & Chosen subsets of $Q_0$. & Notation~\ref{not:E}\\ \hline 
	
	$\eta$ & The Morita idempotent we choose for $\Lambda G$. & Notation~\ref{not:eta}\\ \hline
	
	Types $(1)$--$(4)$ & Types of arrows. Every arrow in $Q$ is in the $G$-orbit of an arrow of one of these types. & Notation~\ref{not:arrowtypes}\\ \hline
	
	$\tilde \alpha, \tilde \alpha^\mu$ & The arrows of $Q_G$ we define. & Notation~\ref{not:arrowtypes}\\ \hline
	
	$t(\alpha)$ & For $\alpha$ of type (1), $\mathfrak s(\alpha)\in g^{t(\alpha)}(\mathcal E)$. Otherwise, $t(\alpha)=0$. & Notation~\ref{not:t}\\ \hline 
	
	$\hat c$ & A chosen cycle in the $G$-orbit of a cycle $c$. & Notation~\ref{not:cmu}\\ \hline
	
	$\tilde c, \tilde c^\mu$ & Cycles in $Q_G$ we define. & Notation~\ref{not:cmu}\\ \hline
	
	$p(c), q(c)$ & Integers associated to cycles of type (ii) and (iii) respectively. & Notation~\ref{not:cmu}\\ \hline
	
	$W_G$ & The potential we define on $Q_G$. & Notation~\ref{not:wg}\\ \hline
	
\end{tabular}

\subsection{Skew group algebras}
Let $G$ be a finite group acting on an algebra $\Lambda$ by automorphisms.
\begin{defin}
  The \emph{skew group algebra} $\Lambda G$ is the algebra defined by:
  \begin{itemize}
    \item its underlying vector space is $\Lambda\otimes_\kk \kk G$;
    \item multiplication is given by 
      \begin{align*}
	(\lambda\otimes g)(\mu\otimes h) = \lambda g(\mu)\otimes gh
      \end{align*}
      for $\lambda, \mu\in \Lambda$ and $g, h\in G$, extended by linearity and distributivity.
  \end{itemize}
\end{defin}
There is a natural algebra monomorphism $\Lambda\to \Lambda G $ given by $\lambda\mapsto \lambda\otimes 1$. Notice that the algebra $\Lambda G$ is not basic in general.

\subsection{Quivers with potential}
We follow \cite{HI11b} in our presentation.
Let $Q$ be a quiver. Denote by $\widehat{\kk Q}$ the completion of $\kk Q$ with respect to the $\langle Q_1\rangle$-adic topology.
Define $$\on{com}_Q = \overline {\left[ \widehat{\kk Q}, \widehat{\kk Q}\right]}\subseteq \widehat{\kk Q},$$
where $\overline{\phantom{xx}}$ denotes closure. Thus $\widehat{\kk Q}/ \on{com}_Q$ has a topological basis consisting of cycles in $Q$.
In particular there is a unique continuous linear map 
\begin{align*}
  \sigma: \widehat{\kk Q}/ \on{com}_Q\to \widehat{\kk Q}
\end{align*}
induced by 
\begin{align*}
  \alpha_1\cdots \alpha_n \mapsto \sum_{m = 1}^n \alpha_m \cdots \alpha_n\alpha_1\cdots \alpha_{m-1}.
\end{align*}
For each $\alpha\in Q_1$ define $d_\alpha:\langle Q_1\rangle\to\widehat{\kk Q} $ to be the continuous linear map given by $d_{\alpha} (\alpha p ) = p$ and $d_{\alpha} (q) = 0$ if $q$ does not end with $\alpha$.
Define the \emph{cyclic derivative} with respect to an arrow $\alpha$ to be $\partial_\alpha = d_\alpha \circ \sigma: \langle Q_1\rangle/ \on{com}_Q\to\widehat{\kk Q} $.
It will be convenient to take derivatives with respect to multiples of arrows. For $\lambda \in \kk^*$, define $\partial_{\lambda\alpha}(c) = \lambda^{-1}\partial_\alpha(c)$. 
A \emph{potential} is an element $W\in \langle Q_1\rangle^3/(\langle Q_1\rangle^3\cap \on{com}_Q)$, i.e.,~a (possibly infinite) linear combination of cycles of length at least 3.
A potential is called \emph{finite} if it can be written as a finite linear combination of cycles.
By an abuse of notation, if $c$ is a cycle in $Q$ we will denote again by $c$ the corresponding element of $\langle Q_1\rangle^3/(\langle Q_1\rangle^3\cap \on{com}_Q)$ and consider it up to cyclic permutation of its arrows.
We call the pair $(Q, W)$ a \emph{quiver with potential (QP)} and define its Jacobian algebra to be 
\begin{align*}
  \cal P(Q, W) = \widehat{\kk Q}\bigg/\overline{ \langle \partial_\alpha W\ | \ \alpha\in Q_1\rangle}.
\end{align*}
In our setting, the completion will not play any role, due the following proposition.
\begin{prop}[{\cite[Proposition 2.3]{Pas17b}}]\label{prop:completion}
  If $W$ is a finite potential and the ideal $\langle \partial_\alpha W\ | \ \alpha\in Q_1\rangle\subseteq\kk Q $ is admissible, then 
  \begin{align*}
    \mathcal P(Q, W) \cong \kk Q\bigg/ \langle \partial_\alpha W\ | \ \alpha\in Q_1\rangle.
  \end{align*}
\end{prop}

\subsection{Self-injective algebras}
We need some facts and notation for Frobenius and self-injective algebras, see for instance \cite{Mur03} or \cite{HZ11}.
An algebra $\Lambda$ is \emph{self-injective} if it is injective as a right $\Lambda$\mbox{-m}odule. 
It is \emph{Frobenius} if there is a bilinear form $(-,-)$ on $\Lambda$ which is nondegenerate and multiplicative (i.e.,~
$(a, bc\mbox{) = (}ab,c)$ for all $a,b,c\in \Lambda$). It is \emph{symmetric} if this form can be taken to be symmetric.
Frobenius algebras are self-injective, and the converse is true if and only if $\on{dim}\on{Hom}_\Lambda(S, \Lambda) = \on{dim}S$ for all simples $S$.
In particular, self-injective basic algebras are exactly the Frobenius basic algebras.

If $\Lambda$ is Frobenius, then from the nondegenerate bilinear form we get an isomorphism $f:\Lambda\to D\Lambda$ of vector spaces, given by $f(v) = (-,v)$.
Moreover $f$ is an isomorphism of left $\Lambda$-modules since 
\begin{align*}
  f(\lambda v) = (-, \lambda v) = (-\lambda, v) = \lambda(f(v)).
\end{align*}
Nondegeneracy of the form implies that there exists a unique $\kk$-linear map $\varphi:\Lambda\to \Lambda$ satisfying
\begin{align*}
  (a,b) = (b,\varphi(a))
\end{align*}
for all $a, b\in\Lambda$. In fact such a $\varphi$ is an algebra automorphism, and $f$ becomes a right module isomorphism $f:\Lambda_{\varphi}\to D\Lambda$.
If we choose a different bilinear form and hence a different isomorphism $g:\Lambda\to D\Lambda$ of vector spaces, then $g(a) = f(au)$ for some unit $u\in \Lambda$.
Then the corresponding automorphism $\psi$ is given by $\psi(a) = u \varphi(a)u^{-1}$, so $\varphi$ is unique as an outer automorphism of $\Lambda$.
The automorphism $\varphi$ is called a Nakayama automorphism of $\Lambda$.
In particular, $\Lambda$ is symmetric if and only if $\varphi = \on{id}_{\on{Out}(\Lambda)}$.

We are interested in studying skew group algebras of Frobenius algebras, and in particular the case where $G$ is generated by a Nakayama automorphism.
\begin{rmk}
  In \cite[Theorem 1.3(c)(iii)]{RR85} it is proved that skew group algebras of self-injective algebras are always self-injective. In the discussion that follows we show that the property of being Frobenius is also preserved under taking skew group algebras.
\end{rmk}

Let $G$ be a finite group acting on a Frobenius algebra $\Lambda$ by automorphisms. 
The algebra $\kk G$ is always Frobenius and in fact symmetric. We denote by $(-,-)$ the corresponding symmetric nondegenerate bilinear form on $\kk G$ as well.
This form can be taken to be $(h, l) = \delta_{hl^{-1}}$ for $h,l\in G$, extended bilinearly.
Then we can define a bilinear form $\langle-,-\rangle$ on the skew group algebra $\Lambda G$ by setting
\begin{align*}
  \langle \lambda\otimes l,\mu\otimes m\rangle = (\lambda, l(\mu))(l,m)
\end{align*}
for $\lambda, \mu\in \Lambda$ and $l,m\in G$, extended bilinearly.

\begin{lemma}
The form $\langle-,-\rangle$ is multiplicative and nondegenerate. In particular, $\Lambda G$ is Frobenius.
\end{lemma}

\begin{proof}
We have
\begin{align*}
  \langle (\lambda\otimes l)(\mu\otimes m), \nu\otimes n\rangle &= \langle \lambda l(\mu)\otimes lm, \nu\otimes n\rangle = \\
  &= (\lambda l(\mu) , (lm)(\nu))(lm,n) = \\
  &= (\lambda, l(\mu)l(m(\nu)))(l,mn) = \\
  &= (\lambda, l(\mu m(\nu)))(l,mn) = \\
  &= \langle \lambda\otimes l, \mu m(\nu)\otimes mn\rangle =\\
  &= \langle \lambda\otimes l, (\mu\otimes m)(\nu\otimes n)\rangle
\end{align*}
for all $\lambda,\mu,\nu\in \Lambda$ and $l,m,n\in G$. This proves multiplicativity.

Assume now that there exists $\sum_{i} \xi_i\otimes z_i\in \Lambda G$ such that $\langle \sum_{i} \xi_i\otimes z_i, x\rangle = 0$ for all $x\in \Lambda G$. Without loss of generality we can take every $z_i$ to be an element of $G$.
Take $x=\lambda\otimes l$ with $\lambda\in \Lambda$ and $l\in G$. 
Then 
\begin{align*}
  0 = \sum_i \langle \xi_i\otimes z_i, \lambda\otimes l\rangle &= \sum_i (\xi_i,z_i(\lambda))(z_i,l)   = \sum_{z_i = l^{-1}}(\xi_i,l^{-1}(\lambda)) = \left(\sum_{z_i = l^{-1}}\xi_i,l^{-1}(\lambda)\right).
\end{align*}
Since $l$ acts by an automorphism and $(-,-)$ is nondegenerate, it follows that $\sum_{z_i = l^{-1}}\xi_i = 0$. By iterating this argument for all possible values of $l$, we get that
$$\sum_{i} \xi_i\otimes z_i = \sum_{l\in G} \left(\sum_{z_i=l^{-1}} \xi_i\right)\otimes l^{-1} = 0.$$

Assume instead that $\langle x,\sum_i\xi_i\otimes z_i\rangle = 0$ for all $x\in \Lambda G$. Again we suppose that $z_i\in G$ and we take $x=\lambda\otimes l$ with $\lambda\in \Lambda$ and $l\in G$.
Then
\begin{align*}
  0 = \sum_i \langle \lambda\otimes l,\xi_i\otimes z_i\rangle &= \sum_i(\lambda,l(\xi_i))(l,z_i) =    \sum_{z_i = l^{-1}}(\lambda, l(\xi_i))  = \left( \lambda,l\left( \sum_{z_i = l^{-1}}\xi_i \right) \right)
\end{align*}
so that $\sum_{z_i= l^{-1}}\xi_i = 0$ and we can argue as above. This proves nondegeneracy.
\end{proof}

If the bilinear form on $\Lambda$ is $G$-equivariant, we can find a Nakayama automorphism of $\Lambda G$. Let us choose a Nakayama automorphism $\varphi$ of $\Lambda$.

\begin{prop}
  \label{prop:nakayama}
  If $(g(\lambda), g(\mu)) = (\lambda,\mu)$ for all $g\in G, \lambda, \mu \in \Lambda$, then $\varphi\otimes 1$ is a Nakayama automorphism of $\Lambda G$.
\end{prop}

\begin{proof}
  Let $\lambda,\mu\in \Lambda$ and $l,m\in G$. Then 
  \begin{align*}
    \langle \lambda\otimes l,\mu\otimes m\rangle &= \delta_{lm^{-1}}(\lambda, l(\mu)) = \\
    &= \delta_{lm^{-1}}(l(\mu), \varphi(\lambda)) = \\
    &= \delta_{lm^{-1}} (\mu, l^{-1}\varphi(\lambda)) = \\
    &= \delta_{lm^{-1}}(\mu,m\varphi(\lambda)) = \\
    &= \langle \mu\otimes m, \varphi(\lambda)\otimes l\rangle. \qedhere
  \end{align*}
\end{proof}

\begin{cor}
  \label{cor:sym}
  If $\varphi$ generates the image $\on{im}(G)\subseteq \on{Aut}(\Lambda)$, then $\Lambda G$ is symmetric.
\end{cor}

\begin{proof} 
Since $\varphi$ is an element in $\on{im}(G)$, we know that there is an $h\in G$ which acts on $\Lambda$ as $\varphi$.
Now let $g\in G$. By assumption, there exists an integer $j$ such that $g$ acts on $\Lambda$ as $\varphi^j$.
Then we have 
$$(\lambda, \mu) = (\mu, \varphi(\lambda)) = (\varphi(\lambda), \varphi(\mu)) = (\varphi^j(\lambda), \varphi^j(\mu)) = (g(\lambda), g(\mu)),$$
so we can apply Proposition~\ref{prop:nakayama} and get that
$$\varphi\otimes 1 : \lambda\otimes l\mapsto h(\lambda)\otimes l$$ is a Nakayama automorphism of $\Lambda G$. Notice now that $h(\lambda)\otimes l = (1\otimes h)(\lambda\otimes l)(1\otimes h)^{-1}$, so that $\varphi\otimes 1$ is the identity as an outer automorphism of $\Lambda G$, which means that $\Lambda G$ is symmetric.
\end{proof}

We include the following lemma, which we will use in Section~\ref{sec:nakayama}.

\begin{lemma}
  \label{lem:basicsym}
  Let $\Lambda$ be a symmetric algebra, and $e\in  \Lambda$ an idempotent. Then $e\Lambda e$ is symmetric.
\end{lemma}

\begin{proof}
  Let $\langle -,-\rangle$ be a symmetric multiplicative nondegenerate bilinear form on $\Lambda$. Then the restricted form on $e\Lambda e$ is a symmetric multiplicative bilinear form on $e\Lambda e$. Let now $u\in e\Lambda e$ such that $\langle u,-\rangle _{|e\Lambda e} = 0$.
  Let $v\in \Lambda$ and observe that 
\begin{align*}
  \langle u,v\rangle &= \langle eue, v\rangle =  \langle eu, ev\rangle = \langle ev, eu\rangle  =\langle eve, u\rangle = 0
\end{align*}
so that $u = 0$ since the form is nondegenerate on $\Lambda$.
\end{proof}

\section{Setup and result}\label{sec:setup}
In this section we set up our assumptions, and fix the notation we need to be able to state our results.

Let $(Q,W)$ be a quiver with potential and let $\Lambda$ be its Jacobian algebra. 
\begin{nota}\label{not:potential}
	We write $W = \sum_c a(c)c$.
\end{nota}
Recall that we consider cycles up to cyclic permutation. We assume 
that $W$ is finite and that the cyclic derivatives of $W$ generate an admissible ideal of $\kk Q$.
In what follows we will freely use integers as indices for convenience, even when they should be seen as elements of $\Z/n\Z$.
\subsection{Assumptions}\label{subsec:assumptions}
Let $G$ be a cyclic group of order $n$ with generator $g$, acting on $\kk Q$.
We make the following assumptions \ref{ass:field}--\ref{ass:cycles}.
\begin{enumerate}[label=(A\arabic*), series=assumptions]
  \item The field $\kk$ contains a primitive $n$-th root of unity $\zeta$. In particular, $n\neq 0$ in $\kk$. \label{ass:field}
  \item The action of $G$ permutes the vertices of $Q$ and maps every arrow to a multiple of an arrow. \label{ass:permuted}
  \item Every vertex of $Q$ which is not fixed by $G$ has an orbit of cardinality $n$. \label{ass:cardinality}
  \item We have $GW= W$ in $\widehat{kQ}/\on{com}_Q$. \label{ass:potential}
\end{enumerate}
Since $G$ preserves the potential, we get an induced action of $G$ on $\Lambda$.
\begin{nota}\label{not:star}
	We define a second ``forgetful'' action $*$ of $G$ on $Q$ by $g*v =g(v)$ for $v\in Q_0$ and $g*\alpha = \beta$ whenever $\beta$ is an arrow and $g(\alpha)$ is a scalar multiple of $\beta$.
\end{nota}

\begin{rmk}\label{rmk:fixed}
  Let $u,v\in Q_0$ be (not necessarily distinct) vertices fixed by $G$. The vector space $V$ spanned by arrows from $u$ to $v$ is a $kG$-module, and since $G$ is abelian it decomposes into 1-dimensional submodules. This means that, up to choosing a different basis for $V$, we can assume that arrows between fixed vertices are mapped to scalar multiples of themselves.
\end{rmk}
By this observation, we can without loss of generality make the additional assumption:
\begin{enumerate}[label=(A\arabic*),resume=assumptions]
	\item If $\alpha$ is an arrow between two fixed vertices, then $g(\alpha)=\zeta^{b(\alpha)}\alpha$ for an integer $b(\alpha)$. \label{ass:fixed}
\end{enumerate}
\begin{nota}
	\label{not:b}
	We define $b(\alpha)$ as above, for $\alpha$ any arrow between fixed vertices.
\end{nota}
\begin{rmk}
  Suppose that an arrow $\alpha$ is such that $g(\alpha) = \zeta^i\beta$ for some arrow $\beta\neq \alpha$.
  Then, by assumptions \ref{ass:fixed} and \ref{ass:cardinality}, one of $\mathfrak s(\alpha)$ and $\mathfrak t(\alpha)$ has an orbit of size $n$, so $|G*\alpha| = n$.
  We can replace $\beta$ with $\zeta^{-i}\beta$ as the element in $\on{rad}\Lambda/\on{rad}^2\Lambda$ representing the corresponding arrow. By doing this for all $n$ distinct arrows in the
  orbit of $\alpha$, we get that on this orbit the action of $G$ coincides with the $*$ action of $G$. 
  The potential $W$ is not affected by this procedure, if we see it as an element of $\widehat{\kk Q}/\on{com}_Q$, so it is still invariant under $G$. However, 
  note that the expression of $W$ as a linear combination of cycles in $Q$ is possibly changed.
\end{rmk}

In view of the above observation, we can without loss of generality make the additional assumption:
\begin{enumerate}[label=(A\arabic*),resume=assumptions]
  \item Arrows with at least one end which is not fixed are sent to arrows by the action of $G$. \label{ass:arrows}
\end{enumerate}
So for an arrow $\alpha$ between two fixed vertices we have $g*\alpha = \alpha = \zeta^{-b(\alpha)}g(\alpha)$, while for all other arrows we have $g*\alpha = g(\alpha) = \beta$ for some arrow $\beta\neq \alpha$.

We need to make some further assumptions about the relationship between $G$ and $W$. It turns out that it is convenient to impose conditions on the number of fixed vertices appearing in cycles of $W$.
We make the following assumption.
\begin{enumerate}[label=(A\arabic*),resume=assumptions]
  \item Every cycle $c$ appearing in $W$ is of one of the following types:\label{ass:cycles}
\end{enumerate}
\begin{enumerate}[label=(\roman*)]
  \item the cycle $c$ goes through no vertices fixed by $G$;
  \item the cycle $c$ goes through exactly one (counted with multiplicity) vertex fixed by $G$;
  \item the cycle $c$ goes through exactly one (counted with multiplicity) vertex not fixed by $G$;
  \item the cycle $c$ goes only through vertices which are fixed by $G$.
\end{enumerate}

\begin{nota}\label{not:cycletypes}
	We call cycles appearing in $W$ cycles of type (i)--(iv) according to the (mutually exclusive) cases of assumption~\ref{ass:cycles}.
\end{nota}

\begin{rmk}
  These assumptions are strong. We need them to construct a QP $(Q_G, W_G)$ such that the skew group algebra of $\mathcal P(Q,W)$ is Morita equivalent to $\mathcal P(Q_G, W_G)$. However, the assumptions are satisfied in many examples,
  and they are weak enough to still hold for $(Q_G, W_G)$. This in turn allows us to come back to $(Q, W)$ via a skew group algebra construction with a natural action of the dual group $\hat G$ (see Section~\ref{sec:dual}).
\end{rmk}

\begin{rmk}
  From our assumptions, it follows that cycles of a given type are mapped by $G$ to multiples of cycles of the same type. By assumption \ref{ass:arrows}, cycles of type (i) and (ii) contain only arrows that are mapped to 
arrows, so those cycles are mapped to cycles. If $c = \alpha_1\dots\alpha_l$ is of type (iv), then $g(c) = \zeta^{\sum_{i}b(\alpha_i)} c$, so from $GW = W$ we obtain that $\sum_{i}b(\alpha_i) = 0 \pmod n$ and
$Gc = c$.
In particular, $g*c = g(c)$ for all $c$ of type (i), (ii), (iv).
\end{rmk}

The reader wishing to have examples of QPs with group actions satisfying these assumptions is advised to have in mind the QPs of Example~\ref{ex:triangles}. 
In particular, the two QPs of Figure \ref{fig:triangle} and Figure \ref{fig:sga_triangle} both have an action of $\Z/3\Z$, one sending arrows to arrows and the other multiplying $\tilde \delta$ by a third root of 
unity. 
They are the quivers with potential corresponding to each other's skew group algebra under these actions.
All cycles of the first one are of type (i) or (ii), while all cycles of the second one are of type (iii) or (iv).

\subsection{The quiver of $\Lambda G$}\label{sec:quiver}
We now describe the quiver $Q_G$ of the skew group algebra $\Lambda G$ following \cite{RR85}.
We first define an idempotent $\eta\in \Lambda G$ such that $\eta( \Lambda G) \eta$ is basic and Morita equivalent to $\Lambda G$. We decompose $\eta$ as a sum of primitive orthogonal idempotents, and
use those to label the vertices of $Q_G$. Then we choose elements in $\eta(\Lambda G)\eta$ to be the arrows. 

\begin{nota}\label{not:emu}
	A complete list of primitive orthogonal idempotents for the group algebra $\kk G$ is given by
$$e_\mu = \frac{1}{n} \sum_{i=0}^{n-1} \zeta^{i\mu}g^i,$$
for $\mu=0,\dots,n-1$.
\end{nota}

\begin{nota}\label{not:E}
Fix a set $\cal E$ of representatives of vertices of $Q$ under the action of $G$. We write $\cal E = \cal E' \sqcup \cal E''$,
where $\cal E'$ and $\cal E''$ consist of the vertices in $\mathcal E$ whose orbits have cardinality $n$ and $1$ respectively. 
\end{nota}

\begin{nota}\label{not:eta}
We define the following idempotents in $\Lambda G$:
\begin{itemize}
	\item for each vertex $\varepsilon\in\cal E'$ we put $\eta^\varepsilon = \varepsilon \otimes 1$;
	\item for each vertex $\varepsilon\in\cal E''$ and $\mu=0,\dots,n-1$ we put $\eta^\varepsilon_\mu = \varepsilon \otimes e_\mu$.
\end{itemize}
Set
$$\eta = \sum_{\varepsilon\in\cal E'} \eta^\varepsilon + \sum_{\varepsilon\in\cal E''}\sum_{\mu=0}^{n-1} \eta^\varepsilon_\mu.$$
Note in particular that $\eta = \e\otimes 1$, where $\e$ is the idempotent of $\Lambda$ corresponding to $\mathcal E$.
By \cite[\S2.3]{RR85} the algebra $\eta(\Lambda G)\eta$ is Morita equivalent to $\Lambda G$. 
A complete list of primitive orthogonal idempotents for $\eta(\Lambda G)\eta$ is 
$$\left\{ \eta^{\varepsilon} \ |\ \varepsilon\in \mathcal E' \right\}\cup \left\{ \eta^\varepsilon_\mu \ |\ \varepsilon\in \mathcal E'', \mu= 0, \dots, n-1\right\}.$$
\end{nota}

\begin{rmk}
  The idempotent $\eta$ is not canonical, in that it depends on choosing some vertices of $Q$. However, it is convenient to define it in this way to get a natural action of the dual group $\hat G$ on $\eta( \Lambda G) \eta$.
  By contrast, the authors of \cite{AP17} choose a canonically defined basic algebra for their Morita equivalence, but in exchange they have to choose vertices of $Q$ in order to be able to define such an action.
\end{rmk}

\begin{nota}\label{not:arrowtypes}
Now we will fix a basis for the arrows of the quiver $Q_G$ of $\eta(\Lambda G)\eta$. 
There are four different cases to consider.
\begin{enumerate}[label=(\arabic*)]
  \item Let $\beta$ be an arrow between two non-fixed vertices of $Q$. Then there is exactly one arrow $\alpha$ in the $G$-orbit of $\beta$ such that $\mathfrak t(\alpha)\in \mathcal E'$.
    Thus $\alpha$ is of the form $\alpha\colon g^t\varepsilon \to \varepsilon'$, with $\varepsilon, \varepsilon'\in \mathcal E'$ and $0\leq t \leq n-1$. We call $\alpha$ an arrow of type (1), and define an element $\tilde \alpha\in \eta( \Lambda  G)\eta$
    by $$\tilde\alpha=\alpha \otimes g^t.$$ This will be an arrow in $Q_G$ from $\eta^\varepsilon$ to $\eta^{\varepsilon'}$.
  \item  Let $\beta$ be an arrow in $Q$ from a non-fixed vertex to a fixed vertex. Then there is exactly one arrow $\alpha$ in the $G$-orbit of $\beta$ such that $\mathfrak s(\alpha) \in \mathcal E'$.
	  Thus $\alpha$ is of the form $\alpha\colon \varepsilon \to \varepsilon'$, with $\varepsilon\in\cal E'$, $\varepsilon'\in\cal E''$.
	  We call $\alpha$ an arrow of type (2), and define elements $\tilde \alpha^\mu\in \eta( \Lambda G)\eta$ by $$\tilde\alpha^\mu=(1 \otimes e_\mu)(\alpha \otimes 1)$$ for $\mu = 0,\dots, n-1$.
	These will be arrows in $Q_G$ from $\eta^\varepsilon$ to $\eta^{\varepsilon'}_\mu$ respectively.
  \item  Let $\beta$ be an arrow in $Q$ from a fixed vertex to a non-fixed vertex. Then there is exactly one arrow $\alpha$ in the $G$-orbit of $\beta$ such that $\mathfrak t(\alpha) \in \mathcal E'$.
	  Thus $\alpha$ is of the form $\alpha\colon \varepsilon \to \varepsilon'$, with $\varepsilon\in\cal E''$, $\varepsilon'\in\cal E'$.
	  We call $\alpha$ an arrow of type (3), and define elements $\tilde \alpha^\mu\in \eta( \Lambda G)\eta$ by $$\tilde\alpha^\mu=\alpha \otimes e_\mu$$ for $\mu = 0,\dots, n-1$.
	These will be arrows in $Q_G$ from $\eta^\varepsilon_\mu$ to $\eta^{\varepsilon'}$ respectively.
      \item Let $\alpha$ be an arrow between two fixed vertices, i.e.,~$\alpha\colon \varepsilon \to \varepsilon'$ with $\varepsilon,\varepsilon'\in\cal E''$. Recall that by assumption $g(\alpha) = \zeta^{b(\alpha)}\alpha$.
	We call $\alpha$ an arrow of type (4), and define elements $\tilde \alpha^\mu\in \eta( \Lambda G)\eta$ by $$\tilde\alpha^\mu=\alpha \otimes e_\mu$$ for $\mu = 0, \dots, n-1$.
	These will be arrows in $Q_G$ from $\eta^\varepsilon_\mu$ to $\eta^{\varepsilon'}_{\mu-b(\alpha)}$ respectively.
\end{enumerate}
\end{nota} 
\begin{rmk}
	By our construction, not every arrow of $Q$ is of type $(1), (2), (3)$ or $(4)$. However, for each arrow $\beta$ of $Q$ there exists a unique arrow $\alpha$ of one of those types which is in the $G$-orbit of $\beta$.
\end{rmk}

\begin{nota}\label{not:t}
For an arrow $\alpha: g^t(\varepsilon)\to \varepsilon'$ of type (1), we define $t(\alpha) = t$. Note that this integer is well defined modulo $n$, since the orbit of $\varepsilon$ has cardinality $n$.
If instead $\alpha$ is an arrow of type (2), (3), or (4), we put $t(\alpha) = 0$.
\end{nota}

\begin{prop}
  This choice gives a basis of $\on{rad}\eta(\Lambda G)\eta/\on{rad}^2\eta(\Lambda G)\eta$, and the start and target of arrows in $Q_G$ are as claimed above.
\end{prop}

\begin{proof}
  The vector space spanned by the arrows of $Q$ decomposes as a direct sum of $\kk G$-modules into the spans of the $G$-orbits of the arrows. Therefore it is enough to look at one $G$-orbit of an arrow at a time,
  and we can assume that there are no multiple arrows in $Q$. 

  Let us now look at the four cases. If $\alpha: g^t\varepsilon\to \varepsilon'$ is of type (1), then the $n$ arrows in $G\alpha$ give rise to a unique arrow $\tilde \alpha:\eta^\varepsilon\to \eta^{\varepsilon'}$.
  By \cite[Theorem 1.3(d)(i)]{RR85} we have that $\on{rad}^i\Lambda G = (\on{rad}^i\Lambda)\Lambda G$, so that a basis of the space of arrows from $\eta^\varepsilon$ to $\eta^{\varepsilon'}$ is given by 
  $\left\{ \varepsilon'\beta h(\varepsilon)\otimes h \right\}$ with $\beta\in Q_1$. So the only $\beta$ contributing is the only arrow in $G\alpha$ ending in $\varepsilon'$, and 
 this basis is $\{\tilde \alpha = \alpha\otimes g^{t(\alpha)}\}$.

  Let now $\alpha:\varepsilon\to \varepsilon'$ be of type (2). Then the $n$ arrows in $G\alpha$ give rise to $n$ arrows of $Q_G$. By the above argument, we get that a basis of 
  $(\varepsilon'\otimes 1)(\on{rad}\Lambda G/\on{rad}^2\Lambda G)(\varepsilon\otimes 1)$ is given by $\left\{ g^i(\alpha)\otimes g^i\ |\ i = 0, \dots, n-1 \right\}$.
  Then the set $\left\{ \tilde \alpha^\mu = (1\otimes e_{\mu})(\alpha\otimes 1)\ |\ \mu = 0, \dots, n-1 \right\}$ is also a basis, since
  $$(1\otimes e_\mu)(\alpha\otimes 1)  =\frac{1}{n}\sum_{i= 0}^{n-1}\zeta^{i\mu}g^i(\alpha)\otimes g^i.$$
  Now $\eta^{\varepsilon'}_\nu \tilde \alpha^\mu= \tilde \alpha^\mu$ if $\nu = \mu$, and 0 otherwise, so each $\tilde \alpha^\mu$ is indeed an arrow of $Q_G$ from $\eta^\varepsilon$ to $\eta^{\varepsilon'}_\mu$.
  
  If $\alpha:\varepsilon \to \varepsilon'$ is an arrow of type (3) or (4), by similar arguments we get that $\left\{ \alpha\otimes g^i \right\}$ is a basis of
$(\varepsilon'\otimes 1)(\on{rad}\Lambda G/\on{rad}^2\Lambda G)(\varepsilon\otimes 1)$.
  Then $\left\{ \tilde \alpha^{\mu} = \alpha\otimes e_\mu \right\}$ is also a basis, and it consists of arrows.
\end{proof}

The choice of vertices and arrows we have made defines an isomorphism $J: \kk Q_G\to \eta( (\kk Q)G) \eta$ by \cite[\S2.3]{RR85}.

\subsection{Cycles in $Q_G$ and the potential $W_{G}$} \label{subsec:cycles}
We want to define a potential $W_G$ on $Q_G$, so we need to construct cycles in $Q_G$ depending on those appearing in $W$. Recall that we write $W = \sum_c a(c)c$, and
that we consider cycles up to cyclic permutation.
\begin{nota}\label{not:cmu}
We will define, for every cycle $c$ appearing in $W$, a cycle $\hat c$ in $G*c$ depending on our choice of representatives of the vertices.
Moreover, to every $\hat c$ we will associate a cycle $\tilde c$ in $Q_G$.
\begin{enumerate}[label=(\roman*)]
  \item Let $c$ be a cycle of type (i) in $W$. Then choose $\hat c$ in $G*c$ such that
    $$
    \xymatrix{
      \hat c  &:& \varepsilon_0 = \varepsilon_l \ar[rr]^-{g^{t_1+\dots+t_{l-1}}(\alpha_l)} & &
      g^{t_1+\dots+t_{l-1}}(\varepsilon_{l-1})\ar[r] & \cdots \ar[r]^-{g^{t_1}(\alpha_2)} & g^{t_1}(\varepsilon_1)
      \ar[r]^-{\alpha_1} &\varepsilon_0
    }
    $$
    with $\varepsilon_i\in \mathcal E'$ for all $i$. Notice that this is indeed (in general) a choice, the only requirement is that $\hat c$ should go through at least one vertex in $\mathcal E'$.
  Set moreover $\hat d = \hat c$ for all the other $d\in G*c$.
    Note that each $\alpha_i$ is an arrow of type (1) and $t_i=t(\alpha_i)$.
    Define a cycle $\tilde c$ in $Q_G$ by $$\tilde c = \tilde{\alpha}_1\cdots \tilde{\alpha}_l.$$
  \item Let $c$ be a cycle of type (ii) in $W$. There is a unique $\hat c\in G*c$ that can be written as above, with $\varepsilon_i\in \mathcal E'$ for $i\neq1$, $\varepsilon_1\in \mathcal E''$ and $t_1 = 0$.
    Note that for $i\geq3$, $\alpha_i$ is of type (1) and $t_i=t(\alpha_i)$, while $g^{-t_2}(\alpha_2)$ is of type (2) and $\alpha_1$ is of type (3).
    Define cycles $\tilde c^\mu$ in $Q_G$ by $$\tilde c^\mu = \tilde{\alpha}_1^\mu \widetilde{g^{-t_2}(\alpha_2)}^\mu\tilde{\alpha}_3\cdots \tilde{\alpha}_l$$ for 
    $\mu = 0,\dots,n-1$, and call $p(c) = t_2$.
  \item Let $c$ be a cycle of type (iii) in $W$. There is a unique $\hat c\in G*c$ that can be written as above, with $\varepsilon_i\in \mathcal E''$ for $i\neq1$, $\varepsilon_1\in \mathcal E'$ and $t_i = 0$ for all $i$.
    Notice that for $i\geq 3$, $\alpha_i$ is of type (4), while $\alpha_2$ is of type (3) and $\alpha_1$ is of type (2).
    Put $b_i=b(\alpha_i)+\dots+b(\alpha_l)$ for all $i\geq 3$ and define cycles $\tilde c^\mu$ in $Q_G$ by 
    $$\tilde c^\mu =
    \tilde{\alpha}_1^\mu \tilde{\alpha}_2^{\mu-b_3} \tilde{\alpha}_3^{\mu-b_4} \cdots \tilde{\alpha}_{l-1}^{\mu-b_l} \tilde{\alpha}_l^\mu $$
    for $\mu = 0, \dots, n-1$. Call $q(c)=b_3$, and notice that $g(\hat c) = \zeta^{q(c)}g*\hat c$ (and in fact $g(c) = \zeta^{q(c)}g*c$).
  \item Let $c$ be a cycle of type (iv) in $W$. Thus $Gc = c$ in $\kk Q$ and we can write $\hat c = c$ as above, with $\varepsilon_i\in \mathcal E'$ and $t_i = 0$ for all $i$. 
    Notice that each $\alpha_i$ is an arrow of type (4).
    Put $b_i=b(\alpha_i)+\dots+b(\alpha_l)$ for all $i$ and define cycles $\tilde c^\mu$ in $Q_G$ by
    $$\tilde c^\mu = \tilde{\alpha}_1^{\mu-b_2} \tilde{\alpha}_2^{\mu-b_3} \cdots \tilde{\alpha}_{l-1}^{\mu-b_l} \tilde{\alpha}_l^\mu $$
    for $\mu = 0,\dots,n-1$.
\end{enumerate}
\end{nota}

Now define $\mathcal C(x) = \left\{ \hat c \ | \ c \text{ cycle of $W$ of type } x \right\}$ for $x =$~(i), (ii), (iii), (iv). Then $\mathcal C = \bigsqcup \mathcal C(x)$ is a cross-section of 
cycles of $W$ under the $*$ action of $G$. 

\begin{nota}\label{not:wg}
	We can now define a (finite) potential $W_G$ on $Q_G$ by setting
	\begin{align*}
	W_G &=\sum_{c\in \mathcal C({\rm i})} a(c)\frac{|Gc|}{n}\tilde c + \sum_{c\in \mathcal C({\rm ii})}a(c) \sum_{\mu = 0}^{n-1}\zeta^{-p(c)\mu}\tilde c^{\mu} 
	+ \sum_{c\in \mathcal C({\rm iii})\cup \mathcal C({\rm iv})}a(c)\sum_{\mu = 0}^{n-1}\tilde c^{\mu}.
	\end{align*}
\end{nota}

\begin{rmk}
Note that all cycles in $W_G$ have length at least $3$, since each of them has the same length as a cycle in $W$. Moreover the sums in $W_G$ are made over subsets of cycles which appear in $W$, hence they are all finite. This means that $W_G$ is indeed a finite potential in $Q_G$.
\end{rmk}

\subsection{Main result}\label{sec:main}
We are ready to state our main result. Recall that we assume that $(Q, W)$ is a QP with finite potential such that the cyclic derivatives of $W$ generate an admissible ideal of $\kk Q$. Call $\Lambda = \mathcal P(Q, W)$ 
the Jacobian algebra of $(Q, W)$.

\begin{thm}
  \label{thm:main} 
  Let $G$ be a finite cyclic group acting on $(Q, W)$ as per the assumptions \ref{ass:field}-\ref{ass:cycles}.
  Then
  \begin{align*}
    \mathcal P(Q_G, W_G) \cong \eta(\Lambda G)\eta.
  \end{align*}
\end{thm}

We give a proof of this result in \S\ref{sec:proof}, and outline here the strategy we will use.
By \cite[\S2.3]{RR85}, the algebra $\eta(\Lambda G)\eta$ is isomorphic to $\kk Q_G$ modulo a certain ideal. Our first step, carried out in \S\ref{sec:gen}, is to give explicit generators for this ideal in our setting.
However, these generators will not be relations of $Q_G$ (i.e.,~linear combinations of paths in $Q_G$ with common start and end). 
In \S\ref{sec:comp}, we express them in terms of the derivatives of the potential $W_G$, which will allow us to conclude.

\begin{rmk}
  The statement that there exists a potential $W'$ such that $\mathcal P(Q_G, W') \cong \eta(\Lambda G)\eta$ follows, by taking the $0$-th cohomology of the corresponding dg algebras, from a much more general result proved in \cite[Corollary~1.3]{LM18}. Moreover,
  \cite[Lemma~4.4.1]{LM18} expresses a suitable $W'$ as an element of $\eta(\Lambda G)\eta$, and $W'$ is written as a linear combination of paths in $Q_G$ in the examples of \cite[\S4.5]{LM18}. 
  Our Theorem \ref{thm:main} states that the potential $W_G$, which we constructed under our assumptions \ref{ass:field}-\ref{ass:cycles}, has the same property.
\end{rmk}

\section{Proof of main result}\label{sec:core}
\subsection{Ideals of skew group algebras}\label{sec:gen}
In order to prove Theorem \ref{thm:main}, we need some observations about ideals of skew group algebras.

\begin{prop}\label{prop:gen}
Let $A$ be a ring and let $\eta$ be an idempotent of $A$. Let $I = AXA$ for some subset $X\subseteq A$, such that $\eta x\eta = x$ for all $x\in X$.
Then 
\begin{align*} 
  \eta \frac{A}{I}\eta = \frac{\eta A\eta}{\langle X\rangle}.
\end{align*}
\end{prop}

\begin{proof}
  It is enough to prove that $\eta I \eta = \langle X\rangle$. Let $\kappa = 1-\eta$.
  Then $\eta A = \eta A \eta \oplus \eta A \kappa$ and $A\eta = \eta A\eta \oplus \kappa A\eta$. Observe that $\eta x\eta = x$ implies $\kappa x = x\kappa  = 0$.
  Then 
  \begin{equation*}
    \eta I \eta = \eta A X A\eta = \eta A \eta X \eta A\eta \oplus \eta A \eta X \kappa A\eta\oplus \eta A \kappa X \eta A\eta \oplus \eta A \kappa X \kappa A\eta  = \eta A \eta X \eta A\eta = \langle X\rangle. \qedhere
  \end{equation*}
\end{proof}

Now retain the notation of Section~\ref{sec:setup}. So $\Lambda = \kk Q/\mathcal R$, where $\mathcal R=\langle R\rangle$ and $R=\{\partial_\alpha W \,|\, \alpha\in Q_1\}$, and the action of $G$ on $\Lambda$ leaves $R$ stable.
Then we know by \cite[\S2.2]{RR85} that 
\begin{align*}
  \Lambda G\cong \frac{(\kk Q)G}{\langle R\otimes 1\rangle}.
\end{align*}
Recall that we have an idempotent $\eta=\e\otimes1$, for an idempotent $\e$ in $\kk Q$, such that $\eta((\kk Q)G)\eta\cong\kk Q_G$. We have the following lemmas.

\begin{lemma}
  \label{lem:admissible}
  Suppose that $\langle R\rangle$ is an admissible ideal of $\kk Q$. Then the ideal $\eta\langle R\otimes 1\rangle\eta$ of $\eta((\kk Q)G)\eta$ is admissible.
\end{lemma}

\begin{proof}
  Let $A=\kk Q$. Since $\mathcal R = \langle R\rangle$ is admissible, we have $(\on{rad}A)^N\subseteq \mathcal R\subseteq (\on{rad}A)^2$ for some $N\geq 2$.
  Consider $\mathcal R$ as a subset of $A G$ under the natural inclusion $A\to A G$, so $\langle R\otimes1 \rangle = (A G)\mathcal R(A G)$. By \cite[Theorem 1.3(d)(ii)]{RR85} we have $(A G)(\on{rad} A)^i=(\on{rad} A)^i(A G)=(\on{rad} A G)^i$ for all $i\geq1$, so
  $$(A G)(\on{rad} A)^N(A G) \subseteq (A G)\mathcal R(A G) \subseteq (A G)(\on{rad} A)^2(A G)$$
becomes
$$(\on{rad} A G)^N \subseteq \langle R\otimes1 \rangle \subseteq (\on{rad} A G)^2.$$
Then the claim follows from the fact that $\eta(\on{rad} A G)\eta=\on{rad}(\eta (A G)\eta)$.
\end{proof}

\begin{lemma}
  \label{lem:gener}
  For each $r\in R$, choose $g_r, h_r\in G$ such that $\mathfrak{t}(r)\in g_r(\mathcal E)$ and $\mathfrak{s}(r)\in h_r(\mathcal E)$.
  Then
		\begin{align*}
		  \eta \frac{(\kk Q)G}{\langle R\otimes 1\rangle}\eta = \frac{\eta ((\kk Q)G)\eta}{\langle g_r^{-1}(r)\otimes h_rg_r^{-1}\ |\ r\in R\rangle}.
	\end{align*}
\end{lemma}

\begin{proof}
 We have
  \begin{align*}
     g_r^{-1}(r)\otimes h_rg_r^{-1}  = (1\otimes g_r^{-1})(r\otimes 1)(1\otimes h_r)
  \end{align*}
  so that $R\otimes 1$ generates the same ideal in $(\kk Q)G$ as the set $\left\{g_r^{-1}(r)\otimes h_rg_r^{-1} \ |\ r\in R \right\}$.
  Now 
  \begin{align*}
    \eta(g_r^{-1}(r)\otimes h_rg_r^{-1})\eta = \e g_r^{-1}(r) (h_r g_r^{-1})(\e)\otimes h_rg_r^{-1} = g_r^{-1}(r)\otimes h_rg_r^{-1},
  \end{align*}
so the claim follows from Proposition \ref{prop:gen}.
\end{proof}

\begin{lemma}
  \label{lem:key}
  In the assumptions \ref{ass:field}-\ref{ass:cycles}, we have 
  \begin{align*}
     \eta(\Lambda G)\eta \cong\frac{\eta ( (\kk Q)G)\eta}{\langle \partial_{g^{-t(\alpha)}\alpha} W\otimes g^{-t(\alpha)} \ |\ \alpha \text{ 
    of type } (1), (2), (3), (4)\rangle}.
  \end{align*}
\end{lemma}

\begin{proof}
  Since $G$ acts on $W$, the ideal of $\kk Q$ generated by $\left\{ \partial_\alpha W \right\}\otimes 1$ is also generated by 
  $$\left\{ \partial_\alpha W\ |\ \alpha \text{ 
    of type } (1), (2), (3), (4) \right\}\otimes 1,$$
    since $h(\partial_\alpha W) = \partial_{h(\alpha)}W$ for any $h\in G$.
    Notice that $\alpha$ is of type (1), (2), (3), (4) precisely if \mbox{$\mathfrak{s}(\partial_\alpha W) = \mathfrak{t}(\alpha) \in \mathcal E$}, and then $\mathfrak{t}(\partial_\alpha W) = \mathfrak{s}(\alpha) \in g^{t(\alpha)}(\mathcal E)$.
  Then we can apply Lemma \ref{lem:gener} with $g_r = g^{t(\alpha)}$ and $h_r = 1$, and we get the claim.
\end{proof}

\subsection{Derivatives of $W_G$ as elements of $\eta (\Lambda G) \eta$}\label{sec:comp}
In this section we shall express elements of the form \mbox{$\partial_{g^{-t(\alpha)}\alpha} W\otimes g^{-t(\alpha)}$} for $\alpha$ of type (1), (2), (3), (4) in terms of the derivatives of the potential $W_G$. 
Precisely, identifying $\eta( (\kk Q)G )\eta$ with $\kk Q_G$ via the isomorphism $J $ of \S\ref{sec:quiver}, each $\partial_{g^{-t(\alpha)}\alpha} W\otimes g^{-t(\alpha)}$ corresponds to 
$\sum_{i,j\in (Q_G)_0}x_{ij}$, where $x_{ij}$ is a linear combination of paths in $Q_G$ from vertex $i$ to vertex $j$ (i.e.,~a relation of $Q_G$). 
%We advise the reader to compare Lemma \ref{lem:1234} with the computations carried out in \cite[\S4.5]{LM18}.
In Lemma \ref{lem:1234} we describe the elements $x_{ij}$ in terms of the derivatives of $W_G$, in a way that depends on the type of $\alpha$.
This will be the last ingredient we need in order to prove Theorem \ref{thm:main}.
We advise the reader to compare Lemma \ref{lem:1234} with the computations carried out in \cite[\S4.5]{LM18}.

In the proof of Lemma \ref{lem:1234}, we will use the following identities.
\begin{lemma}
  \label{lem:e_mu}
  If $\alpha\in Q_1$ and $\beta$ is an arrow of type (4), then
  \begin{align*}
    (\alpha\otimes e_\mu)(\beta\otimes e_\nu) = 
    \begin{cases}
      \alpha\beta\otimes e_\nu,  \text{\ \ if } \nu  = \mu + b(\beta); \\
      0, \text{ \ \ otherwise}.
    \end{cases}
  \end{align*}
\end{lemma}

\begin{proof}
  We compute
  \begin{align*}
    (\alpha\otimes e_\mu)(\beta\otimes e_\nu) &= 
    \frac{1}{n}\sum_{i = 0}^{n-1}\zeta^{i\mu}(\alpha\otimes g^i  )(\beta\otimes e_\nu) =\\
    &= \frac{1}{n}\sum_{i = 0}^{n-1}\zeta^{i\mu} \alpha g^i(\beta) \otimes g^i e_\nu =\\
    &=  \frac{1}{n}\sum_{i = 0}^{n-1}\zeta^{i(\mu+b(\beta))} \alpha \beta \otimes g^i e_\nu =\\
    &=  \alpha \beta \otimes\frac{1}{n}\sum_{i = 0}^{n-1} \zeta^{i(\mu+b(\beta))} g^i e_\nu = \\
    &= \alpha\beta\otimes e_{\mu+b(\beta)}e_\nu
  \end{align*}
  and this proves the claim.
\end{proof}

\begin{lemma}
  If $c$ is a cycle of type (iii), then $a(g*c) = \zeta^{q(c)}a(c)$.
\end{lemma}

\begin{proof}
  From assumption \ref{ass:potential}, it follows that $g(a(c)c) = a(g*c)g*c$. Then we get the claim since $g(c) = \zeta^{q(c)}g*c$.
\end{proof}

Now we use the identification $\kk Q_G\cong \eta((\kk Q)G)\eta$ to see cyclic derivatives of $W_G$ as elements of $\eta((\kk Q)G)\eta$.
To avoid clogging the notation, we will at times write $h \alpha$ and $h c$ instead of $h(\alpha) $ and $h(c)$ for $h\in G$.

\begin{lemma}
  \begin{enumerate}
  \label{lem:1234}
\item Let $\alpha$ be an arrow of $Q$ of type (1). Let $\beta = g^{-t(\alpha)}(\alpha)$. Then 
  \begin{align*}
    \partial_\beta W\otimes g^{-t(\alpha)} =  \partial_{\tilde \alpha}W_G.
  \end{align*}
\item  Let $\alpha$ be an arrow of $Q$ of type (2). Then
  \begin{align*}
    \partial_\alpha W\otimes 1 =\sum_{\mu = 0}^{n-1} \partial_{\tilde \alpha^\mu}W_G.
  \end{align*}
  In particular, 
  \begin{align*}
    \eta^{\mathfrak{s}(\alpha)} (\partial_\alpha W\otimes 1) \eta^{\mathfrak{t}(\alpha)}_\mu   =\partial_{\tilde \alpha^\mu}W_G
 \end{align*}
 for every $\mu = 0, \dots, n-1$.

\item Let $\alpha$ be an arrow of $Q$ of type (3). Then 
  \begin{align*}
    \partial_\alpha W\otimes 1 =\sum_{\mu = 0}^{n-1} \partial_{\tilde \alpha^\mu}W_G.
  \end{align*}
  In particular, 
  \begin{align*}
    \eta^{\mathfrak{s}(\alpha)}_{\mu} (\partial_\alpha W\otimes 1) \eta^{\mathfrak{t}(\alpha)}  = \partial_{\tilde \alpha^\mu}W_G
 \end{align*}
 for every $\mu = 0, \dots, n-1$.

\item  Let $\alpha$ be an arrow of  type (4). Then 
\begin{align*}
  \partial_\alpha W\otimes 1 = n \sum_{\mu = 0}^{n-1} \partial_{\tilde \alpha^\mu}W_G.
\end{align*}
In particular, 
  \begin{align*}
    \eta^{\mathfrak{s}(\alpha)}_{\mu} (\partial_\alpha W\otimes 1) \eta^{\mathfrak{t}(\alpha)}_{\mu-b(\alpha)}  =  n \partial_{\tilde \alpha^\mu}W_G
 \end{align*}
 for every $\mu = 0, \dots, n-1$.
 \end{enumerate}

\end{lemma}

\begin{proof}
First notice that the second part of statements (2), (3), (4) follows directly by multiplying $\sum\partial_{\tilde \alpha^\mu}W_G$, which is a linear combination of paths in $Q_G$, with idempotents corresponding to vertices of $Q_G$. 

  It will be convenient to use the following notation: for integers $t_1,\dots, t_l$, write
  \begin{align*}
    t_{i,j} = 
    \begin{cases}
      t_i + t_{i+1} + \cdots +t_j, \text{  if  } j\geq i; \\
      t_i + t_{i+1} + \cdots +t_{l}+ t_1 + t_2 + \cdots +t_j, \text{  if  } j<i.
    \end{cases}
  \end{align*}

  \begin{enumerate}
     \item We have that
  $$\partial_\beta W\otimes g^{-t(\alpha)} = \sum_{c \text{ of type }({\rm i})} a(c) \partial_\beta c \otimes g^{-t(\alpha)} + \sum_{c \text{ of type }({\rm ii})} a(c) \partial_\beta c \otimes g^{-t(\alpha)}$$
  and
  $$\partial_{\tilde\alpha} W_G = \frac{|Gc|}{n}\sum_{c\in \mathcal C({\rm i})} a(c) \partial_{\tilde\alpha}\tilde c + \sum_{c \in \mathcal C({\rm ii})} a(c) \sum_{\mu = 0}^{n-1} \zeta^{-p(c)\mu} \partial_{\tilde\alpha} \tilde c^{\mu}.$$
  The statement will be proved using the following two claims:
  
  \textbf{Claim (a1).} If $c\in \mathcal C({\rm i})$, then  
      $$\sum_{r=0}^{n-1} \partial_\beta g^rc \otimes g^{-t(\alpha)} = \partial_{\tilde\alpha}\tilde c.$$
      
    \textbf{Claim (b1).} If $c\in \mathcal C({\rm ii})$, then  
$$\sum_{r=0}^{n-1} \partial_\beta g^rc \otimes g^{-t(\alpha)} = \sum_{\mu = 0}^{n-1}\zeta^{-p(c)\mu}\partial_{\tilde\alpha^\mu}\tilde c^\mu .$$

  Assuming these claims hold, let us prove the statement. 
  Recall that by assumption \ref{ass:arrows}, $gc = g*c$ if $c$ is of type (i) or (ii).
  We have
  \begin{align*}
     \sum_{c \text{ of type }({\rm i})} a(c) \partial_\beta c \otimes g^{-t(\alpha)} &= 
     \sum_{c\in \mathcal C({\rm i})}\sum_{r = 0}^{|Gc|-1}a(g^r* c)\partial_\beta(g^r*c)\otimes g^{-t(\alpha)} = \\
     &= \sum_{c\in \mathcal C({\rm i})}\frac{|Gc|}{n}\frac{n}{|Gc|}\sum_{r = 0}^{|Gc|-1}a(c)\partial_\beta g^rc\otimes g^{-t(\alpha)} = \\
     &= \sum_{c\in \mathcal C({\rm i})}\frac{|Gc|}{n}\sum_{r = 0}^{n-1}a(c)\partial_\beta g^rc\otimes g^{-t(\alpha)}  = \\
     &= \frac{|Gc|}{n}\sum_{c\in \mathcal C({\rm i})}a(c)\partial_{\tilde\alpha}\tilde c
  \end{align*}
  and
  \begin{align*}
    \sum_{c \text{ of type }({\rm ii})} a(c) \partial_\beta c \otimes g^{-t(\alpha)} &= 
     \sum_{c\in \mathcal C({\rm ii})}\sum_{r = 0}^{|Gc|-1}a(g^r* c)\partial_\beta(g^r*c)\otimes g^{-t(\alpha)} = \\
     &= \sum_{c\in \mathcal C({\rm ii})}\sum_{r = 0}^{n-1}a(c)\partial_\beta g^rc\otimes g^{-t(\alpha)} = \\
     &= \sum_{c\in \mathcal C({\rm ii})}a(c)\sum_{\mu = 0}^{n-1}\zeta^{-p(c)\mu}\partial_{\tilde\alpha^\mu}\tilde c^\mu 
  \end{align*}
  which together imply that
  \begin{align*}
    \partial_\beta W\otimes g^{-t(\alpha)} =  \partial_{\tilde \alpha}W_G.
  \end{align*}
 It remains to prove the claims (a1) and (b1).
  
  \textbf{Proof of (a1).} Since $c\in \mathcal C({\rm i})$ we can write
  $$
    \xymatrix{
      c  &:& \varepsilon_0 = \varepsilon_l \ar[rr]^-{g^{t_{1,l-1}}(\alpha_l)} & &
      g^{t_{1,l-1}}(\varepsilon_{l-1})\ar[r] & \cdots \ar[r] & g^{t_1}(\varepsilon_1)
      \ar[r]^{\alpha_1} &\varepsilon_0.
    }
    $$
  Let $M = \{m\in\{1,\dots,l\} \,|\, \alpha=\alpha_m\}$. Then
  \begin{align*}
    \partial_{\tilde\alpha}\tilde c &= \partial_{\tilde\alpha} \tilde\alpha_1 \cdots \tilde\alpha_l = \sum_{m\in M} \tilde\alpha_{m+1} \cdots \tilde\alpha_{m-1} =\\
    &= \sum_{m\in M} (\alpha_{m+1}\otimes g^{t_{m+1}})\cdots (\alpha_{m-1}\otimes g^{t_{m-1}}) = \\
    &= \sum_{m\in M} \alpha_{m+1} g^{t_{m+1}}(\alpha_{m+2}) \cdots g^{t_{m+1,m-2}}(\alpha_{m-1}) \otimes g^{-t_m}.
  \end{align*}
  Note that $t_m=t(\alpha)$ for all $m\in M$, so we are left to prove that
  $$\sum_{m\in M} \alpha_{m+1} g^{t_{m+1}}(\alpha_{m+2}) \cdots g^{t_{m+1,m-2}}(\alpha_{m-1}) = \sum_{r=0}^{n-1} \partial_\beta g^rc.$$
  For each $r=0,\dots,n-1$ and $m\in M$, the path $g^rc$ contains the arrow $g^{r+t_{1,m-1}}\alpha_m = g^{r+t_{1,m}}\beta$.
  Hence, if we define $M_r = \{m\in M \,|\, r=-t_{1,m}\}$, we have that
  $$\partial_\beta g^rc = \sum_{m\in M_r} \alpha_{m+1} g^{t_{m+1}}(\alpha_{m+2}) \cdots g^{t_{m+1,m-2}}(\alpha_{m-1}).$$
  So the equality we wanted to show becomes
  $$\sum_{m\in M} \alpha_{m+1} g^{t_{m+1}}(\alpha_{m+2}) \cdots g^{t_{m+1,m-2}}(\alpha_{m-1}) = \sum_{r=0}^{n-1} \sum_{m\in M_r} \alpha_{m+1} g^{t_{m+1}}(\alpha_{m+2}) \cdots g^{t_{m+1,m-2}}(\alpha_{m-1}),$$
  but this holds because $M = \bigsqcup_{r=0}^{n-1} M_r$.
  
\textbf{Proof of (b1).} Since $c\in \mathcal C({\rm ii})$ we can write
  $$
   \xymatrix{
      c  &:& \varepsilon_0 = \varepsilon_l \ar[rr]^-{g^{t_{1,l-1}}(\alpha_l)} & &
      g^{t_{1,l-1}}(\varepsilon_{l-1})\ar[r] & \cdots \ar[r] & g^{t_2}(\varepsilon_2)\ar[r]^-{\alpha_2} &
      \varepsilon_1
      \ar[r]^{\alpha_1} &\varepsilon_0.
    }
  $$
  Recall that by definition $p(c)=t_2$.  
  Let $M = \{m\in\{1,\dots,l\} \,|\, \alpha=\alpha_m\}$.  Then
  \begin{align*}
    \partial_{\tilde\alpha}\tilde c^\mu &= \partial_{\tilde\alpha} \tilde\alpha_1^\mu \widetilde{g^{-p(c)}(\alpha_2)}^\mu\tilde\alpha_3\cdots \tilde\alpha_l  = \\
    &= \sum_{m\in M} \tilde\alpha_{m+1} \cdots \tilde\alpha_1^\mu \widetilde{g^{-p(c)}(\alpha_2)}^\mu\tilde\alpha_3 \cdots \tilde\alpha_{m-1} =\\
    &= \sum_{m\in M} (\alpha_{m+1} \otimes g^{t_{m+1}}) \cdots (\alpha_1 \otimes e_\mu) (g^{-t_2}(\alpha_2) \otimes 1) \cdots (\alpha_{m-1} \otimes g^{t_{m-1}}).
  \end{align*}
  Now, recalling that $\sum_{\mu=0}^{n-1} e_\mu = 1$ and $\zeta^{-t_2\mu}e_\mu = g^{t_2}e_\mu$, we get
  \begin{align*}
     \sum_{\mu = 0}^{n-1} \zeta^{-t_2\mu} \partial_{\tilde\alpha} \tilde c^{\mu}  
    &= \sum_{m\in M} \sum_{\mu = 0}^{n-1} (\alpha_{m+1} \otimes g^{t_{m+1}}) \cdots (\alpha_1 \otimes \zeta^{-t_2\mu} e_\mu) (g^{-t_2}(\alpha_2) \otimes 1) \cdots (\alpha_{m-1} \otimes g^{t_{m-1}})= \\
    &= \sum_{m\in M} (\alpha_{m+1} \otimes g^{t_{m+1}}) \cdots (\alpha_1 \otimes g^{t_2}) (g^{-t_2}(\alpha_2) \otimes 1) \cdots (\alpha_{m-1} \otimes g^{t_{m-1}}) =\\
    &= \sum_{m\in M} (\alpha_{m+1} \otimes g^{t_{m+1}}) \cdots (\alpha_1 \otimes g^{t_1}) (\alpha_2 \otimes g^{t_2}) \cdots (\alpha_{m-1} \otimes g^{t_{m-1}}) =\\
    &= \sum_{m\in M} \alpha_{m+1} g^{t_{m+1}}(\alpha_{m+2}) \cdots g^{t_{m+1,m-2}}(\alpha_{m-1}) \otimes g^{-t_m}.
  \end{align*}
  The rest of the proof of part (b1) is analogous to that of part (a1).

\item We have that
  $$\partial_\alpha W\otimes 1 = \sum_{c \text{ of type }({\rm ii})} a(c) \partial_\alpha c \otimes 1 + \sum_{c \text{ of type }({\rm iii})} a(c) \partial_\alpha c \otimes 1$$
  and
  \begin{align*}
  \sum_{\mu=0}^{n-1} \partial_{\tilde\alpha^\mu} W_G &= \sum_{\mu=0}^{n-1} \sum_{c \in \mathcal C({\rm ii})} a(c) \sum_{\nu = 0}^{n-1} \zeta^{-p(c)\nu} \partial_{\tilde\alpha^\mu}\tilde c^\nu +  \sum_{\mu=0}^{n-1} \sum_{c \in \mathcal C({\rm iii})} a(c) \sum_{\nu = 0}^{n-1}\partial_{\tilde\alpha^\mu} \tilde c^{\nu}.
  \end{align*}
  The statement will be proved using the following two claims:
  
   \textbf{ Claim (a2).} If $c\in \mathcal C({\rm ii})$, then $\partial_{\tilde\alpha^\mu} \tilde c^\nu = 0$ for $\mu \neq \nu$ and
$$\sum_{r=0}^{n-1} \partial_\alpha g^rc \otimes 1 = \sum_{\mu = 0}^{n-1}\zeta^{-p(c)\mu}\partial_{\tilde\alpha^\mu}\tilde c^\mu .$$

    \textbf{Claim (b2).} If $c\in \mathcal C({\rm iii})$, then $\partial_{\tilde\alpha^\mu} \tilde c^\nu = 0$ for $\mu \neq \nu$ and
$$ \partial_\alpha c \otimes 1 = \sum_{\mu = 0}^{n-1}\partial_{\tilde\alpha^\mu}\tilde c^\mu .$$
  
Assuming these claims hold, let us prove the statement. 
First notice that if $c\in \mathcal C({\rm iii})$ and $\alpha\in h*c$, then $h = 1$.
We have
\begin{align*}
  \sum_{c \text{ of type }({\rm ii})} a(c) \partial_\alpha c \otimes 1  &= \sum_{c\in \mathcal C({\rm ii})}\sum_{r = 0}^{|Gc|-1}a(g^r* c)\partial_\alpha(g^r*c)\otimes 1 = \\
  &=  \sum_{c \in \mathcal C({\rm ii})} 	\sum_{r = 0}^{n-1}a(c)\partial_\alpha g^rc\otimes 1 = \\
  &=  \sum_{c \in \mathcal C({\rm ii})} a(c)\sum_{\mu = 0}^{n-1}\zeta^{-p(c)\mu}\partial_{\tilde\alpha^\mu}\tilde c^\mu=\\
  &= \sum_{\mu = 0}^{n-1} \sum_{c \in \mathcal C({\rm ii})} a(c)\sum_{\nu = 0}^{n-1}\zeta^{-p(c)\nu}\partial_{\tilde\alpha^\mu}\tilde c^\nu
\end{align*}
and 
\begin{align*}
  \sum_{c \text{ of type }({\rm iii})} a(c) \partial_\alpha c \otimes 1 &= \sum_{c\in \mathcal C({\rm iii})}\sum_{r = 0}^{|Gc|-1}a(g^r* c)\partial_\alpha(g^r*c)\otimes 1 = \\
  &= \sum_{c \in \mathcal C({\rm iii})} a(c)\partial_\alpha c\otimes 1 = \\
  &= \sum_{c \in \mathcal C({\rm iii})} a(c)\sum_{\mu = 0}^{n-1}\partial_{\tilde\alpha^\mu}\tilde c^\mu = \\
  &= \sum_{\mu=0}^{n-1} \sum_{c \in \mathcal C({\rm iii})} a(c) \sum_{\nu = 0}^{n-1}\partial_{\tilde\alpha^\mu} \tilde c^{\nu}
\end{align*}
which together imply that
\begin{align*}
  \partial_\alpha W\otimes 1 = \sum_{\mu=0}^{n-1} \partial_{\tilde\alpha^\mu} W_G.
\end{align*}
It remains to prove the claims (a2) and (b2).

\textbf{Proof of (a2).}
      Since $c\in \mathcal C({\rm ii})$, we can write it as 
$$
    \xymatrix{
      c  &:& \varepsilon_0 = \varepsilon_l \ar[rr]^-{g^{t_{1,l-1}}(\alpha_l)} & &
      g^{t_{1,l-1}}(\varepsilon_{l-1})\ar[r] & \cdots \ar[r] & g^{t_2}(\varepsilon_2)\ar[r]^-{\alpha_2} &
      \varepsilon_1
      \ar[r]^{\alpha_1} &\varepsilon_0.
    }
    $$
    
    If $\alpha\not\in g^rc$ for all $r$ then the statement is trivially true. Otherwise, suppose $\alpha\in g^rc$ for some $r$. Then, since $\alpha$ is of type (2), we necessarily have that $r=-t_2$ and $\alpha = g^{-t_2}(\alpha_2)$ is the only copy of $\alpha$ in $g^{-t_2}c$.
    Hence $$\tilde c^\nu = \tilde{\alpha}_1^\nu \widetilde{g^{-t_2}(\alpha_2)}^\nu\tilde{\alpha}_3\cdots \tilde{\alpha}_l = \tilde{\alpha}_1^\nu \tilde\alpha^\nu\tilde{\alpha}_3\cdots \tilde{\alpha}_l$$ and $\partial_{\tilde\alpha^\mu} \tilde c^\nu = 0$ for $\mu \neq \nu$. We have
    \begin{align*}
      \partial_{\tilde \alpha^{\mu}}\tilde c^\mu &= 
	\partial_{\tilde \alpha^{\mu}} \tilde \alpha^\mu \tilde \alpha_3 \cdots \tilde \alpha_l \tilde \alpha_1^\mu = \\
		&= \tilde \alpha_3 \cdots \tilde \alpha_l \tilde \alpha_1^\mu = \\
		&= (\alpha_3\otimes g^{t_3})\cdots (\alpha_l\otimes g^{t_l})(\alpha_1\otimes e_{\mu})
    \end{align*}
    so that (recall that $t_{2,l} = 0 \pmod n$)
  \begin{align*}
    \sum_{\mu= 0}^{n-1}\zeta^{-t_2\mu} \partial_{\tilde \alpha^{\mu}}\tilde c^\mu &= \sum_{\mu= 0}^{n-1}(\alpha_3\otimes g^{t_3})\cdots (\alpha_l\otimes g^{t_l})(\alpha_1\otimes g^{t_2}e_{\mu}) = \\
    &= \sum_{\mu= 0}^{n-1}(\alpha_3 g^{t_3}(\alpha_4)\cdots g^{t_{3,l-1}}(\alpha_l)\otimes g^{-t_2})(\alpha_1\otimes e_{\mu})(1\otimes g^{t_2}) = \\
    &= \alpha_3 g^{t_3}(\alpha_4)\cdots g^{t_{3,l-1}}(\alpha_l)g^{t_{3,l}}(\alpha_1)\otimes 1 = \\
    &= \partial_\alpha g^{-t_2} c\otimes 1 = \\
    &= \sum_{r=0}^{n-1} \partial_\alpha g^rc \otimes 1,
  \end{align*}
  which proves the claim.

    \textbf{Proof of (b2).}
      We have, since $c\in \mathcal C({\rm iii})$,  
       $$
    \xymatrix{
       c  &:& \varepsilon_0 = \varepsilon_l \ar[r]^-{\alpha_l}  &
      \varepsilon_{l-1}\ar[r] & \cdots \ar[r] &\varepsilon_1
      \ar[r]^{\alpha_1} &\varepsilon_0,
    }
  $$
  with $\alpha = \alpha_1$, and observe that this is the only instance of $\alpha$ in $c$. Setting $b_i=b(\alpha_i)+\dots+b(\alpha_l)$ for $i\geq 3$, we have $\tilde c^\nu = \tilde{\alpha}^\nu \tilde{\alpha}_2^{\nu-b_3} \tilde{\alpha}_3^{\nu-b_4} \cdots \tilde{\alpha}_l^\nu$, so $\partial_{\tilde\alpha^\mu} \tilde c^\nu = 0$ for $\mu \neq \nu$.
	We can compute
	\begin{align*}
	  \partial_{\tilde \alpha^{\mu}}\tilde c^\mu &=  \partial_{\tilde \alpha^{\mu}}\tilde{\alpha}^\mu \tilde{\alpha}_2^{\mu-b_3} \tilde{\alpha}_3^{\mu-b_4} \cdots \tilde{\alpha}_l^\mu = \\
	  &= \tilde{\alpha}_2^{\mu-b_3} \tilde{\alpha}_3^{\mu-b_4} \cdots \tilde{\alpha}_l^\mu  = \\
	  &= (\alpha_2\otimes e_{\mu-b_3} )(\alpha_3\otimes e_{\mu-b_4} )\cdots(\alpha_l\otimes e_\mu ) = \\
	&= \alpha_2\alpha_3\cdots \alpha_l\otimes e_\mu = \\
	&= \partial_\alpha  c\otimes e_\mu
	\end{align*}
	so that 
	\begin{align*}
 \sum_{\mu= 0}^{n-1}\partial_{\tilde \alpha^{\mu}}\tilde c^\mu = \partial_\alpha c\otimes 1
 \end{align*}
 as claimed.

\item We have that
  $$\partial_\alpha W\otimes 1 = \sum_{c \text{ of type }({\rm ii})} a(c) \partial_\alpha c \otimes 1 + \sum_{c \text{ of type }({\rm iii})} a(c) \partial_\alpha c \otimes 1$$
  and
  \begin{align*}
  \sum_{\mu=0}^{n-1} \partial_{\tilde\alpha^\mu} W_G &=  \sum_{\mu=0}^{n-1} \sum_{c \in\mathcal C({\rm ii})} a(c) \sum_{\nu = 0}^{n-1} \zeta^{-p(c)\nu} \partial_{\tilde\alpha^\mu}\tilde c^\nu +\sum_{\mu=0}^{n-1} \sum_{c \in\mathcal C({\rm iii})} a(c) \sum_{\nu = 0}^{n-1}\partial_{\tilde\alpha^\mu} \tilde c^{\nu}.
  \end{align*}
  The statement will be proved using the following two claims:

    \textbf{Claim (a3).} If $c\in \mathcal C({\rm ii})$, then $\partial_{\tilde\alpha^\mu} \tilde c^\nu = 0$ for $\mu \neq \nu$ and
$$\partial_\alpha c \otimes 1 = \sum_{\mu = 0}^{n-1}\zeta^{-p(c)\mu}\partial_{\tilde\alpha^\mu}\tilde c^\mu .$$
    
    \textbf{Claim (b3).} If $c\in\mathcal C({\rm iii})$, then $\partial_{\tilde\alpha^\mu} \tilde c^\nu = 0$ for $\mu \neq \nu-q(c)$ and
$$\partial_\alpha c \otimes 1 = \sum_{\mu=0}^{n-1} \partial_{\tilde\alpha^\mu} \tilde c^{\mu+q(c)} .$$

Assuming these claims hold, let us prove the statement. 
First notice that if $c\in \mathcal C({\rm ii}) \cup \mathcal C({\rm iii})$ and $\alpha\in h*c$, then $h = 1$.
We have
\begin{align*}
  \sum_{c \text{ of type }({\rm ii})} a(c) \partial_\alpha c \otimes 1 &= \sum_{c\in \mathcal C({\rm ii})}\sum_{r = 0}^{|Gc|-1}a(g^r* c)\partial_\alpha(g^r*c)\otimes 1 = \\
    &= \sum_{c \in \mathcal C({\rm ii})} a(c)\partial_\alpha c\otimes 1 = \\
  &=  \sum_{c \in \mathcal C({\rm ii})} a(c)\sum_{\mu = 0}^{n-1}\zeta^{-p(c)\mu}\partial_{\tilde\alpha^\mu}\tilde c^\mu=\\
  &= \sum_{\mu = 0}^{n-1} \sum_{c \in \mathcal C({\rm ii})} a(c)\sum_{\nu = 0}^{n-1}\zeta^{-p(c)\nu}\partial_{\tilde\alpha^\mu}\tilde c^\nu
\end{align*}
and
\begin{align*}
   \sum_{c \text{ of type }({\rm iii})} a(c) \partial_\alpha c \otimes 1 &=\sum_{c\in \mathcal C({\rm iii})}\sum_{r = 0}^{|Gc|-1}a(g^r* c)\partial_\alpha(g^r*c)\otimes 1 = \\
  &= \sum_{c \in \mathcal C({\rm iii})} a(c)\partial_\alpha c\otimes 1 = \\
  &= \sum_{c \in \mathcal C({\rm iii})} a(c)\sum_{\mu = 0}^{n-1}\partial_{\tilde\alpha^\mu}\tilde c^{\mu+q(c)} = \\
  &= \sum_{\mu=0}^{n-1} \sum_{c \in \mathcal C({\rm iii})} a(c) \sum_{\nu = 0}^{n-1}\partial_{\tilde\alpha^\mu} \tilde c^{\nu}
\end{align*}
which together imply that 
\begin{align*}
  \partial_\alpha W\otimes 1 =\sum_{\mu=0}^{n-1} \partial_{\tilde\alpha^\mu} W_G.
\end{align*}
It remains to prove the claims (a3) and (b3).

    \textbf{Proof of (a3).}
      We have, for $c\in\mathcal C({\rm ii})$,
      $$
    \xymatrix{
      c  &:& \varepsilon_0 = \varepsilon_l \ar[rr]^-{g^{t_{1,l-1}}(\alpha_l)} & &
      g^{t_{1,l-1}}(\varepsilon_{l-1})\ar[r] & \cdots \ar[r] & g^{t_2}(\varepsilon_2)\ar[r]^-{\alpha_2} &
      \varepsilon_1
      \ar[r]^{\alpha_1} &\varepsilon_0,
    }
    $$
    where $\alpha = \alpha_1$, and this is the only copy of $\alpha_1$ in $c$.
    Hence $\tilde c^\nu = \tilde{\alpha}^\nu \widetilde{g^{-t_2}(\alpha_2)}^\nu\tilde{\alpha}_3\cdots \tilde{\alpha}_l$ and $\partial_{\tilde\alpha^\mu} \tilde c^\nu = 0$ for $\mu \neq \nu$.
    Then
\begin{align*}
      \partial_{\tilde \alpha^{\mu}}\tilde c^\mu &= 
      \partial_{\tilde \alpha^{\mu}} \tilde \alpha^\mu \widetilde{g^{-t_2}(\alpha_2)} \tilde \alpha_3 \cdots \tilde \alpha_l = \\
		&= \widetilde{g^{-t_2}(\alpha_2)} \tilde \alpha_3 \cdots \tilde \alpha_l  = \\
		&= (1\otimes e_{\mu})(g^{-t_2}(\alpha_2)\otimes 1)(\alpha_3\otimes g^{t_3})\cdots (\alpha_l\otimes g^{t_l})
	      \end{align*}
and so
\begin{align*}
  \sum_{\mu= 0}^{n-1}\zeta^{-t_2\mu} \partial_{\tilde \alpha^{\mu}}\tilde c^\mu &= \sum_{\mu= 0}^{n-1}(1\otimes e_{\mu})(1\otimes g^{t_2})(g^{-t_2}(\alpha_2)\otimes 1)(\alpha_3\otimes g^{t_3})\cdots (\alpha_l\otimes g^{t_l})=\\
  &= (\alpha_2\otimes g^{t_2})(\alpha_3\otimes g^{t_3})\cdots (\alpha_l\otimes g^{t_l}) = \\
  &= \partial_\alpha c\otimes 1
\end{align*}
as claimed.

    \textbf{Proof of (b3).}
 We have
       $$
    \xymatrix{
       c  &:& \varepsilon_0 = \varepsilon_l \ar[r]^-{\alpha_l}  &
      \varepsilon_{l-1}\ar[r] & \cdots \ar[r] &\varepsilon_1
      \ar[r]^{\alpha_1} &\varepsilon_0,
    }
  $$
  with $\alpha = \alpha_2$, and again observe that this is the only instance of $\alpha$ in $ c$.
  Write $b_i=b(\alpha_i)+\dots+b(\alpha_l)$ for $i\geq 3$, and recall that $b_3=q(c)$.
  Then $\tilde c^\nu = \tilde{\alpha}_1^\nu \tilde{\alpha}^{\nu-b_3} \tilde{\alpha}_3^{\nu-b_4} \cdots \tilde{\alpha}_l^\nu$, so $\partial_{\tilde\alpha^\mu} \tilde c^\nu = 0$ for $\mu \neq \nu-q(c)$. Hence
  \begin{align*}
\partial_{\tilde \alpha^{\mu}}\tilde c^{\mu+q(c)}
     &=  \partial_{\tilde \alpha^{\mu}} \tilde{\alpha}_1^{\mu+b_3} \tilde{\alpha}^{\mu} \tilde{\alpha}_3^{\mu+b_3-b_4} \cdots \tilde{\alpha}_l^{\mu+b_3} = \\
 &= \tilde{\alpha}_1^{\mu+b_3} \tilde{\alpha}^{\mu} \tilde{\alpha}_3^{\mu+b_3-b_4} \cdots \tilde{\alpha}_l^{\mu+b_3} \tilde{\alpha}_1^{\mu+b_3} = \\
 &= (\alpha_3\otimes e_{\mu+b_3-b_4})\cdots (\alpha_l\otimes e_{\mu+b_3})(1\otimes e_{\mu+b_3})(\alpha_1\otimes 1) = \\
 &= (\alpha_3\cdots\alpha_l\otimes e_{\mu+b_3})(1\otimes e_{\mu+b_3})(\alpha_1\otimes 1) = \\
 &= (\alpha_3\cdots\alpha_l\otimes e_{\mu+b_3})(\alpha_1\otimes 1)
  \end{align*}
so 
\begin{align*}
  \sum_{\mu= 0}^{n-1} \partial_{\tilde \alpha^{\mu}}\tilde c^{\mu+q(c)} = \alpha_3\cdots \alpha_l\alpha_1\otimes 1 = \partial_\alpha c\otimes 1
\end{align*}
which concludes the proof.

\item  We have that
  $$\partial_\alpha W\otimes 1 = \sum_{c \text{ of type }({\rm iii})} a(c) \partial_\alpha c \otimes 1 + \sum_{c \text{ of type }({\rm iv})} a(c) \partial_\alpha c \otimes 1$$
  and
  $$n\sum_{\mu = 0}^{n-1}\partial_{\tilde\alpha^\mu} W_G = n\sum_{\mu = 0}^{n-1}\sum_{c \in \mathcal C({\rm iii})} a(c) \sum_{\nu=0}^{n-1} \partial_{\tilde\alpha^\mu}\tilde c^\nu + n\sum_{\mu = 0}^{n-1}\sum_{c \in\mathcal C({\rm iv})} a(c) \sum_{\nu=0}^{n-1} \partial_{\tilde\alpha^\mu}\tilde c^\nu.$$
  The statement will be proved using the following two claims:
  
    \textbf{Claim (a4).} If $c\in\mathcal C({\rm iii})$, then  
 $$\sum_{r=0}^{n-1} \partial_\alpha g^rc \otimes 1 = \sum_{\mu=0}^{n-1} \sum_{\nu = 0}^{n-1}\partial_{\tilde\alpha^\mu} \tilde c^{\nu} .$$
 
   \textbf{Claim (b4).} If $c\in\mathcal C({\rm iv})$, then 
      $$ \partial_\alpha c \otimes 1 = \sum_{\mu=0}^{n-1} \sum_{\nu = 0}^{n-1}\partial_{\tilde\alpha^\mu} \tilde c^{\nu} .$$
   
Assuming these claims hold, let us prove the statement. 
We have that
\begin{align*}
  \sum_{c \text{ of type }({\rm iii})} a(c) \partial_\alpha c \otimes 1 &= \sum_{c \in \mathcal C({\rm iii})} \sum_{r=0}^{n-1} a(g^r*c)\partial_{\alpha}(g^r*c)\otimes 1 = \\
  &=  \sum_{c \in \mathcal C({\rm iii})} \sum_{r=0}^{n-1}\zeta^{rq(c)}a(c)\partial_{\alpha}(g^r*c)\otimes 1 	= \\
  &=  \sum_{c \in \mathcal C({\rm iii})} \sum_{r=0}^{n-1}a(c)\partial_{\alpha}g^r c\otimes 1 = \\
  &=  n\sum_{\mu = 0}^{n-1}\sum_{c \in \mathcal C({\rm iii})} a(c) \sum_{\nu=0}^{n-1} \partial_{\tilde\alpha^\mu}\tilde c^\nu
\end{align*}
and
\begin{align*}
  \sum_{c \text{ of type }({\rm iv})} a(c) \partial_\alpha c \otimes 1 &= \sum_{c \in \mathcal C({\rm iv})} \sum_{r=0}^{n-1} a(g^r*c)\partial_{\alpha}(g^r*c)\otimes 1 = \\
 &= \sum_{c \in \mathcal C({\rm iv})} \sum_{r=0}^{n-1} a(c) \partial_{\alpha}c\otimes 1 = \\
 &= n\sum_{c \in \mathcal C({\rm iv})}a(c) \partial_{\alpha}c\otimes 1 = \\
 &=  n\sum_{\mu = 0}^{n-1}\sum_{c \in\mathcal C({\rm iv})} a(c) \sum_{\nu=0}^{n-1} \partial_{\tilde\alpha^\mu}\tilde c^\nu
\end{align*}
which together imply that
\begin{align*}
  \partial_\alpha W\otimes 1 =n\sum_{\mu = 0}^{n-1}\partial_{\tilde\alpha^\mu} W_G.
\end{align*}
It remains to prove the claims (a4) and (b4).
  
     \textbf{Proof of (a4).}
   Let us write, for $c\in\mathcal C({\rm iii})$,
  $$
    \xymatrix{
      c  &:& \varepsilon_0 = \varepsilon_l \ar[r]^-{\alpha_l}  &
      \varepsilon_{l-1}\ar[r] & \cdots \ar[r] &\varepsilon_1
      \ar[r]^{\alpha_1} &\varepsilon_0,
    }
  $$
  where $\varepsilon_1\in\cal E'$ and $\varepsilon_i\in\cal E''$ for $i\neq1$.
  
  Let $M = \{m\in\{1,\dots,l\} \,|\, \alpha=\alpha_m\}$ and put $b_i=b(\alpha_i)+\dots+b(\alpha_l)$ for all $i\geq3$.
  We have
  $$
    \xymatrix{
      \tilde c^\nu  &:& \eta_\nu^{\varepsilon_l} \ar[rr]^-{\tilde\alpha_l^\nu} & &
      \eta_{\nu-b_l}^{\varepsilon_{l-1}}\ar[rr]^-{\tilde\alpha_{l-1}^{\nu-b_l}} & & \cdots \ar[rr]^-{\tilde\alpha_2^{\nu-b_3}} & &\eta^{\varepsilon_1}
      \ar[r]^{\tilde\alpha_1^{\nu}} &\eta_\nu^{\varepsilon_0},
    }
  $$
  so we may note that, if $m\in M$, the $m$-th arrow of $\tilde c^\nu$ is $\tilde\alpha_m^{\nu-b_{m+1}}$, and it coincides with $\tilde\alpha^\mu$ if and only if $\nu=\mu+b_{m+1}$. Hence
  \begin{align*}
    \sum_{\mu=0}^{n-1} \sum_{\nu=0}^{n-1} \partial_{\tilde\alpha^\mu}\tilde c^\nu &= \sum_{\mu=0}^{n-1} \sum_{\nu=0}^{n-1} \partial_{\tilde\alpha^\mu} \tilde{\alpha}_1^\nu \tilde{\alpha}_2^{\nu-b_3} \tilde{\alpha}_3^{\nu-b_4} \cdots \tilde{\alpha}_{l-1}^{\nu-b_l} \tilde{\alpha}_l^\nu =\\
    &= \sum_{\mu=0}^{n-1} \sum_{m\in M} \tilde\alpha_{m+1}^{\mu+b_{m+1}-b_{m+2}} \cdots \tilde\alpha_l^{\mu+b_{m+1}} \tilde\alpha_1^{\mu+b_{m+1}} \tilde\alpha_2^{\mu+b_{m+1}-b_3} \cdots \tilde\alpha_{m-1}^{\mu+b_{m+1}-b_m} =\\
    &= \sum_{\mu=0}^{n-1} \sum_{m\in M} (\alpha_{m+1}\otimes e_{\mu+b_{m+1}-b_{m+2}}) \cdots (\alpha_l\otimes e_{\mu+b_{m+1}}) \\
    & (1\otimes e_{\mu+b_{m+1}})(\alpha_1\otimes 1)(\alpha_2\otimes e_{\mu+b_{m+1}-b_3}) \cdots (\alpha_{m-1}\otimes e_{\mu+b_{m+1}-b_m}) =\\
    &= \sum_{\mu=0}^{n-1} \sum_{m\in M} (\alpha_{m+1} \cdots \alpha_l\otimes e_{\mu+b_{m+1}})(\alpha_1\alpha_2 \cdots \alpha_{m-1}\otimes e_{\mu+b_{m+1}-b_m}) =\\
    &= \sum_{\mu=0}^{n-1} \sum_{m\in M} \frac{1}{n} \sum_{i=0}^{n-1} \zeta^{i(\mu+b_{m+1})} (\alpha_{m+1} \cdots \alpha_l\otimes g^i)(\alpha_1\alpha_2 \cdots \alpha_{m-1}\otimes e_{\mu+b_{m+1}-b_m}) =\\
    &= \sum_{\mu=0}^{n-1} \sum_{m\in M} \frac{1}{n} \sum_{i=0}^{n-1} \zeta^{i(\mu+b_{m+1})} \alpha_{m+1} \cdots \alpha_l g^i(\alpha_1\alpha_2 \cdots \alpha_{m-1}) \otimes g^i e_{\mu+b_{m+1}-b_m} =\\
    &= \sum_{\mu=0}^{n-1} \sum_{m\in M} \frac{1}{n} \sum_{i=0}^{n-1} \zeta^{ib_m} \alpha_{m+1} \cdots \alpha_l g^i(\alpha_1\alpha_2 \cdots \alpha_{m-1}) \otimes e_{\mu+b_{m+1}-b_m} =\\
    &= \sum_{\mu=0}^{n-1} \sum_{m\in M} \frac{1}{n} \sum_{i=0}^{n-1} \zeta^{ib_3} \alpha_{m+1} \cdots \alpha_l g^i(\alpha_1\alpha_2) \cdots \alpha_{m-1} \otimes e_{\mu+b_{m+1}-b_m} =\\
    &= \sum_{\mu=0}^{n-1} \sum_{m\in M} \frac{1}{n} \sum_{i=0}^{n-1} \partial_\alpha g^i c \otimes e_{\mu+b_{m+1}-b_m} =\\
    &= \sum_{m\in M} \frac{1}{n} \sum_{i=0}^{n-1} \partial_\alpha g^i c \otimes 1 =\\
    &= \frac{1}{n} \sum_{i=0}^{n-1} \partial_\alpha g^i c \otimes 1
  \end{align*}
  which is what we wanted to prove.
  
\textbf{Proof of (b4).} Let
  $$
    \xymatrix{
      c  &:& \varepsilon_0 = \varepsilon_l \ar[r]^-{\alpha_l}  &
      \varepsilon_{l-1}\ar[r] & \cdots \ar[r] &\varepsilon_1
      \ar[r]^{\alpha_1} &\varepsilon_0,
    }
  $$
   where $\varepsilon_i\in\cal E''$ for all $i=1,\dots,l$. Let $M = \{m\in\{1,\dots,l\} \,|\, \alpha=\alpha_m\}$, and put as usual \mbox{$b_i=b(\alpha_i)+\dots+b(\alpha_l)$} for all $i$.
  We have
  $$
    \xymatrix{
      \tilde c^\nu  &:& \eta_\nu^{\varepsilon_l} \ar[rr]^-{\tilde\alpha_l^\nu} & &
      \eta_{\nu-b_l}^{\varepsilon_{l-1}}\ar[rr]^-{\tilde\alpha_{l-1}^{\nu-b_l}} & & \cdots \ar[rr]^-{\tilde\alpha_2^{\nu-b_3}} & &\eta_{\nu-b_2}^{\varepsilon_1}
      \ar[r]^{\tilde\alpha_1^{\nu-b_2}} &\eta_\nu^{\varepsilon_0},
    }
  $$
  hence
  \begin{align*}
    \sum_{\mu=0}^{n-1} \sum_{\nu=0}^{n-1} \partial_{\tilde\alpha^\mu}\tilde c^\nu &= \sum_{\mu=0}^{n-1} \sum_{\nu=0}^{n-1} \partial_{\tilde\alpha^\mu} \tilde{\alpha}_1^{\nu-b_2} \tilde{\alpha}_2^{\nu-b_3} \cdots \tilde{\alpha}_{l-1}^{\nu-b_l} \tilde{\alpha}_l^\nu =\\
    &= \sum_{\mu=0}^{n-1} \sum_{m\in M} \tilde\alpha_{m+1}^{\mu+b_{m+1}-b_{m+2}} \cdots \tilde\alpha_{m-1}^{\mu+b_{m+1}-b_m} =\\
    &= \sum_{\mu=0}^{n-1} \sum_{m\in M} (\alpha_{m+1}\otimes e_{\mu+b_{m+1}-b_{m+2}}) \cdots (\alpha_{m-1}\otimes e_{\mu+b_{m+1}-b_m}) =\\
    &= \sum_{\mu=0}^{n-1} \sum_{m\in M} \alpha_{m+1} \cdots \alpha_{m-1}\otimes e_{\mu+b_{m+1}-b_m} =\\
    &= \sum_{m\in M} \alpha_{m+1} \cdots \alpha_{m-1}\otimes 1 =\\
    &= \partial_\alpha c \otimes 1,
  \end{align*}
  and the claim is proved.\qedhere
\end{enumerate}
\end{proof}

\subsection{Isomorphism of algebras}\label{sec:proof}

We are now ready to prove our main result.

\begin{proof}[Proof of Theorem \ref{thm:main}]
  We will first prove that 
  \begin{align*} 
    \frac{\kk Q_G}{\langle \partial_{\gamma}W_G\ |\ \gamma \in (Q_G)_1\rangle}\cong \eta(\Lambda G)\eta.
  \end{align*}
  By Lemma \ref{lem:key}, the right-hand side is isomorphic to 
  \begin{align*}
    \frac{\eta ( (\kk Q)G)\eta}{\langle \partial_{g^{-t(\alpha)}\alpha} W\otimes g^{-t(\alpha)} \ |\ \alpha \text{ 
    of type } (1), (2), (3), (4)\rangle},
  \end{align*}
  and by \cite[\S2.2,\S2.3]{RR85} we have that $\kk Q_G \cong \eta((\kk Q)G)\eta$ via the isomorphism $J$ of  \S\ref{sec:quiver}. For every arrow $\alpha$ of $Q$ of type (1),(2),(3),(4), we can write
  (recall that for types (2),(3),(4) we set $t(\alpha ) = 0$)
  $$J^{-1}\left(\partial_{g^{-t(\alpha)}\alpha} W\otimes g^{-t(\alpha)}\right) = \sum_{i,j\in (Q_G)_0} x_{ij}$$ such that $x_{ij}$ are linear combinations of paths from $i$ to $j$ in $\kk Q_G$.
  By Lemma \ref{lem:1234}, every nonzero $x_{ij}$ is associated in $\kk Q_G$ to a unique element of the form $\partial_\gamma W_G$ for some $\gamma\in (Q_G)_1$, and moreover every nonzero $\partial_\gamma W_G$ appears 
  in this way for some $\alpha$.
  This means that $$
  J\left( \langle \partial_{\gamma}W_G\ |\ \gamma \in (Q_G)_1\rangle \right) = \langle \partial_{g^{-t(\alpha)}\alpha} W\otimes g^{-t(\alpha)} \ |\ \alpha \text{ 
    of type } (1), (2), (3), (4)\rangle
    $$
    so the claim is proved.
    Now notice that by Lemma \ref{lem:admissible}, the ideal $\langle \partial_{\gamma}W_G\ |\ \gamma \in (Q_G)_1\rangle\subseteq \kk Q_G$ is admissible, so by Proposition \ref{prop:completion} we conclude
    that 
    \begin{align*}
      \mathcal P(Q_G, W_G) \cong \frac{\kk Q_G}{\langle \partial_{\gamma}W_G\ |\ \gamma \in (Q_G)_1\rangle}
    \end{align*}
    and we are done. 
\end{proof}

\section{Dual group action}\label{sec:dual}
It was proved in \cite{RR85} that we can always recover the algebra $\Lambda$ from $\Lambda G$ by applying another skew group algebra construction. In this section we will show that in our case this construction satisfies again the assumptions \ref{ass:field}-\ref{ass:cycles}, and the potential we obtain corresponds to the potential we started with.

Let $\Lambda$ be a finite dimensional algebra and $G$ be a finite abelian group acting on $\Lambda$ by automorphisms.
We denote by $\hat G$ the dual group of $G$. Its elements are the group homomorphism $\chi\colon G\to \kk^*$.

\begin{thm}[{\cite[Corollary 5.2]{RR85}}]\label{thm:RR}
  Define an action of $\hat G$ on $\Lambda G$ by $\chi(\lambda\otimes g)=\chi(g)\lambda\otimes g$, \mbox{$\lambda\in\Lambda$}, \mbox{$g\in G$}. Then the skew group algebra $(\Lambda G)\hat G$ is Morita equivalent to $\Lambda$.
\end{thm}

We want to apply Theorem \ref{thm:RR} to our setting, so we retain the notation of Section~\ref{sec:setup} (in particular we are assuming that
$\Lambda = \mathcal P(Q,W)$). Since $G$ is finite and cyclic, there is an isomorphism $G\cong\hat G$. We can write $\hat G = \{\chi_0,\dots,\chi_{n-1}\}$, where we define $\chi_\mu$ to be the homomorphism which sends $g$ to $\zeta^\mu$. Put $\chi=\chi_1$ and note that it is a generator of $\hat G$.

Recall that, by Theorem~\ref{thm:main}, we have an isomorphism $\cal P(Q_G,W_G)\cong\eta(\Lambda G)\eta$, where $\eta\in \Lambda G$ is an idempotent such that $\eta(\Lambda G)\eta$ is Morita equivalent to $\Lambda G$ and $(Q_G,W_G)$ is the QP described in \S\ref{sec:quiver}.

We will now show that the process of getting back $\Lambda$ from $\Lambda G$ is achieved via a construction which satisfies the assumptions \ref{ass:field}-\ref{ass:cycles}. 

\begin{prop}
  \label{prop:main2a}
  The dual group $\hat G$ acts on $\cal P(Q_G,W_G)$ by automorphisms and $\cal P(Q_G,W_G)\hat G$ is Morita equivalent to $\Lambda$. Moreover this action satisfies the assumptions \ref{ass:field}-\ref{ass:cycles}.
\end{prop}
\begin{proof}
  Since $\eta=\e\otimes1$ for an idempotent $\e\in\Lambda$, we have that $\hat G$ acts trivially on $\eta$ and so the action of $\hat G$ on $\Lambda G$ restricts to an action on $\eta(\Lambda G)\eta\cong \cal P(Q_G,W_G)$.
  Hence, by \cite[Lemma~2.2]{RR85}, we have that $(\eta(\Lambda G)\eta)\hat G$ is Morita equivalent to $(\Lambda G)\hat G$, and the latter is Morita equivalent to $\Lambda$ by Theorem~\ref{thm:RR}.
  So the first assertion is proved and we are left to check that the action of $\hat G$ on $(Q_G,W_G)$ satisfies the assumptions \ref{ass:field}-\ref{ass:cycles}.
  
  Assumption \ref{ass:field} holds because $\hat G$ has the same order of $G$.

  If $\varepsilon\in \cal E'$, then $\chi(\eta^{\varepsilon})=\chi(\varepsilon\otimes1)=\eta^{\varepsilon}$. If $\varepsilon'\in \cal E''$ and $0\leq\mu\leq n-1$, then
  \begin{align*}
    \chi(\eta^{\varepsilon}_\mu) = \chi(\varepsilon\otimes e_\mu) = \frac{1}{n}\sum_{i=0}^{n-1} \zeta^{i\mu} \chi(\varepsilon\otimes g^i) = \frac{1}{n}\sum_{i=0}^{n-1} \zeta^{i(\mu+1)} \varepsilon\otimes g^i = \varepsilon\otimes e_{\mu+1} = \eta^{\varepsilon}_{\mu+1}.
  \end{align*}
  Hence $\hat G$ permutes the vertices of $Q_G$. In particular assumption \ref{ass:cardinality} holds.

  Now we consider the action on the arrows of $Q_G$. Four cases have to be analysed.
  \begin{enumerate}
    \item[(1)] Let $\alpha$ be an arrow of type (1) in $Q$. Then we have an arrow $\tilde\alpha = \alpha\otimes g^{t(\alpha)}$ in $Q_G$ and $\hat G$ acts on it as
    \begin{align*}
      \chi(\tilde\alpha) = \chi(\alpha\otimes g^{t(\alpha)}) = \chi(g^{t(\alpha)}) \alpha\otimes g^{t(\alpha)} = \zeta^{t(\alpha)} \alpha\otimes g^{t(\alpha)} = \zeta^{t(\alpha)} \tilde\alpha.
    \end{align*}
    \item[(2)] Let $\alpha$ be an arrow of type (2) in $Q$ and $0\leq\mu\leq n-1$. Then $\hat G$ acts on $\tilde\alpha^\mu = (1\otimes e_\mu)(\alpha\otimes1)$ as
    \begin{align*}
      \chi(\tilde\alpha^\mu) = \chi((1\otimes e_\mu)(\alpha\otimes1)) = (1\otimes e_{\mu+1})(\alpha\otimes1) = \tilde\alpha^{\mu+1}.
    \end{align*}
        \item[(3),(4)] Let $\alpha$ be an arrow of type either (3) or (4) in $Q$ and $0\leq\mu\leq n-1$. Then $\hat G$ acts on $\tilde\alpha^\mu = \alpha\otimes e_\mu$ as
    \begin{align*}
      \chi(\tilde\alpha^\mu) = \chi(\alpha\otimes e_\mu) = \alpha\otimes e_{\mu+1} = \tilde\alpha^{\mu+1}.
    \end{align*}
  \end{enumerate}
  This proves assumptions \ref{ass:permuted} and \ref{ass:fixed}.

  From these calculations we can deduce how $\hat G$ acts on the cycles of $W_G$. Again we distinguish four cases.
  \begin{enumerate}
    \item[(i)] Let $c$ be a cycle of type (i) and write $\tilde c=\tilde\alpha_1\cdots\tilde\alpha_l$. Then, observing that $t(\alpha_1)+\cdots+t(\alpha_l)=0\pmod n$,
      we get $\chi(\tilde c)=\zeta^{t(\alpha_1)+\cdots+t(\alpha_l)}\tilde c=\tilde c$.
    \item[(ii)] Let $c$ be a cycle of type (ii) and $0\leq\mu\leq n-1$. Write $\tilde c^\mu=\tilde\alpha_1^\mu\widetilde{g^{-p(c)}(\alpha_2)}^\mu\tilde\alpha_3\cdots \tilde\alpha_l$.
      Then we get $\chi(\tilde c^\mu)=\zeta^{t(\alpha_3)+\cdots+t(\alpha_l)}\tilde c^{\mu+1}=\zeta^{-t(\alpha_2)}\tilde c^{\mu+1}=\zeta^{-p(c)}\tilde c^{\mu+1}$, since $t(\alpha_1)+\cdots+t(\alpha_l)=0\pmod n$ and $t(\alpha_1)=0$.
    \item[(iii)] Let $c$ be a cycle of type (iii) and $0\leq\mu\leq n-1$. Write $\tilde c^\mu=\tilde\alpha_1^\mu\tilde\alpha_2^\mu\tilde\alpha_3^\mu\cdots \tilde\alpha_l^\mu$. Then we get $\chi(\tilde c^\mu)=\tilde c^{\mu+1}$.
    \item[(iv)] Let $c$ be a cycle of type (iv) and $0\leq\mu\leq n-1$. Write $\tilde c^\mu = \tilde{\alpha}_1^{\mu-b_2} \tilde{\alpha}_2^{\mu-b_3} \cdots \tilde{\alpha}_{l-1}^{\mu-b_l} \tilde{\alpha}_l^\mu$. Then we get $\chi(\tilde c^\mu)=\tilde c^{\mu+1}$.
  \end{enumerate}
  So assumption \ref{ass:cycles} is proved.
  
  Finally we get that
  \begin{align*}
  \chi (W_G) =& \sum_{c\in\cal C({\rm i})}a(c) \chi(\tilde c) +\sum_{c\in\cal C({\rm ii})}a(c)\sum_{\mu = 0}^{n-1}\zeta^{-p(c)\mu} \chi(\tilde c^{\mu})+ \\
  &+ \sum_{c\in\cal C({\rm iii})}a(c)\sum_{\mu = 0}^{n-1} \chi(\tilde c^{\mu}) + \sum_{c\in\cal C({\rm iv})}a(c)\sum_{\mu = 0}^{n-1} \chi(\tilde c^{\mu})= \\
  =& \sum_{c\in\cal C({\rm i})}a(c)\tilde c +\sum_{c\in\cal C({\rm ii})}a(c)\sum_{\mu = 0}^{n-1}\zeta^{-p(c)(\mu+1)}\tilde c^{\mu+1}+ \\
  &+ \sum_{c\in\cal C({\rm iii})}a(c)\sum_{\mu = 0}^{n-1}\tilde c^{\mu+1} + \sum_{c\in\cal C({\rm iv})}a(c)\sum_{\mu = 0}^{n-1}\tilde c^{\mu+1}= \\
  =& W_G,
  \end{align*}
  so the potential $W_G$ is fixed by $\hat G$ and thus assumption \ref{ass:potential} holds.
\end{proof}

To sum up, we have an action of $\hat G$ on the Jacobian algebra $\cal P(Q_G,W_G)$ which satisfies the assumptions \ref{ass:field}-\ref{ass:cycles}. Using the procedure described in Section~\ref{sec:setup} we can construct from it a new QP $((Q_G)_{\hat G},(W_G)_{\hat G})$ whose Jacobian algebra is Morita equivalent to $\Lambda$. Now we want to construct an explicit isomorphism $\cal P((Q_G)_{\hat G},(W_G)_{\hat G}) \cong \Lambda$.

Firstly, let us give an explicit description of $((Q_G)_{\hat G},(W_G)_{\hat G})$.

Let $\cal E_G=\cal E'_G \sqcup \cal E''_G$, where $\cal E'_G=\{\eta_0^\varepsilon \,|\, \varepsilon\in\cal E''\}$ and $\cal E''_G=\{\eta^\varepsilon \,|\, \varepsilon\in\cal E'\}$. Then $\cal E_G$ is a set of representatives for the orbits of the action of $\hat G$ on $Q_G$. The elements of $\cal E'_G$ and $\cal E''_G$ have orbits of cardinality $n$ and $1$ respectively.

The arrows of $Q_G$ can be divided into four families, according to whether their starting and ending points are fixed or not by the action of $\hat G$.

\begin{enumerate}[label=(\arabic*)]
  \item Arrows between two non-fixed vertices. These are all the arrows of the form $\tilde \alpha^\mu\colon \eta^{\varepsilon}_\mu \to \eta^{\varepsilon'}_{\mu-b(\alpha)}$, where $\alpha: \varepsilon\to \varepsilon'$ is an arrow of type (4) in $Q$ and $0\leq\mu\leq n-1$. Among them, the arrows which are of type (1) with respect to the action of $\hat G$ on $Q_G$ are the ones which end in $\cal E'_G$, i.e.,~the ones of the form $\tilde\alpha^{b(\alpha)}\colon \eta^\varepsilon_{b(\alpha)} \to \eta^{\varepsilon'}_0$. Since $\eta^\varepsilon_{b(\alpha)}=\chi^{b(\alpha)}(\eta^\varepsilon_0)$, we have that $t(\tilde\alpha^{b(\alpha)})=b(\alpha)$.
  \item Arrows from a non-fixed vertex to a fixed one. These are all the arrows of the form $\tilde \alpha^\mu\colon \eta^\varepsilon_\mu \to \eta^{\varepsilon'}$, where $\alpha:\varepsilon\to \varepsilon'$ is an arrow of type (3) in $Q$ and $0\leq\mu\leq n-1$. Among them, the arrows which are of type (2) with respect to the action of $\hat G$ on $Q_G$ are the ones which start in $\cal E'_G$, i.e.,~the ones of the form $\tilde\alpha^0\colon \eta^\varepsilon_0 \to \eta^{\varepsilon'}$.
  \item Arrows from a fixed vertex to a non-fixed one. These are all the arrows of the form $\tilde \alpha^\mu\colon \eta^{\varepsilon} \to \eta^{\varepsilon'}_\mu$, where $\alpha:\varepsilon\to \varepsilon'$ is an arrow of type (2) in $Q$ and $0\leq\mu\leq n-1$. Among them, the arrows which are of type (3) with respect to the action of $\hat G$ on $Q_G$ are the ones which end in $\cal E'_G$, i.e.,~the ones of the form $\tilde\alpha^0\colon \eta^\varepsilon \to \eta^{\varepsilon'}_0$.
  \item Arrows between two fixed vertices. These are all the arrows of the form $\tilde\alpha\colon \eta^{\varepsilon} \to \eta^{\varepsilon'}$, where $\alpha:\varepsilon\to \varepsilon'$ is an arrow of type (1) in $Q$. All of them are of type (4) with respect to the action of $\hat G$ on $Q_G$. Since $\chi(\tilde\alpha)=\zeta^{t(\alpha)} \tilde\alpha$, we have that $b(\tilde\alpha)=t(\alpha)$.
\end{enumerate}

We deduce that the quiver $(Q_G)_{\hat G}$ is made as follows. Its vertices are $\eta^\varepsilon_0\otimes1$ for $\varepsilon\in\cal E''$ and $\eta^\varepsilon\otimes e_\nu$ for $\varepsilon\in\cal E'$, $0\leq\nu\leq n-1$, 
while its arrows are the following:
\begin{enumerate}[label = (\arabic*)]
  \item $\tilde\beta\colon \eta^\varepsilon_0\otimes1 \to \eta^{\varepsilon'}_0\otimes1$, where $\beta=\tilde\alpha^{b(\alpha)}$ and $\alpha:\varepsilon\to \varepsilon'$ is an arrow of type (4) in $Q$,
  \item $\tilde\beta^\nu\colon \eta^\varepsilon_0\otimes1 \to \eta^{\varepsilon'}\otimes e_\nu$, where $\beta=\tilde\alpha^0$, $0\leq\nu\leq n-1$ and $\alpha:\varepsilon\to \varepsilon'$ is an arrow of type (3) in $Q$,
  \item $\tilde\beta^\nu\colon \eta^\varepsilon\otimes e_\nu \to \eta^{\varepsilon'}_0\otimes1$, where $\beta=\tilde\alpha^0$, $0\leq\nu\leq n-1$ and $\alpha:\varepsilon\to \varepsilon'$ is an arrow of type (2) in $Q$,
  \item $\tilde\beta^\nu\colon \eta^\varepsilon\otimes e_\nu \to \eta^{\varepsilon'}\otimes e_{\nu-t(\alpha)}$, where $\beta=\tilde\alpha$, $0\leq\nu\leq n-1$ and $\alpha:\varepsilon\to \varepsilon'$ is an arrow of type (1) in $Q$.
\end{enumerate}

\begin{prop}
  \label{prop:main2b}
Let $\phi\colon (Q_G)_{\hat G} \to Q$ be the morphism of quivers defined as follows.
\begin{itemize}
  \item $\phi(\eta^\varepsilon_0\otimes1) = \varepsilon$ for $\varepsilon\in\cal E''$.
  \item $\phi(\eta^\varepsilon\otimes e_\mu) = g^\mu(\varepsilon)$ for $\varepsilon\in\cal E'$, $0\leq\mu\leq n-1$.
  \item $\phi(\tilde\beta) = \alpha$, where $\beta=\tilde\alpha^{b(\alpha)}$ and $\alpha$ is an arrow of type (4) in $Q$.
  \item $\phi(\tilde\beta^\nu) = g^\nu(\alpha)$, where $\beta=\tilde\alpha^0$, $0\leq\nu\leq n-1$ and $\alpha$ is an arrow of type (3) in $Q$.
  \item $\phi(\tilde\beta^\nu) = g^\nu(\alpha)$, where $\beta=\tilde\alpha^0$, $0\leq\nu\leq n-1$ and $\alpha$ is an arrow of type (2) in $Q$.
  \item $\phi(\tilde\beta^\nu) = g^{\nu-t(\alpha)}(\alpha)$, where $\beta=\tilde\alpha$, $0\leq\nu\leq n-1$ and $\alpha$ is an arrow of type (1) in $Q$.
\end{itemize}
Then $\phi$ is an isomorphism and, if we extend it to an isomorphism between the corresponding path algebras, we have $\phi((W_G)_{\hat G})=W$.
\end{prop}

\begin{proof}
We first note that $\phi$ is a well defined morphism of quivers. Moreover, by what we observed earlier in this section, $\phi$ is a bijection on both the sets of vertices and arrows, thus it is an isomorphism.

Given the set $\cal E_G$ defined above, we can choose a set $\cal C_G= \{\hat d \,|\, d \text{ cycle in } W_G\}$ of representatives for the $*$ action of $\hat G$ on cycles as in \S\ref{subsec:cycles}. We have that $\cal C_G = \cal C_G({\rm i}) \sqcup \cal C_G({\rm ii}) \sqcup \cal C_G({\rm iii}) \sqcup \cal C_G({\rm iv})$.
We now describe each of these four subsets and show where their elements are sent by $\phi$. We use the notation $t_{i,j}$ of the proof of Lemma \ref{lem:1234}.

\begin{enumerate}[label=(\roman*)]
  \item Cycles of type (i) in $Q_G$ are the ones of the form $d=\tilde c^\mu$, where $c\in\cal C({\rm iv})$.
    If we write \mbox{$c=\alpha_1\cdots\alpha_l$} for some arrows $\alpha_i$ of type (4) in $Q$, then $\tilde c^\mu = \tilde{\alpha}_1^{\mu-b_2} \tilde{\alpha}_2^{\mu-b_3} \tilde{\alpha}_3^{\mu-b_4} \cdots \tilde{\alpha}_{l-1}^{\mu-b_l} \tilde{\alpha}_l^\mu$, where \mbox{$b_i=b(\alpha_i)+\dots+b(\alpha_l)$}.
  Hence we can choose $\hat d = \tilde{\alpha}_1^{-b_2} \tilde{\alpha}_2^{-b_3} \tilde{\alpha}_3^{-b_4} \cdots \tilde{\alpha}_{l-1}^{-b_l} \tilde{\alpha}_l^0=\tilde c^0$, and $\cal C_G({\rm i})$ is the subset of all the cycles of this kind.
  Moreover we have that $\tilde d = \tilde\beta_1 \cdots \tilde\beta_l$, where $\beta_i=\tilde\alpha_i^{b(\alpha_i)}$. It follows that
  $$\phi(\tilde d) = \phi(\tilde\beta_1 \cdots \tilde\beta_l) = \alpha_1\cdots\alpha_l = c.$$
  Let us now look at the coefficient $a(d)$ of $d$ as a summand of $W_G$. 
  The cycle $c$ of $W$ gives rise to a number $x = |\hat G \tilde c^\mu|$ of distinct cycles in $W_G$ (this does not depend on the choice of $\mu$). 
  Then $a(d) = a(c)\frac{n}{x}$.
  
  \item Cycles of type (ii) in $Q_G$ are the ones of the form $d=\tilde c^\mu$, where $c\in\cal C({\rm iii})$.
    If we write $c=\alpha_1\alpha_2\cdots\alpha_l$ for $\alpha_1$ of type (2), $\alpha_2$ of type (3), and $\alpha_3,\dots,\alpha_l$ of type (4) in $Q$, then \mbox{$\tilde c^\mu = \tilde{\alpha}_1^{\mu} \tilde{\alpha}_2^{\mu-b_3} \tilde{\alpha}_3^{\mu-b_4} \cdots \tilde{\alpha}_{l-1}^{\mu-b_l} \tilde{\alpha}_l^\mu$}, where we write $b_i=b(\alpha_i)+\dots+b(\alpha_l)$.
    Hence we obtain that \mbox{$\hat d = \tilde{\alpha}_1^0 \tilde{\alpha}_2^{-b_3} \tilde{\alpha}_3^{-b_4} \cdots \tilde{\alpha}_{l-1}^{-b_l} \tilde{\alpha}_l^0=\tilde c^0$}, and $\cal C_G({\rm ii})$ is the subset of all the cycles of this kind.
  Moreover we have that $\tilde d^\nu = \tilde\beta_1^\nu \tilde\beta_2^\nu \tilde\beta_3 \cdots \tilde\beta_l$, where $\beta_1=\tilde\alpha_1^0$, $\beta_2=\tilde\alpha_2^0$ and $\beta_i=\tilde\alpha_i^{b(\alpha_i)}$ for $i\geq3$. It follows that (recall that by definition $q(c)=b_3$)
  $$\phi(\tilde d) = \phi(\tilde\beta_1^\nu \tilde\beta_2^\nu \tilde\beta_3 \cdots \tilde\beta_l) = g^\nu(\alpha_1)g^\nu(\alpha_2)\alpha_3\cdots\alpha_l = \zeta^{-b_3} g^\nu(c) = \zeta^{-q(c)} g^\nu(c).$$
  Note that $\beta_2=\chi^{b_3} \tilde\alpha_2^{-b_3}$. This implies that $p(d)=-q(c)$ and so $\phi(\tilde d) = \zeta^{p(d)} g^\nu(c)$.
  
  \item Cycles of type (iii) in $Q_G$ are the ones of the form $d=\tilde c^\mu$, where $c\in\cal C({\rm ii})$.
    If we write \mbox{$c=\alpha_1 \alpha_2 g^{t_2}(\alpha_3) \cdots g^{t_{2,l-1}}(\alpha_l)$} for $\alpha_1$ of type (3), $\alpha_2$ of type (2), and $\alpha_3,\dots,\alpha_l$ of type (1) in $Q$, then $\tilde c^\mu = \tilde{\alpha}_1^\mu \widetilde{g^{-t_2}(\alpha_2)}^\mu\tilde{\alpha}_3\cdots \tilde{\alpha}_l$.
  Hence $\hat d = \tilde{\alpha}_1^0 \widetilde{g^{-t_2}(\alpha_2)}^0 \tilde{\alpha}_3\cdots \tilde{\alpha}_l=\tilde c^0$, and $\cal C_G({\rm iii})$ is the subset of all the cycles of this kind.
  Now define $\beta_1=\tilde\alpha_1^0$, $\beta_2=\widetilde{g^{-t_2}(\alpha_2)}^0$ and $\beta_i=\tilde\alpha_i$ for $i\geq3$. Recall that, for $i\geq3$, $\chi(\beta_i)=\zeta^{t(\alpha_i)}\beta_i$, so $b(\beta_i)=t(\alpha_i)$. If we put $b'_i=b(\beta_i)+\dots+b(\beta_l)$ for $i\geq3$, we have that $\tilde d^\nu = \tilde\beta_1^\nu \tilde\beta_2^{\nu-b'_3} \tilde\beta_3^{\nu-b'_4} \cdots \tilde\beta_{l-1}^{\nu-b'_l} \tilde\beta_l^\nu$. Then
  \begin{align*}
    \phi(\tilde d^\nu) &= \phi(\tilde\beta_1^\nu \tilde\beta_2^{\nu-b'_3} \tilde\beta_3^{\nu-b'_4} \cdots \tilde\beta_{l-1}^{\nu-b'_l} \tilde\beta_l^\nu) = \\
    &= g^\nu(\alpha_1)g^{\nu-b'_3}(g^{-t_2}(\alpha_2))g^{\nu-b'_4-t(\alpha_3)}(\alpha_3) \cdots g^{\nu-t(\alpha_l)}(\alpha_l) = \\
    &= g^\nu(\alpha_1 g^{-t_{2,l}}(\alpha_2) g^{-t_{3,l}}(\alpha_3) \cdots g^{-t_l}(\alpha_l))  = \\
    &= g^\nu(\alpha_1 \alpha_2 g^{t_2}(\alpha_3) \cdots g^{t_{2,{l-1}}}(\alpha_l))  = \\
     &= g^\nu(c).
  \end{align*}
  
  \item Cycles of type (iv) in $Q_G$ are the ones of the form $d=\tilde c$, where $c\in\cal C({\rm i})$.
  If we write $c=\alpha_1 g^{t_1}(\alpha_2) g^{t_{1,2}}(\alpha_3) \cdots g^{t_{1,l-1}}(\alpha_l)$ for $\alpha_i$ of type (1) in $Q$, then $\tilde c = \tilde{\alpha}_1 \cdots \tilde{\alpha}_l$.
  Hence $\hat d = d$, and $\cal C_G({\rm iv})$ is the subset of all the cycles of this kind.
  If we put $\beta_i=\tilde\alpha_i$ for all $i$, then \mbox{$\tilde d^\nu = \tilde\beta_1^{\nu-b'_2} \tilde\beta_2^{\nu-b'_3} \tilde\beta_3^{\nu-b'_4} \cdots \tilde\beta_{l-1}^{\nu-b'_l} \tilde\beta_l^\nu$}.
  It follows that
  \begin{align*}
    \phi(\tilde d^\nu) &= \phi(\tilde\beta_1^{\nu-b'_2} \tilde\beta_2^{\nu-b'_3} \tilde\beta_3^{\nu-b'_4} \cdots \tilde\beta_{l-1}^{\nu-b'_l} \tilde\beta_l^\nu) = \\
    &= g^{\nu-b'_2-t(\alpha_1)}(\alpha_1)g^{\nu-b'_3-t(\alpha_2)}(\alpha_2) \cdots g^{\nu-t(\alpha_l)}(\alpha_l) = \\
    &= g^\nu(\alpha_1 g^{t_1}(\alpha_2) g^{t_{1,2}}(\alpha_3) \cdots g^{t_{1,l-1}}(\alpha_l)) = \\
    &= g^\nu(c).
  \end{align*}

\end{enumerate}

Now we can write $(W_G)_{\hat G}$ as follows:
\begin{align*}
  (W_G)_{\hat G} &= \sum_{d\in \mathcal C_G({\rm i})} a(d)\frac{|\hat G d|}{n}\tilde d + \sum_{d\in \mathcal C_G({\rm ii})}a(d) \sum_{\nu = 0}^{n-1}\zeta^{-p(d)\nu}\tilde d^{\nu} + \\
  &+ \sum_{d\in \mathcal C_G({\rm iii})}a(d)\sum_{\nu = 0}^{n-1}\tilde d^{\nu} + \sum_{d\in \mathcal C_G({\rm iv})}a(d)\sum_{\nu = 0}^{n-1}\tilde d^{\nu} = \\
  &= \sum_{c\in \mathcal C({\rm iv}),d=\tilde c^0} a(c)\tilde d + \sum_{c\in \mathcal C({\rm iii}),d=\tilde c^0}a(c) \sum_{\nu = 0}^{n-1}\zeta^{q(c)\nu}\tilde d^{\nu} + \\
  &+ \sum_{c\in \mathcal C({\rm ii}),d=\tilde c^0}a(c)\sum_{\nu = 0}^{n-1}\tilde d^{\nu} + \sum_{c\in \mathcal C({\rm i}),d=\tilde c}a(c)\frac{|Gc|}{n}\sum_{\nu = 0}^{n-1}\tilde d^{\nu}.
\end{align*}
Applying $\phi$ we get
\begin{align*}
  \phi((W_G)_{\hat G}) &= \sum_{c\in \mathcal C({\rm iv}),d=\tilde c^0} a(c)\phi(\tilde d) + \sum_{c\in \mathcal C({\rm iii}),d=\tilde c^0}a(c) \sum_{\nu = 0}^{n-1}\zeta^{q(c)\nu}\phi(\tilde d^{\nu}) + \\
  &+ \sum_{c\in \mathcal C({\rm ii}),d=\tilde c^0}a(c)\sum_{\nu = 0}^{n-1}\phi(\tilde d^{\nu}) + \sum_{c\in \mathcal C({\rm i}),d=\tilde c}a(c)\frac{|Gc|}{n}\sum_{\nu = 0}^{n-1}\phi(\tilde d^{\nu}) = \\
  &= \sum_{c\in \mathcal C({\rm iv}),d=\tilde c^0} a(c)c + \sum_{c\in \mathcal C({\rm iii}),d=\tilde c^0}a(c) \sum_{\nu = 0}^{n-1}\zeta^{q(c)\nu} \zeta^{-q(c)\nu} g^\nu(c) + \\
  &+ \sum_{c\in \mathcal C({\rm ii}),d=\tilde c^0}a(c)\sum_{\nu = 0}^{n-1} g^\nu(c) + \sum_{c\in \mathcal C({\rm i}),d=\tilde c}a(c)\frac{|Gc|}{n}\sum_{\nu = 0}^{n-1} g^\nu(c) = \\
  &= \sum_{c\in \mathcal C({\rm iv})} a(c)c + \sum_{\nu = 0}^{n-1} g^\nu\left( \sum_{c\in \mathcal C({\rm iii})} a(c)c + \sum_{c\in \mathcal C({\rm ii})} a(c)c + \sum_{c\in \mathcal C({\rm i})} a(c)\frac{|Gc|}{n} c \right) = \\
  &= W. \qedhere
\end{align*}
\end{proof}

\begin{cor}\label{cor:main2}
Let $\theta$ be the idempotent $\sum_{s\in \mathcal E_G} s\otimes 1$ in $(\eta( \Lambda G )\eta) \hat G$. Then the isomorphism of quivers $\phi:(Q_G)_{\hat G}\to Q$ induces an isomorphism of algebras 
  \begin{align*}
    \theta \left( \left(\eta\left( \Lambda G \right)\eta\right) \hat G\right)\theta \cong \Lambda,
  \end{align*}
 where $\Lambda = \mathcal P(Q, W)$.
\end{cor}
\begin{proof}
Applying Theorem~\ref{thm:main} to $\eta(\Lambda G)\eta$ with the action of $\hat G$, we get $$\theta \left( \left(\eta\left( \Lambda G \right)\eta\right) \hat G\right)\theta \cong \mathcal P( (Q_G)_{\hat G}, (W_G)_{\hat G}),$$ and the latter is isomorphic to $\cal P(Q,W)$ by Proposition~\ref{prop:main2b}.
\end{proof}

\section{Planar rotation-invariant QPs}\label{sec:nakayama}
Our main result Theorem \ref{thm:main} is about skew group algebras of Jacobian algebras of QPs, but it only applies under some assumptions on the group action. 
There is however a class of QPs which satisfy these assumptions, as well as a way of generating many examples in this class. 
To define this class, we follow \cite{HI11b} and associate a CW-complex to a QP called its canvas. First we need to fix some notation.

We denote by $D^d$ the $d$-disk and by $S^{d-1}=\partial D^d$ the $(d-1)$-sphere in $\R^d$. We suppose that $D^1=[0,1]$ and $S^0=\{0,1\}$.
A CW-complex is a topological space realized as a union $\bigcup_{d\in\Z_{\geq 0}}X^d$, where $X^0$ is a discrete space and each $X^d$ is obtained from $X^{d-1}$ in the following way.
For each $d$ there are a set $\{D^d_a\}_{a\in I_d}$ of copies of the $d$-disk and continuous maps $\phi_a\colon S^{d-1}_a=\partial D^d_a \to X^{d-1}$, such that we have a pushout diagram
$$
\xymatrix{
  \displaystyle\bigsqcup_{a\in I_d} S^{d-1}_a \ar^{(\phi_a)}[r] \arrow[d] & X^{d-1} \ar[d] \\
  \displaystyle\bigsqcup_{a\in I_d} D^d_a \ar^{(\epsilon_a)}[r] & X^d
}
$$
in the category of topological spaces with continuous maps (the left vertical map is given by the inclusions of $S^{d-1}_a$ as boundaries of $D^d_a$). For $d\geq1$ the image of the interior of $D^d_a$ under $\epsilon_a$ is called a $d$-cell. The elements of $X_0$ are called $0$-cells. We say that $X$ has dimension $m$ if $X=X^m$.

\begin{defin}[{\cite[Definition 8.1]{HI11b}}]
Let $(Q,W)$ be a QP and let $Q_2$ be a set of representatives modulo $\on{com}_Q$ of the cycles which appear in $W$. The \emph{canvas} of $(Q,W)$ is the $2$-dimensional CW-complex $X_{(Q,W)}$ defined in the following way. Its cells are indexed by the sets $X_0=Q_0$, $I_1=Q_1$, $I_2=Q_2$. For each $\alpha\in Q_1$ we have an attaching map $\phi_\alpha\colon S^0_\alpha\to X_0$ defined by $\phi_\alpha(0)=\mathfrak s(\alpha)$ and $\phi_\alpha(1)=\mathfrak t(\alpha)$.
If $c=\alpha_0\cdots\alpha_{l-1}\in Q_2$, we define the attaching map $\phi_c\colon S^1_c\to X_1$ by
$$\phi_c\left(\cos\left(\frac{2\pi}{l}(i+t)\right), \sin\left(\frac{2\pi}{l}(i+t)\right) \right)=\epsilon_{\alpha_i}(t)$$
for $i=0,\dots,l-1$ and $t\in [0,1)$. 
\end{defin}

\begin{rmk}
  In other (imprecise) words, the 1-skeleton of $X_{(Q, W)}$ is the underlying graph of $Q$, and we attach 2-cells along the cycles appearing in $W$.
\end{rmk}

\begin{defin}[{\cite[Definition 9.1]{HI11b}}]
  A QP $(Q, W)$ is \emph{planar} if it is simply connected and there exists an embedding of $X_{(Q, W)} $ into $\R^2$.
  We call it \emph{strongly planar} if it is planar and $X_{(Q,W)}$ is homeomorphic to a disk.
\end{defin}

If $(Q, W)$ is a planar QP, then by \cite[Proposition 9.3]{HI11b} the embedding of the quiver $Q$ in $\R^2$ determines the Jacobian algebra, so we can assume
that the coefficients in $W$ are +1 for the clockwise faces, and -1 for the anticlockwise faces.

\begin{defin}
  Let $(Q,W)$ be a planar QP and $G$ be a cyclic group acting on $Q$. 
  We say that $G$ \emph{acts on $(Q,W)$ by rotations} if:
  \begin{itemize}
    \item there is an embedding of $X_{(Q,W)}$ in $\R^2$ such that the action of a generator of $G$ is induced 
      by a rotation of the plane; 
    \item the action of $G$ is faithful;
    \item assumption \ref{ass:cycles} is satisfied.
  \end{itemize}
\end{defin}
Notice that in this case the image $\on{im}(G)\subseteq \on{Aut}(Q)$ is necessarily finite. For simplicity, we will identify $G$ with $\on{im}(G)$.

  We remark some facts which follow immediately from the definition, and directly imply that this class of quivers falls within the scope of Theorem \ref{thm:main}.

\begin{lemma}
  Let $G$ act on a planar QP $(Q,W)$ by rotations. Then the action of $G$ satisfies the assumptions \ref{ass:permuted}-\ref{ass:cycles}.
\end{lemma}

\begin{proof}
  A rotation permutes the vertices and maps arrows to arrows, so assumptions \ref{ass:permuted} and \ref{ass:arrows} are satisfied. By Remark~\ref{rmk:fixed}, we can assume that assumption~\ref{ass:fixed} is also satisfied.
  Since we are assuming that $G$ acts faithfully, we have that every vertex which is not fixed has order the order of a rotation generating $G$, hence assumption \ref{ass:cardinality} is satisfied. Assumption \ref{ass:potential} holds because $G$ maps faces of $X_{(Q, W)}$ to faces.
  Finally, assumption \ref{ass:cycles} holds by definition.
\end{proof}

There is a way of producing strongly planar QPs with a group acting by rotations by means of so-called Postnikov diagrams (see \cite{Pos06}, \cite{BKM16}, \cite{Pas17b}).
A Postnikov diagram is a collection of oriented curves in a disk subject to some axioms depending on two integer parameters $a,n\geq 1$, and it naturally gives rise to a planar QP.
For this result we need to assume that $\kk = \C$.
\begin{thm}[{\cite[Corollary 7.3]{Pas17b}}]\label{thm:postnikov}
  An $(a,n)$-Postnikov diagram is invariant under rotation by $\frac{2\pi a}{n}$ if and only if the corresponding QP is self-injective.
In this case, a Nakayama automorphism is given by this rotation.
\end{thm}
In particular, there is a finite cyclic group acting by rotations on a planar QP, so we can apply our construction.
The following result justifies the claim that Postnikov diagrams give rise to many examples. Namely, rotation-invariant Postnikov diagrams exist and in fact abound. 

\begin{thm}\cite{PTZ18}
 There exists an $(a,n)$-Postnikov diagram which is invariant under rotation by $\frac{2\pi a}{n}$ if and only if $a$ is congruent to -1, 0 or 1 modulo $n/\on{GCD}(n,a)$.
 In particular there are infinitely many self-injective planar QPs with Nakayama automorphism of order $d$, for any choice of $d$.
\end{thm}

\begin{rmk}
  There exist self-injective planar QPs with Nakayama automorphism acting by rotation which do not come from Postnikov diagrams. For instance, the quiver of the 3-preprojective algebra of 
  type ${\rm A}_n$ (see Example~\ref{ex:triangles}) with $n$ odd.
\end{rmk}

We conclude this section by observing that 
Theorem \ref{thm:main} can be naturally applied to any self-injective QP where the Nakayama automorphism satisfies our assumptions. In this case we get:

\begin{prop}
  Let $(Q, W)$ be a self-injective QP with Nakayama automorphism $\varphi$ of finite order. Call 
  $G = \langle \varphi\rangle\subseteq \on{Aut}(\mathcal P(Q, W))$, and assume that the assumptions \ref{ass:field}-\ref{ass:cycles} are satisfied. Then $\mathcal P(Q_G, W_G)$ is symmetric.
\end{prop}

\begin{proof}
  By Theorem \ref{thm:main}, $\mathcal P(Q_G, W_G)$ is a self-injective algebra which is Morita equivalent to $\Lambda G$. The latter is symmetric by Corollary \ref{cor:sym} using Lemma 
  \ref{lem:basicsym}.
\end{proof}

Combining this with our previous discussion, we remark that by Theorem \ref{thm:postnikov} there is a symmetric Jacobian algebra associated to every rotation-invariant Postnikov diagram.

\begin{cor}\label{prop:combining}
  If $(Q, W)$ is a self-injective QP coming from a Postnikov diagram with Nakayama automorphism $\varphi$, then $\mathcal P(Q_{\langle \varphi\rangle}, W_{\langle \varphi \rangle})$ is symmetric.
\end{cor}

These results are illustrated in Example \ref{ex:postnikov}.

\section{Cuts and 2-representation finite algebras}\label{sec:cuts}

In this section we apply our construction to the study of 2-representation finite algebras. These are by definition algebras of global dimension at most 2 admitting a 
cluster tilting module, and were introduced by Iyama as a natural generalisation of hereditary representation finite algebras. We refer the interested reader to \cite{Iya08}, \cite{JK16} 
for general higher Auslander-Reiten theory, and to \cite{HI11b} for the $2$-dimensional case. For the general interaction between higher representation finiteness and skew group algebras, 
see also \cite{LM18}.

Let $(Q,W)$ be a QP. For a subset $C\subseteq Q_1$ we can define a grading $d_C$ on $Q$ by setting
$$
d_C(\alpha) = 
\begin{cases}
  1, \text{ if } \alpha\in C;\\
  0, \text{ otherwise.}
\end{cases}
$$
\begin{defin}
A subset $C\subseteq Q_1$ is called a \emph{cut} if $W$ is homogeneous of degree $1$ with respect to $d_C$.
\end{defin}

Note that a cut induces a grading on the Jacobian algebra $\cal P(Q,W)$. We call its degree $0$ part a \emph{truncated Jacobian algebra} and denote it by $\cal P(Q,W)_C$.

Our interest in truncated Jacobian algebras lies in the following result.

\begin{thm}[{\cite[Theorem 3.11]{HI11b}}]\label{thm:HI}
If $(Q,W)$ is a self-injective QP and $C$ is a cut, then $\cal P(Q,W)_C$ is $2$-representation finite. Moreover, every basic $2$-representation finite algebra is obtained in this way.
\end{thm}

Now assume that a finite cyclic group $G$ acts on $\cal P(Q, W)$ satisfying the assumptions \ref{ass:field}-\ref{ass:cycles}. 
We want to understand when a cut in $(Q_G,W_G)$ can be induced from one in $(Q,W)$.
We call a cut in $(Q, W)$ invariant under the $*$ action of $G$ a \emph{$G$-invariant cut}. 

\begin{prop}
  \label{prop:cut}
Let $C$ be a $G$-invariant cut in $(Q,W)$. Then the subset $C_G=C_1\cup C_2\cup C_3\cup C_4$ of $(Q_G)_1$ defined by
$$C_1 = \{\tilde\alpha \,|\, \alpha\in C \text{ of type }(1)\},\quad C_x = \{\tilde\alpha^\mu \,|\, \alpha\in C \text{ of type ($x$), } 0\leq\mu\leq n-1\},\; x=2,3,4,$$
is a cut in $(Q_G,W_G)$.
\end{prop}
\begin{proof}
In order to show that $C_G$ is a cut in $(Q_G,W_G)$, we shall prove that every cycle in $W_G$ has degree $1$ with respect to $d_{C_G}$. Thus we have four different cases to consider.
\begin{enumerate}[label=(\roman*)]
  \item Let $c\in\cal C({\rm i})$, so $c=\alpha_1 g^{t_1}(\alpha_2) \cdots g^{t_1+\dots+t_{l-1}}(\alpha_l)$ for some arrows $\alpha_i\in Q_1$ of type (1). Then $W_G$ contains the cycle $\tilde c=\tilde\alpha_1\cdots\tilde\alpha_l$ and, since $C$ is $G$-invariant, we have
  $$d_{C_G}(\tilde c) = \sum_{i=1}^l d_{C_G}(\tilde\alpha_i) = \sum_{i=1}^l d_C(\alpha_i) = \sum_{i=1}^l d_C(g^{t_1+\dots+t_{i-1}}(\alpha_i)) = d_C(c) = 1.$$
  \item Let $c\in\cal C({\rm ii})$, so $c=\alpha_1 \alpha_2 g^{t_2}(\alpha_3) \cdots g^{t_2+\dots+t_{l-1}}(\alpha_l)$ for $\alpha_1$ of type (3), $\alpha_2$ of type (2) and $\alpha_3,\dots,\alpha_l$ of type (1). For each $\mu=0,\dots,n-1$ we have a cycle $\tilde c^\mu = \tilde\alpha_1^\mu \widetilde{g^{-t_2}(\alpha_2)}^\mu\tilde\alpha_3\cdots \tilde\alpha_l$ in $W_G$ and
  $$d_{C_G}(\tilde c^\mu) = d_{C_G}(\tilde\alpha_1^\mu) + d_{C_G}(\widetilde{g^{-t_2}(\alpha_2)}^\mu) + \sum_{i=3}^l d_{C_G}(\tilde\alpha_i) = \sum_{i=1}^l d_C(\alpha_i) = \sum_{i=1}^l d_C(g^{t_1+\dots+t_{i-1}}(\alpha_i)) = d_C(c) = 1.$$
\item Let $c\in\cal C({\rm iii})$, so $c=\alpha_1 \alpha_2 \cdots \alpha_l$ for $\alpha_1$ of type (2), $\alpha_2$ of type (3) and $\alpha_3,\dots,\alpha_l$ of type (4). For each $\mu=0,\dots,n-1$ we have a cycle $\tilde c^\mu = \tilde{\alpha}_1^{\mu} \tilde{\alpha}_2^{\mu-b_3} \cdots \tilde{\alpha}_{l-1}^{\mu-b_l} \tilde{\alpha}_l^\mu$ in $W_G$, where \mbox{$b_i=b(\alpha_i)+\dots+b(\alpha_l)$}. Hence
  $$d_{C_G}(\tilde c^\mu) = d_{C_G}(\tilde\alpha_1^\mu) + \sum_{i=2}^l d_{C_G}(\tilde\alpha_i^{\mu-b_{i+1}}) = \sum_{i=1}^l d_C(\alpha_i) = d_C(c) = 1.$$
  \item Let $c\in\cal C({\rm iv})$, so $c=\alpha_1 \alpha_2 \cdots \alpha_l$ for $\alpha_i$ of type (4). For each $\mu=0,\dots,n-1$ we have a cycle $\tilde c^\mu = \tilde{\alpha}_1^{\mu-b_2} \tilde{\alpha}_2^{\mu-b_3} \cdots \tilde{\alpha}_{l-1}^{\mu-b_l} \tilde{\alpha}_l^\mu$ in $W_G$, where $b_i=b(\alpha_i)+\dots+b(\alpha_l)$. Hence
    \begin{equation*}
      d_{C_G}(\tilde c^\mu) = \sum_{i=1}^l d_{C_G}(\tilde\alpha_i^{\mu-b_{i+1}}) = \sum_{i=1}^l d_C(\alpha_i) = d_C(c) = 1.\qedhere
    \end{equation*}
\end{enumerate}
\end{proof}

Observe that from \cite[Corollary~1.6(1)]{LM18}, 2-representation finiteness is preserved by taking skew group algebras. Thus it follows from Theorem \ref{thm:HI} that the property of being a truncated
Jacobian algebra is also preserved. In our setting, the corresponding cut on $(Q_G, W_G)$ is precisely $C_G$:

\begin{prop}\label{prop:cut_truncated}
Let $C$ be a $G$-invariant cut in $(Q,W)$ and let $C_G$ be the cut constructed in Proposition~\ref{prop:cut}. Then the action of $G$ on $\cal P(Q,W)$ restricts to an action on $\cal P(Q,W)_C$, and the skew group algebra $(\cal P(Q,W)_C) G$ is Morita equivalent to $\cal P(Q_G,W_G)_{C_G}$.
\end{prop}
\begin{proof}
  Call $\Lambda=\cal P(Q,W)$ and let $\Lambda_0$ be its degree $0$ part with respect to the grading $d_C$, so \mbox{$\Lambda_0=\cal P(Q,W)_C$}. The fact that $C$ is $G$-invariant implies that $G$ preserves the grading, so the first assertion holds.

Now note that we can define a grading on $\Lambda G$ by assigning degree $d_C(x)$ to $x\otimes h$ for all $h\in G$ and all homogeneous elements $x\in\Lambda$. Moreover this induces a grading on $\eta(\Lambda G)\eta$ and we have that $(\eta(\Lambda G) \eta)_0 = \eta(\Lambda_0 G)\eta$.
Hence, in order to prove the claim, it is enough to show that the grading on $\eta(\Lambda G)\eta$ coincides with the grading $d_{C_G}$ on $\cal P(Q_G,W_G)$ under the isomorphism $\eta(\Lambda G)\eta \cong \cal P(Q_G,W_G)$. But this follows immediately from the definition of $C_G$, since both algebras are generated in degree $0$ and $1$ and the elements of degree $1$ in $\eta(\Lambda G)\eta$ are exactly the ones given by $C_G$.
\end{proof}

%\begin{rmk}
  %The statement that $(\cal P(Q,W)_C) G$ is Morita equivalent to a truncated Jacobian algebra follows directly from \cite[Corollary~1.6(1)]{LM18}.
%\end{rmk}

Let $(Q,W)$ be a self-injective QP with a group $G$ acting as per the assumptions \ref{ass:field}-\ref{ass:cycles}. Then $(Q_G, W_G)$ is self-injective, so its truncated Jacobian algebras are 2-representation finite.
In the spirit of \cite[\S7]{HI11b}, we will give sufficient conditions on $(Q, W)$ for the truncated Jacobian algebras of $(Q_G, W_G)$ to be derived equivalent to each other.

In the following discussion we do not need to assume self-injectivity.

\begin{defin}
  We say that $(Q, W)$ \emph{has enough cuts} if every arrow of $Q$ is contained in a cut.
  We say that $(Q, W)$ \emph{has enough $G$-invariant cuts} if every arrow of $Q$ is
  contained in a $G$-invariant cut (cf.~\cite[Definition 7.4]{HI11b}). 
\end{defin}

\begin{lemma}
  \label{lem:enough}
  If $(Q,W)$ has enough $G$-invariant cuts, then $(Q_G, W_G)$ has enough cuts.
\end{lemma}

\begin{proof}
  Let $\beta\in (Q_G)_1$, so $\beta=\tilde\alpha$ or $\beta=\tilde\alpha^\mu$ for some $\alpha\in Q_1$. Let $C$ be a $G$-invariant cut in $(Q,W)$ containing $\alpha$, then the cut $C_G$ in $(Q_G, W_G)$ constructed in Proposition \ref{prop:cut} contains $\beta$.
\end{proof}

To use the results of \cite{HI11b}, we need to study the topology of the canvas of $(Q_G, W_G)$. 
We will do this in the case of $G$ acting by rotations on a strongly planar QP.

\begin{prop}
  \label{prop:simcon}
  Let $(Q, W)$ be a strongly planar QP with a group $G$ acting by rotations, and assume that there is a vertex of $Q$ fixed by $G$. Then $X_{(Q_G, W_G)}$ is simply connected.
\end{prop}

\begin{proof}
   Let us decompose $X_{(Q, W)}  = \mathcal U\cup \mathcal V$, where $\mathcal V$ is the subcomplex consisting of all the faces adjacent to the central vertex, and $\mathcal U$ 
  is the subcomplex consisting of the other faces.
  Since $(Q, W)$ is strongly planar, $X_{(Q, W)}$ is homeomorphic to a disk. 
Note that if $G$ is trivial, then the statement is immediate. Otherwise, this implies that the central vertex $\Omega$ has a neighbourhood in $X_{(Q, W)}$ which is itself homeomorphic to a disk. So $\mathcal V$ is homeomorphic to a disk as well.
  Thus $\mathcal V$ looks as in Figure \ref{fig:V}, where
  $\alpha_i, \beta_i$ are arrows, $\gamma_i, \delta_i$ are paths, and all cycles $\alpha_i\gamma_i\beta_i$, $\alpha_{i+1}\delta_i\beta_i$, and 
  $\alpha_1\delta_l\beta_l$ bound faces. The action of a generator $g$ of $G$ is given by adding $a$ to indices.
  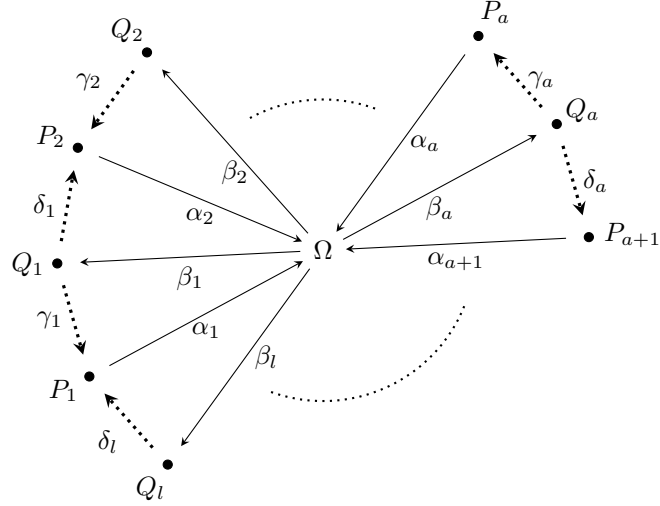
\begin{figure}[h]
\[
\begin{tikzpicture}[baseline=(bb.base),scale = 1,
    quivarrow/.style ={black, ->, >=stealth, shorten <= 3mm , shorten >= 3mm},
    cutarrow/.style = {black, -latex,very thick, dotted},
    patharrow/.style = {black,->,>= stealth, dotted, very thick, shorten <= 3mm , shorten >= 3mm}
  ]

\newcommand{\bstart}{260} %starting angle
\newcommand{\nth}{360/14} 
\newcommand{\radius}{3.5cm} % radius of boundary circle
\newcommand{\eps}{11pt} % offset for boundary labels
\newcommand{\dotrad}{0.06cm} % size of node

% ------ baseline ------
\path (0,0) node (bb) {};

% ------------------ node labels
\foreach \n in {1,2}
{ \coordinate (p\n) at (\bstart-\nth*\n*2:\radius);
  \draw (\bstart-\nth*\n*2:\radius+\eps) node {$P_\n$}; }
\foreach \n in {1,2}
{ \coordinate (q\n) at (\bstart-\nth*\n*2-\nth:\radius);
  \draw (\bstart-\nth*\n*2-\nth:\radius+\eps) node {$Q_\n$}; }
 \coordinate (p3) at (\bstart-4*2*\nth:\radius);
  \draw (\bstart-\nth*2*4:\radius+\eps) node {$P_a$}; 
 \coordinate (q3) at (\bstart-4*2*\nth-\nth:\radius);
  \draw (\bstart-\nth*2*4-\nth:\radius+\eps) node {$Q_a$}; 
 \coordinate (p4) at (\bstart-5*2*\nth:\radius);
 \draw (\bstart-\nth*2*5:\radius+1.5*\eps) node {$P_{a+1}$}; 
 \coordinate (q4) at (\bstart-\nth:\radius);
 \draw (\bstart-\nth:\radius+\eps) node {$Q_{l}$}; 
 \coordinate (o) at (0,0);
 \draw (0,0) node {$\Omega$};
  
\foreach \n in {1,2,3,4} {\draw (p\n) circle(\dotrad) [fill=black];
\draw (q\n) circle(\dotrad) [fill=black];
}

%---------------------arrows

\draw [patharrow] (q1) to node[midway,left]{$\gamma_1$} (p1);
\draw [patharrow] (q1) to node[midway,left]{$\delta_1$} (p2);
\draw [patharrow] (q2) to node[midway,above left]{$\gamma_2$} (p2);
\draw [patharrow] (q4) to node[midway,below left]{$\delta_l$} (p1);
\draw [patharrow] (q3) to node[midway,right]{$\delta_a$} (p4);
\draw [patharrow] (q3) to node[midway,right]{$\gamma_a$} (p3);
\draw [quivarrow] (p1) to node[midway,below]{$\alpha_1$} (o);
\draw [quivarrow] (p2) to node[midway,below]{$\alpha_2$} (o);
\draw [quivarrow] (p3) to node[midway,right]{$\alpha_a$} (o);
\draw [quivarrow] (p4) to node[midway,below]{$\alpha_{a+1}$} (o);
\draw [quivarrow] (o) to node[midway,below]{$\beta_1$} (q1);
\draw [quivarrow] (o) to node[midway,below]{$\beta_2$} (q2);
\draw [quivarrow] (o) to node[midway,below]{$\beta_a$} (q3);
\draw [quivarrow] (o) to node[midway,right]{$\beta_l$} (q4);

%------------------dots

\draw [black,thick, dotted,domain=70:120] plot ({2*cos(\x)}, {2*sin(\x)});
\draw [black,thick, dotted,domain=250:340] plot ({2*cos(\x)}, {2*sin(\x)});
 \end{tikzpicture}
\]
\caption{The subcomplex $\mathcal V$ of $X_{(Q, W)}$.}
\label{fig:V}
\end{figure}
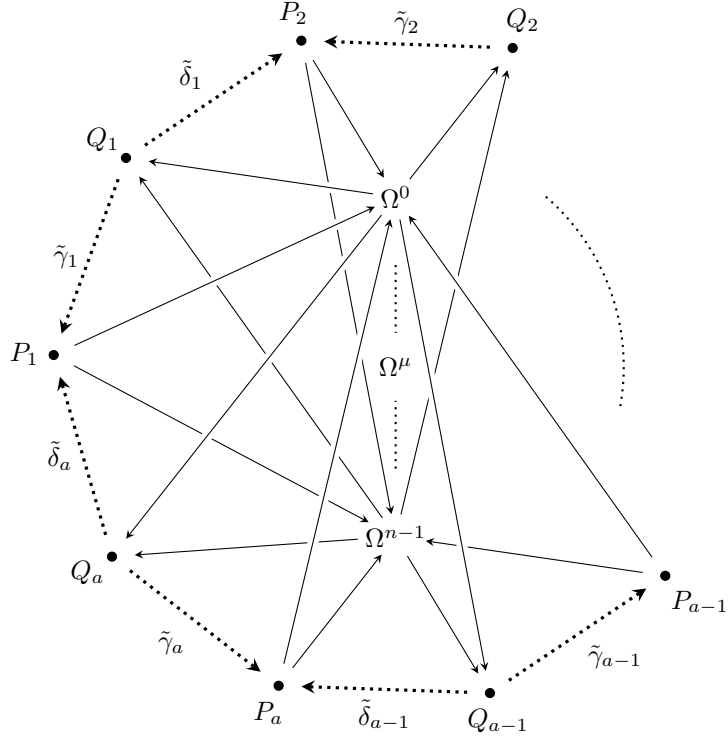
\begin{figure}[h]
\[
\begin{tikzpicture}[baseline=(bb.base),
      quivarrow/.style ={black, ->, >=stealth, shorten <= 3mm , shorten >= 3mm},
    cutarrow/.style = {black, -latex,very thick, dotted},
    patharrow/.style = {black,->,>= stealth, dotted, very thick, shorten <= 3mm , shorten >= 3mm}
  ]

\newcommand{\bstart}{250} %starting angle
\newcommand{\nth}{360/10} 
\newcommand{\radius}{4.5cm} % radius of boundary circle
\newcommand{\eps}{11pt} % offset for boundary labels
\newcommand{\dotrad}{0.06cm} % size of node

% ------ baseline ------
\path (0,0) node (bb) {};

% ------------------ node labels
\foreach \n in {1,2}
{ \coordinate (p\n) at (\bstart-\nth*\n*2:\radius);
  \draw (\bstart-\nth*\n*2:\radius+\eps) node {$P_\n$}; }
\foreach \n in {1,2}
{ \coordinate (q\n) at (\bstart-\nth*\n*2-\nth:\radius);
  \draw (\bstart-\nth*\n*2-\nth:\radius+\eps) node {$Q_\n$}; }
 \coordinate (p3) at (\bstart-4*2*\nth:\radius);
 \draw (\bstart-\nth*2*4:\radius+1.5*\eps) node {$P_{a-1}$}; 
 \coordinate (q3) at (\bstart-4*2*\nth-\nth:\radius);
 \draw (\bstart-\nth*2*4-\nth:\radius+\eps) node {$Q_{a-1}$}; 
 \coordinate (p4) at (\bstart-5*2*\nth:\radius);
 \draw (\bstart-\nth*2*5:\radius+\eps) node {$P_{a}$}; 
 \coordinate (q4) at (\bstart-\nth:\radius);
 \draw (\bstart-\nth:\radius+\eps) node {$Q_{a}$};

 \coordinate (o1) at (0,\radius*.5);
 \draw (0,\radius*.5) node {$\Omega^0$};
 \coordinate (o2) at (0,\radius*-.5);
 \draw (0,\radius*-.5) node {$\Omega^{n-1}$};
 \coordinate (o3) at (0,0);
 \draw (0,0) node {$\Omega^\mu$};
  
\foreach \n in {1,2,3,4} {\draw (p\n) circle(\dotrad) [fill=black];
\draw (q\n) circle(\dotrad) [fill=black];
}

%---------------------arrows
\draw [patharrow] (q1) to node[midway,left]{$\tilde\gamma_1$} (p1);
\draw [patharrow] (q1) to node[midway,above left]{$\tilde\delta_1$} (p2);
\draw [patharrow] (q2) to node[midway,above]{$\tilde\gamma_2$} (p2);
\draw [patharrow] (q4) to node[midway,left]{$\tilde\delta_a$} (p1);
\draw [patharrow] (q3) to node[midway,below]{$\tilde\delta_{a-1}$} (p4);
\draw [patharrow] (q3) to node[midway,below right]{$\tilde\gamma_{a-1}$} (p3);
\draw [patharrow] (q4) to node[midway,below left]{$\tilde\gamma_a$} (p4);

\draw [quivarrow, shorten <= 3mm , shorten >= 4mm] (p1) to node[midway,below]{} (o2);
\draw [quivarrow] (p2) to node[midway,below]{} (o2);
\draw [quivarrow, shorten <= 3mm , shorten >= 4mm] (p3) to node[midway,right]{} (o2);
\draw [quivarrow] (p4) to node[midway,below]{} (o2);
\draw [quivarrow] (o2) to node[midway,below]{} (q1);
\draw [quivarrow] (o2) to node[midway,below]{} (q2);
\draw [quivarrow] (o2) to node[midway,below]{} (q3);
\draw [quivarrow, shorten <= 5mm , shorten >= 3mm] (o2) to node[midway,right]{} (q4);

\draw (p1) edge [-,line width=4pt,draw=white, shorten >= 2mm, shorten <= 2mm] (o1);
\draw (p2) edge [-,line width=4pt,draw=white, shorten >= 2mm, shorten <= 2mm] (o1);
\draw (p3) edge [-,line width=4pt,draw=white, shorten >= 2mm, shorten <= 2mm] (o1);
\draw (p4) edge [-,line width=4pt,draw=white, shorten >= 2mm, shorten <= 2mm] (o1);
\draw (o1) edge [-,line width=4pt,draw=white, shorten >= 2mm, shorten <= 2mm] (q1);
\draw (o1) edge [-,line width=4pt,draw=white, shorten >= 2mm, shorten <= 2mm] (q2);
\draw (o1) edge [-,line width=4pt,draw=white, shorten >= 2mm, shorten <= 2mm] (q3);
\draw (o1) edge [-,line width=4pt,draw=white, shorten >= 2mm, shorten <= 2mm] (q4);

\draw [quivarrow] (p1) to node[midway,below]{} (o1);
\draw [quivarrow] (p2) to node[midway,below]{} (o1);
\draw [quivarrow] (p3) to node[midway,right]{} (o1);
\draw [quivarrow] (p4) to node[midway,below]{} (o1);
\draw [quivarrow] (o1) to node[midway,below]{} (q1);
\draw [quivarrow] (o1) to node[midway,below]{} (q2);
\draw [quivarrow] (o1) to node[midway,below]{} (q3);
\draw [quivarrow] (o1) to node[midway,right]{} (q4);

%------------------dots

\draw [black, thick, dotted] (0,\radius*.3) edge (0,\radius*.1);
\draw [black, thick, dotted] (0,\radius*-.3) edge (0,\radius*-.1);

\draw [black,thick, dotted,domain=-10:50] plot ({3*cos(\x)}, {3*sin(\x)});

 \end{tikzpicture}
\]
\caption{The subcomplex $\tilde{\mathcal V}$ of $X_{(Q_G, W_G)}$.}
\label{fig:tildeV}
\end{figure}
  By picking $g$ suitably, we can assume that $an = l$, where $n = |G|$. We choose as representatives of vertices a set
  $\mathcal E$ which contains $\left\{ \Omega, P_1, \dots, P_a, Q_1, \dots, Q_a \right\}.$
  Observe that $G$ acts freely on $\mathcal U$, and it also acts freely on $\mathcal U\cap \mathcal V$.
  Then $X_{(Q_G, W_G)} = \tilde{ \mathcal U}\cup \tilde{ \mathcal V}$, where $\tilde{\mathcal V}$ is as in Figure \ref{fig:tildeV}
  and $\tilde{\mathcal U}\cong \mathcal U/G$ is the quotient space of $\mathcal U$ by $G$, by our construction of $Q_G$.
  In the picture we denote by $\tilde\delta_i$ the product of $\tilde d$, where $d$ is an arrow of $\delta_i$, and similarly for $\tilde \gamma_i$.
  We have that $\tilde{\mathcal U}$ is attached to $\tilde{\mathcal V}$ along $(\mathcal U\cap \mathcal V)/G = \tilde{\mathcal U}\cap \tilde{\mathcal V} $.
  Now observe that since $X_{(Q, W)}$ is simply connected, it must retract to $\mathcal V$. In particular there is a deformation retraction $F$ between $\mathcal U$ and 
  $\mathcal U\cap \mathcal V$. We choose $F$ such that it commutes with the action of $G$ on $\mathcal U$. Then there is an induced deformation retraction 
  $\tilde F$ between $\tilde{\mathcal U}$ and $\tilde{\mathcal U}\cap \tilde{\mathcal V}$.
  In particular $X_{(Q_G, W_G)} $ retracts to $\tilde{\mathcal V}$, so they have the same homotopy type.

  We need to describe the faces of $\tilde{\mathcal V}$. Let us look at the set of cycles in $W$ involving only vertices in $\mathcal V$. 
  These are $\gamma_1\beta_1\alpha_1, \dots, \gamma_a\beta_a\alpha_a, \delta_1\beta_1\alpha_2, \dots, \delta_a\beta_a \alpha_{a+1}$ and their orbits. These cycles are all of type (ii), so we 
  have
  \begin{align*}
    \widetilde{\gamma_i\beta_i\alpha_i}^{\mu} &= \tilde{\gamma}_i\tilde{\beta}_i^\mu\tilde{\alpha}_i^\mu\\
    \widetilde{\delta_i\beta_i\alpha_{i+1}}^\mu &= \tilde{\delta}_i\tilde{\beta}_i^\mu\tilde{\alpha}_{i+1}^\mu,
  \end{align*}
  for $i= 1, \dots, a$, with the notation $\tilde{\alpha}_{a+1}^\mu = \tilde{\alpha}_1^\mu$.
  Now fix $\mu\in \left\{ 0, \dots, n-1 \right\}$. Then $\Omega^\mu$ is contained in every $\tilde{\gamma}_i\tilde{\beta}_i^\mu\tilde{\alpha}_i^\mu$, in every $ \tilde{\delta}_i\tilde{\beta}_i^\mu\tilde{\alpha}_{i+1}^\mu$,
  and no other cycle in $W_G$.
  The subcomplex consisting of the faces corresponding to these $2a$ cycles is a disk with center $\Omega^\mu$. Thus $\tilde{\mathcal V}$ consists of $n$ disks glued along their boundary
  $\tilde\delta_a\cdots\tilde\gamma^{-1}_2\tilde\delta_1\tilde\gamma^{-1}_1$, and therefore has the 
  homotopy type of a bouquet of spheres. In particular it is simply connected, which concludes the proof.
\end{proof}

In Example \ref{ex:postnikov} we proceed as in the proof of Proposition \ref{prop:simcon} to determine the canvas of $(Q_G, W_G)$. 

\begin{rmk}
  If $G$ acts on a planar QP $(Q, W)$ by rotations and $(Q, W)$ has a $G$-invariant cut, then $Q$ must have a central vertex. Indeed, $Q$ has either a central vertex or a central cycle, but 
  on a central cycle one cannot choose exactly one arrow in a way which is invariant under rotations.
\end{rmk}

In the self-injective case we have the following result.

\begin{thm}
  Let $(Q, W)$ be a strongly planar self-injective QP, with a group $G$ acting by rotations and enough $G$-invariant cuts. 
  Then all the truncated Jacobian algebras of $(Q_G, W_G)$ are derived equivalent to each other.
\end{thm}

\begin{proof}
  By Lemma \ref{lem:enough}, $(Q_G, W_G)$ has enough cuts. By Proposition \ref{prop:simcon}, $X_{(Q_G, W_G)}$ is simply connected. Then we conclude by \cite[Theorem 8.7]{HI11b}.
\end{proof}

In particular, note that this result applies to QPs coming from Postnikov diagrams, provided they have enough $G$-invariant cuts. It should be noted that we know of no examples of a self-injective QP with a cut 
that does not have enough cuts, nor of a self-injective QP with a $G$-invariant cut that does not have enough $G$-invariant cuts.

\section{Examples}\label{sec:ex}

\tikzset{>=stealth}
\tikzset{cutarrow/.style={densely dashed}}

In this section we will illustrate our construction with some examples. For simplicity we will assume that $\kk=\C$, so the assumption \ref{ass:field} will be always satisfied.

\subsection{Examples from planar rotation-invariant QPs}

As we have seen in Section~\ref{sec:nakayama}, many examples where our construction may be applied are given by quivers embedded in the plane with a group acting by rotations. Let us illustrate some of them.

\begin{ex}[$2$-representation finite algebras of type A]\label{ex:triangles}

A family of examples of self-injective planar QPs is given by $3$-preprojective algebras of $2$-representation finite algebras of type A, which were introduced in \cite{IO11} and are defined as follows.

Let $s\geq1$ and $Q=Q^{(s)}$ be the quiver defined by
$$Q_0 = \{(x_1,x_2,x_3)\in \Z_{\geq0}^3 \ |\ x_1+x_2+x_3=s-1\},$$
$$Q_1 = \{\alpha_i\colon x \to x+f_i \ |\  1\leq i\leq3,\, x,x+f_i\in Q_0 \},$$
where $f_1=(-1,1,0)$, $f_2=(0,-1,1)$, $f_3=(1,0,-1)$. The potential $W$ is given by the sum of all cycles of the form $\alpha_1\alpha_2\alpha_3$ minus the ones of the form $\alpha_1\alpha_3\alpha_2$.

The Nakayama automorphism of $\Lambda=\cal P(Q,W)$ is induced by the unique automorphism of $Q$ given on vertices by $(x_1,x_2,x_3)\mapsto(x_3,x_1,x_2)$. Then the group $G$ generated by it acts on $Q$ by an anticlockwise rotation by $2\pi/3$.
We may note that this action has a (unique) fixed vertex if and only if $s\equiv 1\pmod 3$. In that case the vertex $(\frac{s-1}{3},\frac{s-1}{3},\frac{s-1}{3})$ is fixed.

\begin{prop}
If $s\equiv 1\pmod 3$, then $Q^{(s)}$ has enough $G$-invariant cuts.
\end{prop}
\begin{proof}
Call $x_0=(\frac{s-1}{3},\frac{s-1}{3},\frac{s-1}{3})$ the unique fixed vertex.
Let $L=\{(x_1,x_2,x_3)\in \Z^3 \,|\, x_1+x_2+x_3=0\}$ and note that it is a free abelian group of rank $2$ with basis $\{f_1,f_2\}$.
We may embed $Q_0$ in $L$ via the map $x\mapsto x-x_0$. Note that the action of $G$ on $Q_0$ can be naturally extended to an action on $L$, which is again given by $(x_1,x_2,x_3)\mapsto(x_3,x_1,x_2)$.

Let $\omega\colon L\to\Z/3\Z$ be the group homomorphism defined by $\omega(f_i)=1$ for $i=1,2,3$. For each $j\in\Z/3\Z$ we define the following subset of $Q_1$:
$$C_j = \{\alpha_i\colon x\to x+f_i \,|\, \omega(x-x_0)=j\}.$$
Then $C_j$ is a cut (cf.~\cite[Example~5.8]{HIO14}). It is symmetric because $\omega$ is invariant on $G$-orbits.
Moreover every arrow is contained in a cut of this type, so the statement follows.
\end{proof}

As an example, we illustrate the cut $C_0$ of $Q^{(7)}$ in Figure~\ref{fig:triangle_cut}.

\begin{figure}[H]
\centering
\begin{tikzpicture}[->, scale=1.1, baseline=2cm, inner sep=0.9, shape=circle]
\begin{scriptsize}
\begin{scope}[rotate=-15]

\pgfmathtruncatemacro{\s}{6}  % \s is s+1 in the text

\foreach \x in {0,...,\s}
{
  \pgfmathtruncatemacro{\maxy}{\s-\x-1}
  \ifnum \maxy>-1 {
  \foreach \y in {0,...,\maxy}
  {
    \pgfmathtruncatemacro{\z}{\s-\y-\x}
    \pgfmathtruncatemacro{\cut}{Mod(2*\x+\y,3)}
    
    \node (\x \z \y) at (\x,\z,\y) {$\x\z\y$};
    \pgfmathtruncatemacro{\zm}{\z-1}
    \pgfmathtruncatemacro{\yp}{\y+1}
    \pgfmathtruncatemacro{\zx}{\z-\x}
    \node (\x \zm \yp) at (\x,\zm,\yp) {$\x\zm\yp$};
    \ifnum \cut=0 \draw[cutarrow] (\x \z \y) to (\x \zm \yp);
	  [\else \draw (\x \z \y) to (\x \zm \yp);] \fi
    
    \node (\y \x \z) at (\y,\x,\z) {$\y\x\z$};
    \pgfmathtruncatemacro{\zm}{\z-1}
    \pgfmathtruncatemacro{\yp}{\y+1}
    \pgfmathtruncatemacro{\zy}{\z-\y}
    \node (\yp \x \zm) at (\yp,\x,\zm) {$\yp\x\zm$};
    \ifnum \cut=0 \draw[cutarrow] (\y \x \z) to (\yp \x \zm);
	  [\else \draw (\y \x \z) to (\yp \x \zm);] \fi
    
    \node (\z \y \x) at (\z,\y,\x) {$\z\y\x$};
    \pgfmathtruncatemacro{\zm}{\z-1}
    \pgfmathtruncatemacro{\yp}{\y+1}
    \node (\zm \yp \x) at (\zm,\yp,\x) {$\zm\yp\x$};
    \ifnum \cut=0 \draw[cutarrow] (\z \y \x) to (\zm \yp \x);
	  [\else \draw (\z \y \x) to (\zm \yp \x);] \fi
  }
  } \fi;
}
\end{scope}
\end{scriptsize}
\end{tikzpicture}
\caption{The quiver $Q^{(7)}$. The cut $C_0$ is given by the dashed arrows.} \label{fig:triangle_cut}
\end{figure}
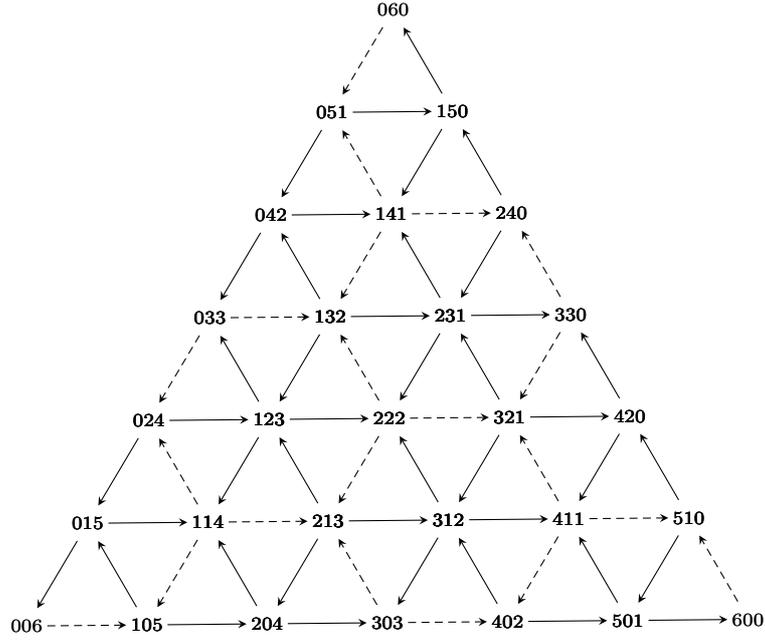

Now we will describe our skew group algebra construction for the quiver $Q=Q^{(4)}$ (which is depicted in Figure~\ref{fig:triangle}).

\begin{figure}[H]
\centering
\begin{tikzpicture}[->, scale=1.1, baseline=2cm, inner sep=0.9, shape=circle]
\begin{scriptsize}
\begin{scope}[rotate=-15]

\pgfmathtruncatemacro{\s}{3}  % \s is s+1 in the text

\foreach \x in {0,...,\s}
{
  \pgfmathtruncatemacro{\maxy}{\s-\x-1}
  \ifnum \maxy>-1 {
  \foreach \y in {0,...,\maxy}
  {
    \pgfmathtruncatemacro{\z}{\s-\y-\x}
    
    \node (\x \z \y) at (\x,\z,\y) {$\x\z\y$};
    \pgfmathtruncatemacro{\zm}{\z-1}
    \pgfmathtruncatemacro{\yp}{\y+1}
    \pgfmathtruncatemacro{\zx}{\z-\x}
    \node (\x \zm \yp) at (\x,\zm,\yp) {$\x\zm\yp$};
    \ifnum \y=2 { \draw (\x \z \y) to node[midway,left] {$\beta$} (\x \zm \yp);}\fi;
    \ifnum \y=\z { \draw (\x \z \y) to node[midway,right] {$\lambda$} (\x \zm \yp);}\fi;
    \ifnum \zx=2 { \draw (\x \z \y) to node[midway,left] {$\delta$} (\x \zm \yp);}\fi;
    \draw (\x \z \y) to (\x \zm \yp);
    
    \node (\y \x \z) at (\y,\x,\z) {$\y\x\z$};
    \pgfmathtruncatemacro{\zm}{\z-1}
    \pgfmathtruncatemacro{\yp}{\y+1}
    \pgfmathtruncatemacro{\zy}{\z-\y}
    \node (\yp \x \zm) at (\yp,\x,\zm) {$\yp\x\zm$};
    \ifnum \z=3 { \draw (\y \x \z) to node[midway,below] {$\alpha$} (\yp \x \zm);}\fi;
    \ifnum \zy=2 { \draw (\y \x \z) to node[midway,above] {$\theta$} (\yp \x \zm);}\fi;
    \draw (\y \x \z) to (\yp \x \zm);
    
    \node (\z \y \x) at (\z,\y,\x) {$\z\y\x$};
    \pgfmathtruncatemacro{\zm}{\z-1}
    \pgfmathtruncatemacro{\yp}{\y+1}
    \node (\zm \yp \x) at (\zm,\yp,\x) {$\zm\yp\x$};
    \ifnum \x=2 { \draw (\z \y \x) to node[midway,right] {$\gamma$} (\zm \yp \x);}\fi;
    \draw (\z \y \x) to (\zm \yp \x);
  }
  } \fi;
}
\end{scope}
\end{scriptsize}
\end{tikzpicture}
\caption{The quiver $Q^{(4)}$.} \label{fig:triangle}
\end{figure}
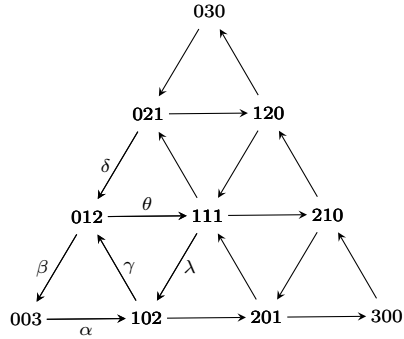

We can choose, for example, $\cal E=\{(0,0,3),(0,1,2),(1,0,2),(1,1,1)\}$ as a set of representatives of vertices. For simplicity we shall denote the elements of this set by $\{1,2,3,4\}$ respectively. Then $Q_G$ (depicted in Figure~\ref{fig:sga_triangle}) has vertices $\eta^1,\eta^2,\eta^3,\eta^4_0,\eta^4_1,\eta^4_2$, which will be denoted respectively by $1,2,3,4^0,4^1,4^2$. We will also rename the arrows of type (1), (2), (3) in $Q$. These are
$$\alpha\colon 1\to 3,\quad \beta\colon 2\to 1,\quad \gamma\colon 3\to 2,\quad \delta\colon g^2(3)\to 2\quad \text{ of type (1),}$$
$$\theta\colon 2\to 4\quad \text{ of type (2),}$$
$$\lambda\colon 4\to 3\quad \text{ of type (3).}$$
We take $\cal C=\{c_1,c_2,c_3\}$, where $c_1=\alpha\beta\gamma$ is of type (i) and $c_2=\lambda\theta\gamma, c_3=\lambda g(\theta)g(\delta)$ are of type (ii). Note that $p(c_2)=0$ and $p(c_3)=1$. Then we get
$$W_G = -\tilde\alpha\tilde\beta\tilde\gamma + \sum_{\mu=0}^2 \tilde\lambda^\mu\tilde\theta^\mu\tilde\gamma - \sum_{\mu=0}^2 \zeta^{-\mu} \tilde\lambda^\mu\tilde\theta^\mu\tilde\delta.$$

\begin{figure}[H]
\centering
\begin{tikzpicture}[->, scale=1.8, baseline=2cm, inner sep=2, shape=circle]
\begin{scriptsize}
\begin{scope}[rotate=-15]

\node (a) at (1,0,0) {$3$};
\node (b) at (0,1,0) {$2$};
\node (c) at (0,0,1) {$1$};

\begin{scope}[xshift=-0.5cm,yshift=1cm]
\node (d0) at (1,1,-1) {$4^0$};
\end{scope}
\begin{scope}[xshift=0,yshift=0]
\node (d1) at (1,1,-1) {$4^1$};
\end{scope}
\begin{scope}[xshift=0.5cm,yshift=-1cm]
\node (d2) at (1,1,-1) {$4^2$};
\end{scope}

\draw (b) to node[midway,above] {$\tilde\beta$} (c);
\draw (c) to node[midway,above] {$\tilde\alpha$} (a);
\draw[transform canvas={yshift=0.15em,xshift=0.2em}] (a) to node[midway,right] {$\tilde\delta$} (b);
\draw[transform canvas={yshift=-0.15em,xshift=-0.2em}] (a) to node[midway,left] {$\tilde\gamma$} (b);

\foreach \i in {0,...,2}
{
  \draw (b) to node[near start,above] {$\tilde\theta^{\i}$} (d\i);
  \draw (d\i) to node[near start,above left] {$\tilde\lambda^{\i}$} (a);
}
\end{scope}
\end{scriptsize}
\end{tikzpicture}
\caption{The quiver $Q^{(4)}_G$.} \label{fig:sga_triangle}
\end{figure}
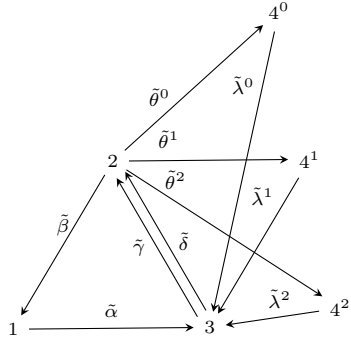

By the results in Section~\ref{sec:dual}, the dual group $\hat G=\langle \chi\rangle$ acts on $Q_G$ as follows.
The vertices $1,2,3$ are fixed, while $\chi(4^\mu)=4^{\mu+1}$, $\mu=0,1,2$.
The arrows $\tilde\alpha,\tilde\beta,\tilde\gamma$ are fixed, $\chi(\tilde\theta^\mu)=\tilde\theta^{\mu+1}$ and $\chi(\tilde\lambda^\mu)=\tilde\lambda^{\mu+1}$, $\mu=0,1,2$.
Since $t(\delta)=2$, we have $\chi(\tilde\delta)=\zeta^2\tilde\delta$. Note that, in the process of getting back the initial quiver using the isomorphism $\phi$ of Proposition~\ref{prop:main2b}, the vertices $4^0,4^1,4^2$ give rise to the vertex $(1,1,1)$ of $Q$, the vertex $2$ gives rise to $(0,1,2),(1,2,0),(2,0,1)$, $3$ to $(0,2,1),(2,1,0),(1,0,2)$ and $1$ to $(0,0,3),(0,3,0),(3,0,0)$.

\end{ex}

\begin{ex}[Self-injective QPs from Postnikov diagrams]\label{ex:postnikov}
  In this example we illustrate Corollary~\ref{prop:combining} and (the proof of) Proposition~\ref{prop:simcon}. Let $Q$ be the quiver of Figure~\ref{fig:416}, with the potential $W$ given by the sum of the 
  clockwise faces minus the sum of the anticlockwise faces.
  Thus $(Q, W)$ is a strongly planar quiver with potential. 
  It is constructed from a rotation-invariant $(4,16)$-Postnikov diagram, see \cite[Figure 19]{Pas17b}. By Theorem~\ref{thm:postnikov}, its Jacobian algebra $\Lambda$ is therefore self-injective, with Nakayama automorphism $\varphi$
induced by a rotation by $\frac{\pi}{2}$. Let us consider the group $G = \langle \varphi^2\rangle$. Then the skew group algebra $\Lambda G$ is Morita equivalent to the Jacobian algebra 
$\mathcal P(Q_G,W_G)$, where $Q_G$ is depicted in Figure~\ref{fig:416/2}. The canvas $X_{(Q_G, W_G)}$ is given by an octahedron in the middle attached to an annulus made of all the remainining faces.
Note that this describes the potential $W_G$ completely up to signs.
This algebra is self-injective with Nakayama automorphism given by $\varphi\otimes 1$, but it is not symmetric since its Nakayama permutation has order 2.

If we instead take the skew group algebra construction with respect to $\langle \varphi\rangle$, we get the quiver of Figure~\ref{fig:416/4}. Its canvas is an annulus consisting of the outer cycles,
attached to four disks sharing their boundary circle. These disks are subdivided into two triangles each. Again note that describing the canvas determines the potential up to fourth roots of unity.
This algebra is symmetric by Corollary~\ref{prop:combining}.

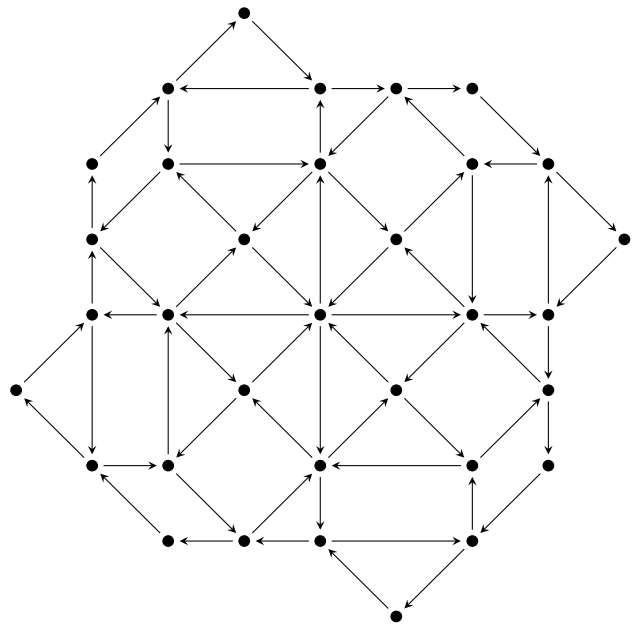
\begin{figure}[H]
\[
\begin{tikzpicture}[baseline=(bb.base),scale=1,shape=circle,
    quivarrow/.style={black, ->, >= stealth,shorten <= 1.5mm, shorten >= 1.5mm},
    cutarrow/.style = {black, -latex,very thick, dotted},
    patharrow/.style = {black,-latex, dotted, very thick},
  ]

\newcommand{\goodarrow}{\arrow{angle 60}}
\newcommand{\nth}{360/4}
\newcommand{\radius}{3.5cm} % radius of boundary circle
\newcommand{\dotrad}{0.07cm} % size of node

% ------ baseline ------
\path (0,0) node (bb) {};

% ------ boundary node labels

\coordinate (x) at (0,0);
  \draw (0,0) node {};
  \draw (x) circle(\dotrad) [fill=black];

\foreach \n in {1,2,3,4}
{
\coordinate (p1\n) at ( {-3*cos(\nth*\n)},{-3*sin(\nth*\n)});
  \draw (p1\n) circle(\dotrad) [fill=black];

\coordinate (p7\n) at ( {-2*cos(\nth*\n)},{-2*sin(\nth*\n)});
  \draw (p7\n) circle(\dotrad) [fill=black];

\coordinate (p8\n) at ( {-cos(\nth*\n)-sin(\nth*\n)},{cos(\nth*\n)-sin(\nth*\n)});
  \draw (p8\n) circle(\dotrad) [fill=black];

\coordinate (p6\n) at ( {-2*cos(\nth*\n)-2*sin(\nth*\n)},{2*cos(\nth*\n)-2*sin(\nth*\n)});
  \draw (p6\n) circle(\dotrad) [fill=black];

\coordinate (p2\n) at ( {-3*cos(\nth*\n)-1*sin(\nth*\n)},{1*cos(\nth*\n)-3*sin(\nth*\n)});
  \draw (p2\n) circle(\dotrad) [fill=black];
  
\coordinate (p3\n) at ( {-3*cos(\nth*\n)-2*sin(\nth*\n)},{2*cos(\nth*\n)-3*sin(\nth*\n)});
  \draw (p3\n) circle(\dotrad) [fill=black];

\coordinate (p4\n) at ( {-2*cos(\nth*\n)-3*sin(\nth*\n)},{3*cos(\nth*\n)-2*sin(\nth*\n)});
  \draw (p4\n) circle(\dotrad) [fill=black];

\coordinate (p5\n) at ( {-1*cos(\nth*\n)-4*sin(\nth*\n)},{4*cos(\nth*\n)-1*sin(\nth*\n)});
  \draw (p5\n) circle(\dotrad) [fill=black];
}

%------------arrows

\foreach \n in {1,2,3,4}
{
  \draw  [quivarrow] (p1\n) to  (p2\n);
\draw [quivarrow] (p2\n) to  (p7\n);
\draw[quivarrow] (p7\n) to  (p1\n);
\draw[quivarrow] (p6\n) to  (p2\n);
\draw[quivarrow] (p2\n) to  (p3\n);
\draw [quivarrow](p3\n) to  (p4\n);
\draw [quivarrow] (p4\n) to  (p6\n);
\draw [quivarrow] (p7\n) to  (p8\n);
\draw [quivarrow] (p8\n) to  (p6\n);
\draw [quivarrow] (p4\n) to  (p5\n);
\draw [quivarrow] (p8\n) to  (x);
\draw [quivarrow] (x) to  (p7\n);

}

\foreach \n in {2,3,4}
{
\draw [quivarrow] (p1\the\numexpr\n-1\relax) to  (p4\n);
\draw [quivarrow] (p7\the\numexpr\n-1\relax) to  (p8\n);
\draw [quivarrow] (p5\n) to (p1\the\numexpr\n-1\relax);
\draw [quivarrow] (p6\n) to (p7\the\numexpr\n-1\relax);
}
\draw [quivarrow] (p14) to  (p41);
\draw [quivarrow] (p74) to  (p81);
\draw [quivarrow] (p51) to (p14);
\draw [quivarrow] (p61) to (p74);

 \end{tikzpicture}
\]
\caption{A self-injective QP with Nakayama automorphism $\varphi$ of order 4.}
\label{fig:416}
\end{figure}

\begin{figure}[H]
\[
\begin{tikzpicture}[baseline=(bb.base),scale =1,
     quivarrow/.style={black, ->, >= stealth,shorten <= 1.5mm, shorten >= 1.5mm},
    cutarrow/.style = {black, -latex,very thick, dotted},
    patharrow/.style = {black,-latex, dotted, very thick
    }
  ]

\newcommand{\goodarrow}{\arrow{angle 60}}
\newcommand{\nth}{360/2}
\newcommand{\dotrad}{0.07cm} % size of node

% ------ baseline ------
\path (0,0) node (bb) {};

% ------ center

\foreach \n in {1,2}
{

\coordinate (c\n) at ( {-.6*cos(\nth*\n)+.5*sin(\nth*\n)},{.5*cos(\nth*\n)-.5*sin(\nth*\n)});
  \draw (c\n) circle(\dotrad) [fill=black];

}

%--------------sides

\foreach \n in {1,2}
{
\coordinate (p1\n) at ( {2*cos(\nth*\n)},{2*sin(\nth*\n)});
  \draw (p1\n) circle(\dotrad) [fill=black];

\coordinate (p2\n) at ( {3*cos(\nth*\n)},{3*sin(\nth*\n)});
  \draw (p2\n) circle(\dotrad) [fill=black];

\coordinate (p3\n) at ( {2*sin(\nth*\n)},{2*cos(\nth*\n)});
  \draw (p3\n) circle(\dotrad) [fill=black];

  \coordinate (p4\n) at ( {-2.5*cos(\nth*\n)-2.5*sin(\nth*\n)},{cos(\nth*\n)-2.5*sin(\nth*\n)} );
  \draw (p4\n) circle(\dotrad) [fill=black];

\coordinate (p5\n) at ( {3*sin(\nth*\n)},{3*cos(\nth*\n)});
  \draw (p5\n) circle(\dotrad) [fill=black];

\coordinate (p6\n) at ( {-2*cos(\nth*\n)-4*sin(\nth*\n)},{4*cos(\nth*\n)-2*sin(\nth*\n)});
  \draw (p6\n) circle(\dotrad) [fill=black];

\coordinate (p7\n) at ( {1*cos(\nth*\n)-4.5*sin(\nth*\n)},{4.5*cos(\nth*\n)+sin(\nth*\n)});
  \draw (p7\n) circle(\dotrad) [fill=black];

\coordinate (p8\n) at ( {3*cos(\nth*\n)-3*sin(\nth*\n)},{3*cos(\nth*\n)+3*sin(\nth*\n)});
  \draw (p8\n) circle(\dotrad) [fill=black];

}

%------------arrows

\foreach \n in {1,2}
{
\draw [quivarrow] (p1\n) to  (p3\n);
\draw [quivarrow] (p5\n) to  (p1\n);
\draw [quivarrow] (p3\n) to  (p5\n);
\draw [quivarrow] (p5\n) to  (p4\n);
\draw [quivarrow] (p4\n) to  (p6\n);
\draw [quivarrow] (p6\n) to  (p7\n);
\draw [quivarrow] (p7\n) to  (p5\n);
\draw [quivarrow] (p7\n) to  (p8\n);
\draw [quivarrow] (p8\n) to  (p2\n);
\draw [quivarrow] (p2\n) to  (p7\n);
\draw [quivarrow] (p1\n) to  (p2\n);

%----------middle

\draw [quivarrow] (p3\n) to  (c1);
\draw [quivarrow] (c1) to  (p1\n);
\draw (c2) edge [-,line width=4pt,draw=white, shorten >= 1.5mm, shorten <= 1.5mm] (p1\n);
\draw (p3\n) edge [-,line width=4pt,draw=white, shorten >= 1.5mm, shorten <= 1.5mm] (c2);
\draw [quivarrow] (c2) to  (p1\n);
\draw [quivarrow] (p3\n) to  (c2);

}

\draw [quivarrow] (p11) to  (p32);
\draw [quivarrow] (p12) to  (p31);
\draw [quivarrow] (p22) to  (p41);
\draw [quivarrow] (p21) to  (p42);
\draw [quivarrow] (p41) to  (p12);
\draw [quivarrow] (p42) to  (p11);

 \end{tikzpicture}
\]
\caption{The quiver of the skew group algebra $\Lambda \langle\varphi^2\rangle$.}
\label{fig:416/2}
\end{figure}
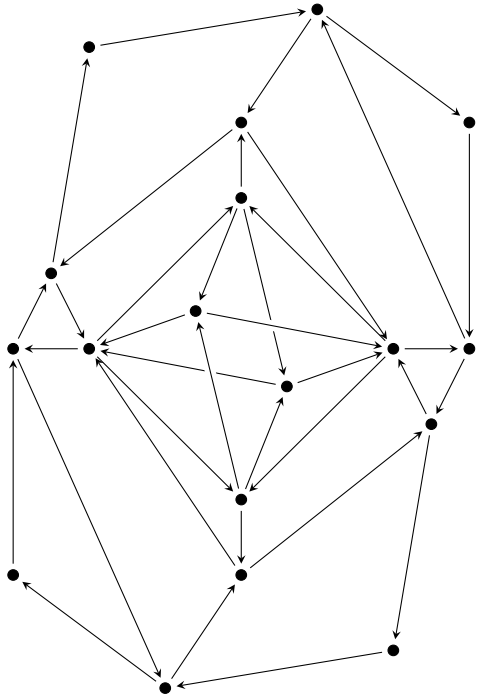

\begin{figure}[H]
\[
\begin{tikzpicture}[baseline=(bb.base),scale = .7,
     quivarrow/.style={black, ->, >= stealth,shorten <= 1.5mm, shorten >= 1.5mm},
    cutarrow/.style = {black, -latex,very thick, dotted},
    patharrow/.style = {black,-latex, dotted, very thick
    }
  ]

\newcommand{\goodarrow}{\arrow{angle 60}}
\newcommand{\eps}{11pt} % offset for boundary labels
\newcommand{\dotrad}{0.1cm} % size of node

% ------ baseline ------
\path (0,0) node (bb) {};

% ------ center

\foreach \n in {1,2,3,4}
{

\coordinate (c\n) at (-2.5+\n,0);
  \draw (c\n) circle(\dotrad) [fill=black];

}

%--------------sides

\coordinate (p1) at ( 0,3);
  \draw (p1) circle(\dotrad) [fill=black];

\coordinate (p2) at ( 0,-3);
  \draw (p2) circle(\dotrad) [fill=black];

\coordinate (p3) at ( 0,-4);
  \draw (p3) circle(\dotrad) [fill=black];

\coordinate (p6) at ( 0,5);
  \draw (p6) circle(\dotrad) [fill=black];

\coordinate (p5) at (-8,0);
  \draw (p5) circle(\dotrad) [fill=black];

\coordinate (p8) at (4,4);
  \draw (p8) circle(\dotrad) [fill=black];

\coordinate (p7) at (5,-4);
  \draw (p7) circle(\dotrad) [fill=black];

\coordinate (p4) at (-2,-5);
  \draw (p4) circle(\dotrad) [fill=black];
%------------arrows

  \draw [quivarrow] (p3) to[out = 180, in=270] (-4,0) to[out =90, in = 170] (p1);
  \draw [quivarrow] (p1) to[out = -10, in=90] (3,0) to[out =-90, in = 0] (p2);
  \draw [quivarrow] (p1) to[out = 190, in=90] (-3,0) to[out =-90, in = 180] (p2);
  \draw [quivarrow] (p3) to[out =0, in=-90] (4,0) to[out =90, in = -90] (p8);
  \draw [quivarrow] (p6) to[out =200, in=90] (-6,0) to[out =-90, in = 160] (p4);
  \draw [quivarrow] (p4) to[out =200, in=-90]  (p5);
  \draw [quivarrow] (p5) to[out =90, in=160]  (p6);
  \draw [quivarrow] (p8) to[bend left]  (p7);
  \draw [quivarrow] (p7) to[bend left]  (p4);

  \draw [quivarrow] (p4) to (p3);
  \draw [quivarrow] (p1) to (p6);
  \draw [quivarrow] (p8) to (p1);
  \draw [quivarrow] (p6) to (p8);
  \draw [quivarrow] (p2) to (p3);

%-----------mid arrows
  \foreach \n in {1,2,3,4}
  {
  
  \draw [quivarrow] (c\n) to (p1);
  \draw [quivarrow] (p2) to (c\n);

  }

 \end{tikzpicture}
\]
\caption{The quiver of the skew group algebra $\Lambda \langle \varphi\rangle$.}
\label{fig:416/4}
\end{figure}
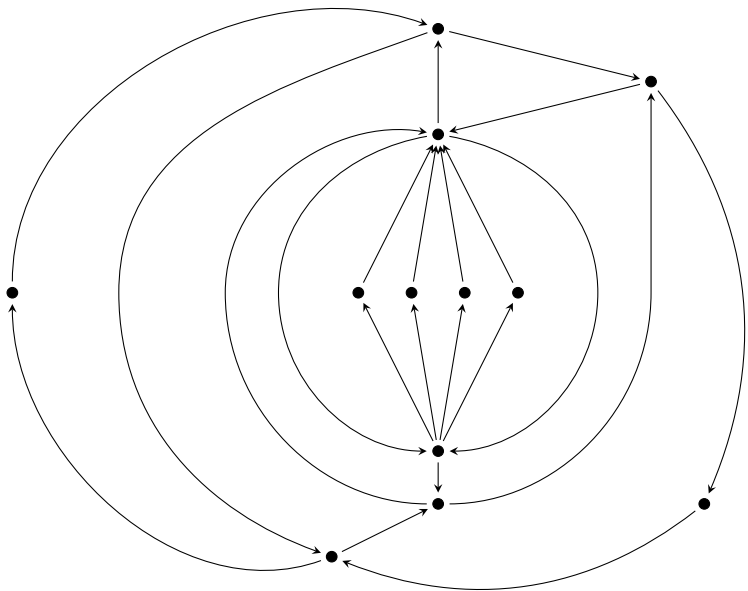

\end{ex}

\subsection{Examples from tensor products of quivers}

The following family of self-injective QPs was introduced in \cite[\S5.2]{HI11b}. Let us recall their definition.

Given two quivers $Q^1,Q^2$ without oriented cycles we can define a new quiver $Q=Q^1\tilde\otimes Q^2$ with $Q_0=Q^1_0 \times Q^2_0$ and $Q_1 = (Q^1_0 \times Q^2_1) \sqcup (Q^1_1 \times Q^2_1) \sqcup (Q^1_1 \times Q^2_0)$. The starting and ending points of the arrows of $Q$ are given by
$$\mathfrak{s}(\alpha,y)=(\mathfrak{s}(\alpha),y),\quad \mathfrak{s}(x,\beta)=(x,\mathfrak{s}(\beta)),\quad \mathfrak{s}(\alpha,\beta)=(\mathfrak{t}(\alpha),\mathfrak{t}(\beta)),$$
$$\mathfrak{t}(\alpha,y)=(\mathfrak{t}(\alpha),y),\quad \mathfrak{t}(x,\beta)=(x,\mathfrak{t}(\beta)),\quad \mathfrak{t}(\alpha,\beta)=(\mathfrak{s}(\alpha),\mathfrak{s}(\beta)),$$
for $x\in Q^1_0$, $y\in Q^2_0$, $\alpha\in Q^1_1$, $\beta\in Q^2_1$. We define a potential on $Q$ by
$$W = W^{\tilde\otimes}_{Q^1,Q^2} = \sum_{\alpha\in Q^1_1, \beta\in Q^2_1} (\alpha,\mathfrak{t}(\beta))(\mathfrak{s}(\alpha),\beta)(\alpha,\beta) - (\mathfrak{t}(\alpha),\beta)(\alpha,\mathfrak{s}(\beta))(\alpha,\beta).$$

Now we consider group actions on $\kk Q$. Let $G_1=\langle g_1\rangle$ and $G_2=\langle g_2\rangle$ be finite cyclic groups and suppose that the following condition holds:
\begin{enumerate}[label=($*$)]
  \item either one of $G_1$ or $G_2$ is trivial, or $G_1\cong G_2$. \label{ass:tensor}
\end{enumerate}
We denote by $n$ the maximum of the orders of $G_1$ and $G_2$. Let $G$ be the subgroup of $G_1\times G_2$ generated by $(g_1,g_2)$, and note that it is cyclic of order $n$.

\begin{lemma}\label{lem:tensor}
  Let $Q^1,Q^2,G_1,G_2$ as above. Suppose we have actions of $G_i$ on $\kk Q^i$, $i=1,2$, which satisfy the assumptions \ref{ass:field}, \ref{ass:permuted}, \ref{ass:cardinality}, \ref{ass:arrows}, and:
  \begin{enumerate}[label=(A3')]
  \item every arrow in $Q^i$ between two fixed vertices is fixed by $G_i$. \label{ass:superfixed}
\end{enumerate}
Then the induced action of $G$ on $\kk Q$ satisfies the assumptions \ref{ass:field}-\ref{ass:cycles}.
\end{lemma}
\begin{proof}
  Assumption \ref{ass:field} holds by the assumptions on the orders of $G_1$ and $G_2$.
  The assumptions \ref{ass:permuted}, \ref{ass:fixed} and \ref{ass:cardinality} follow immediately by hypothesis. Now note that $G$ permutes the cycles of the potential $W^{\tilde\otimes}_{Q^1,Q^2}$, and every cycle is sent to a cycle with the same coefficient. Hence $GW=W$. Finally, assumption \ref{ass:cycles} is satisfied because all cycles have length $3$.
\end{proof}

If $Q^1$ and $Q^2$ are Dynkin quivers with the same Coxeter number and which are stable under their canonical involutions (see \cite[\S5.2]{HI11b} for definitions), then $(Q,W)=(Q^1\tilde\otimes Q^2,W^{\tilde\otimes}_{Q^1,Q^2})$ is a self-injective QP by \cite[Proposition~5.1]{HI11b}.
Let $g_1$ and $g_2$ be the unique automorphisms of, respectively, $Q^1$ and $Q^2$ given by extending to arrows their canonical involutions.

\begin{prop}\label{prop:tensor}
Let $Q^1$ and $Q^2$ be Dynkin quivers which are stable under their canonical involutions and have the same Coxeter number. Let $G$ be the cyclic group generated by $(g_1,g_2)$ and consider the induced action of $G$ on $Q=Q^1\tilde\otimes Q^2$. Then $(Q_G,W_G)$ is a self-injective QP with enough cuts.
\end{prop}
\begin{proof}
  Note that $g_1$ and $g_2$ have order either $1$ or $2$, so the condition \ref{ass:tensor} for $G_1=\langle g_1\rangle$ and $G_2=\langle g_2\rangle$ is satisfied.
  The assumptions \ref{ass:field}, \ref{ass:permuted}, \ref{ass:superfixed}, \ref{ass:cardinality}, and \ref{ass:arrows} for $G_1$ and $G_2$ are immediately checked, so by Lemma~\ref{lem:tensor} we can apply the construction of Section~\ref{sec:setup} to $(Q,W)$ and $G$. By \cite[Proposition~5.1]{HI11b} $(Q,W)$ is self-injective, hence so is $(Q_G,W_G)$.

From the definition of $W^{\tilde\otimes}_{Q^1,Q^2}$ it follows that the subsets $(Q^1_0,Q^2_1)$, $(Q^1_1,Q^2_1)$ and $(Q^1_1,Q^2_0)$ of $Q_1$ are all $G$-invariant cuts. Since every arrow of $Q$ is contained in one of them, we have that $(Q,W)$ has enough $G$-invariant cuts. Hence $(Q_G,W_G)$ has enough cuts by Lemma~\ref{lem:enough}.
\end{proof}

\begin{ex}\label{ex:tensor}
Consider the following Dynkin quivers:

\begin{tikzpicture}[->,scale=2,shape=circle,>=stealth,inner sep=1.5,outer sep=2]
\node[fill] (1) {};
\node[fill,right of=1] (2) {};
\node[fill,right of=2] (3) {};
\node[fill,right of=3] (4) {};
\node[fill,right of=4] (5) {};
\node[left of=1] (a) {$Q^1\colon$};

\draw (2) to (1);
\draw (3) to (2);
\draw (3) to (4);
\draw (4) to (5);

\begin{scope}[xshift=5cm,scale=0.5]
\node[fill] (1) {};
\node[fill] (2) at ([shift={(180:1)}]1)  {};
\node[fill] (3) at ([shift={(300:1)}]1)  {};
\node[fill] (4) at ([shift={(60:1)}]1)  {};
\node[left of=2] (a) {$Q^2\colon$};

\draw (1) to (2);
\draw (1) to (3);
\draw (1) to (4);
\end{scope}
\end{tikzpicture}
\vspace{.5cm}

Here $Q^1$ is of type ${\rm A}_5$ and $Q^2$ of type ${\rm D}_4$, so they have the same Coxeter number. The canonical involution of $Q^1$ is the reflection with respect to the central vertex, while the one of $Q^2$ is the identity. Hence the two quivers are stable and, by Proposition~\ref{prop:tensor}, $(Q_G,W_G)$ is a self-injective QP with enough cuts. The quivers $Q$ and $Q_G$ are illustrated respectively in Figures~\ref{fig:tensor} and \ref{fig:sga_tensor}.

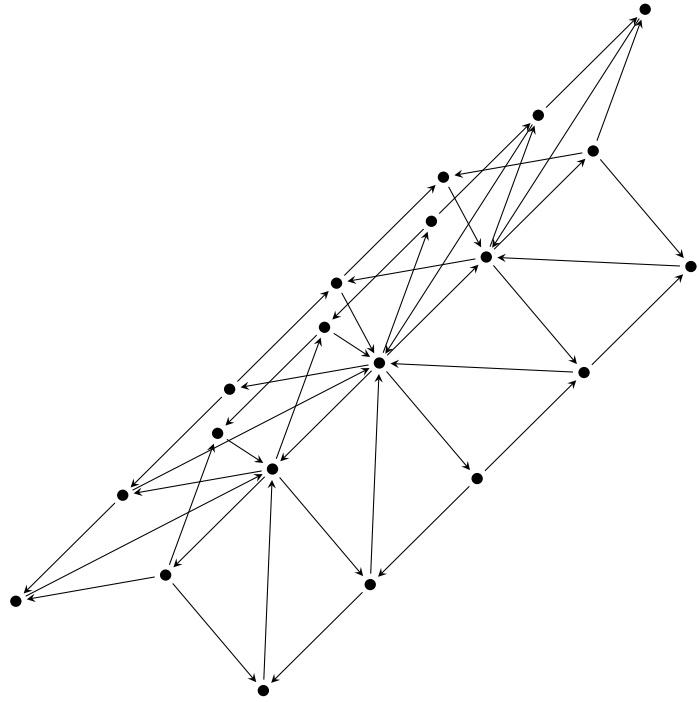
\begin{figure}[h]
\centering
\begin{tikzpicture}[->,scale=2,shape=circle,>=stealth,inner sep=1.5,outer sep=2]
\pgfmathtruncatemacro{\xscale}{20}
\pgfmathtruncatemacro{\yscale}{20}
\foreach \i in {0,...,4}
{
\begin{scope}[xshift=\i*\xscale, yshift=\i*\yscale, rotate=10]
  \node[fill] (z\i) {};
  \node[fill] (a\i) at ([shift={(180:1)}]z\i)  {};
  \node[fill] (b\i) at ([shift={(300:1)}]z\i)  {};
  \node[fill] (c\i) at ([shift={(60:1)}]z\i)  {};
  \draw (z\i) -- (a\i);
  \draw (z\i) -- (b\i);
  \draw (z\i) -- (c\i);
\end{scope}
}
\foreach \i in {0,...,1}
{
  \pgfmathtruncatemacro{\ip}{\i+1}
  \draw (z\ip) -- (z\i);
  \draw (a\ip) -- (a\i);
  \draw (b\ip) -- (b\i);
  \draw (c\ip) -- (c\i);
  \draw (a\i) -- (z\ip);
  \draw (b\i) -- (z\ip);
  \draw (c\i) -- (z\ip);
}
\foreach \i in {3,...,4}
{
  \pgfmathtruncatemacro{\im}{\i-1}
  \draw (z\im) -- (z\i);
  \draw (a\im) -- (a\i);
  \draw (b\im) -- (b\i);
  \draw (c\im) -- (c\i);
  \draw (a\i) -- (z\im);
  \draw (b\i) -- (z\im);
  \draw (c\i) -- (z\im);
}
\end{tikzpicture}
\caption{The quiver $Q^1\tilde\otimes Q^2$.} \label{fig:tensor}
\end{figure}

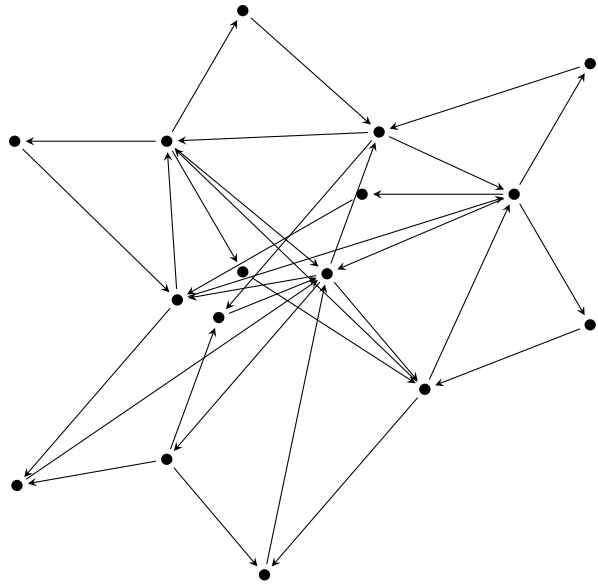
\begin{figure}[h]
\centering
\begin{tikzpicture}[->,scale=2,shape=circle,>=stealth,inner sep=1.5,outer sep=2]
\begin{scope}[xshift=0, yshift=0, rotate=10]
  \node[fill] (z0) {};
  \node[fill] (a0) at ([shift={(180:1)}]z0)  {};
  \node[fill] (b0) at ([shift={(300:1)}]z0)  {};
  \node[fill] (c0) at ([shift={(60:1)}]z0)  {};
  \draw (z0) -- (a0);
  \draw (z0) -- (b0);
  \draw (z0) -- (c0);
\end{scope}
\begin{scope}[xshift=30, yshift=35, rotate=10]
  \node[fill] (z1) {};
  \node[fill] (a1) at ([shift={(180:1)}]z1)  {};
  \node[fill] (b1) at ([shift={(300:1)}]z1)  {};
  \node[fill] (c1) at ([shift={(60:1)}]z1)  {};
  \draw (z1) -- (a1);
  \draw (z1) -- (b1);
  \draw (z1) -- (c1);
\end{scope}
\begin{scope}[xshift=65, yshift=50, rotate=0]
  \node[fill] (z2) {};
  \node[fill] (a2) at ([shift={(180:1)}]z2)  {};
  \node[fill] (b2) at ([shift={(300:1)}]z2)  {};
  \node[fill] (c2) at ([shift={(60:1)}]z2)  {};
  \draw (z2) -- (a2);
  \draw (z2) -- (b2);
  \draw (z2) -- (c2);
\end{scope}
\begin{scope}[xshift=0, yshift=60, rotate=0]
  \node[fill] (z3) {};
  \node[fill] (a3) at ([shift={(180:1)}]z3)  {};
  \node[fill] (b3) at ([shift={(300:1)}]z3)  {};
  \node[fill] (c3) at ([shift={(60:1)}]z3)  {};
  \draw (z3) -- (a3);
  \draw (z3) -- (b3);
  \draw (z3) -- (c3);
\end{scope}

\foreach \i in {0,...,1}
{
  \pgfmathtruncatemacro{\ip}{\i+1}
  \draw (z\ip) -- (z\i);
  \draw (a\ip) -- (a\i);
  \draw (b\ip) -- (b\i);
  \draw (c\ip) -- (c\i);
  \draw (a\i) -- (z\ip);
  \draw (b\i) -- (z\ip);
  \draw (c\i) -- (z\ip);
}

  \draw (z3) -- (z1);
  \draw (a3) -- (a1);
  \draw (b3) -- (b1);
  \draw (c3) -- (c1);
  \draw (a1) -- (z3);
  \draw (b1) -- (z3);
  \draw (c1) -- (z3);
\end{tikzpicture}
\caption{The quiver $(Q^1\tilde\otimes Q^2)_G$.} \label{fig:sga_tensor}
\end{figure}
\end{ex}

All examples we have illustrated so far are related to self-injective QPs. In the next one we will consider a case where the QP we start with is not self-injective.

\begin{ex}\label{ex:nonselfinj}
Consider the Dynkin quivers

\begin{tikzpicture}[->,scale=2,shape=circle,>=stealth,inner sep=0.1]
\node (1) {$2$};
\node[right of=1] (2) {$1$};
\node[right of=2] (3) {$0$};
\node[right of=3] (4) {$1'$};
\node[right of=4] (5) {$2'$};
\node[left of=1] (a) {$Q^1\colon$};

\draw (2) to node[midway,above]{$\beta$} (1);
\draw (3) to node[midway,above]{$\alpha$} (2);
\draw (3) to node[midway,above]{$\alpha'$} (4);
\draw (4) to node[midway,above]{$\beta'$} (5);

\begin{scope}[xshift=4.5cm]
\node (1) {$0$};
\node[right of=1] (2) {$1$};
\node[right of=2] (3) {$2$};
\node[left of=1] (a) {$Q^2\colon$};

\draw (3) to node[midway,above]{$\gamma_2$} (2);
\draw (2) to node[midway,above]{$\gamma_1$} (1);
\end{scope}
\end{tikzpicture}

and let $Q=Q^1\tilde\otimes Q^2$ (see Figure~\ref{fig:ausl}).

\begin{figure}[h]
\centering
\begin{tikzpicture}[->,scale=3,shape=circle,>=stealth,inner sep=0.1]
\begin{scriptsize}
\pgfmathtruncatemacro{\xscale}{30}
\foreach \i in {0,...,2}
{
\begin{scope}[xshift=\i*\xscale,node distance=2*\xscale]
  \node (2\i) {$(2,\i)$};
  \node[below of=2\i] (1\i) {$(1,\i)$};
  \node[below of=1\i] (0\i) {$(0,\i)$};
  \node[below of=0\i] (1'\i) {$(1',\i)$};
  \node[below of=1'\i] (2'\i) {$(2',\i)$};
  \draw (0\i) -- node[midway,left]{$(\alpha,\i)$} (1\i);
  \draw (0\i) -- node[midway,left]{$(\alpha',\i)$} (1'\i);
  \draw (1\i) -- node[midway,left]{$(\beta,\i)$} (2\i);
  \draw (1'\i) -- node[midway,left]{$(\beta',\i)$} (2'\i);
\end{scope}
}
\foreach \i in {0,...,1}
{
  \pgfmathtruncatemacro{\ip}{\i+1}
  \draw (2\ip) -- node[midway,above=-7]{$(2,\gamma_\ip)$} (2\i);
  \draw (1\ip) -- node[midway,above=-7]{$(1,\gamma_\ip)$} (1\i);
  \draw (0\ip) -- node[midway,above=-7]{$(0,\gamma_\ip)$} (0\i);
  \draw (1'\ip) -- node[midway,above=-7]{$(1',\gamma_\ip)$} (1'\i);
  \draw (2'\ip) -- node[midway,above=-7]{$(2',\gamma_\ip)$} (2'\i);
  \draw (2\i) -- node[midway,above]{$(\beta,\gamma_\ip)$} (1\ip);
  \draw (1\i) -- node[midway,above]{$(\alpha,\gamma_\ip)$} (0\ip);
  \draw (2'\i) -- node[midway,above]{$(\alpha',\gamma_\ip)$} (1'\ip);
  \draw (1'\i) -- node[midway,above]{$(\beta',\gamma_\ip)$} (0\ip);
}
\end{scriptsize}
\end{tikzpicture}
\caption{The quiver $Q^1\tilde\otimes Q^2$.} \label{fig:ausl}
\end{figure}
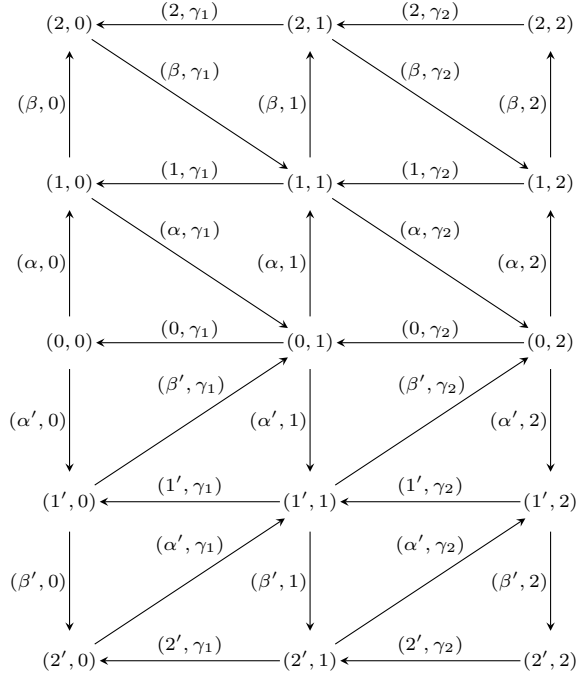

Let $g$ be the unique automorphism of $Q^1$ given on vertices by $g(0)=0$, $g(i)=i'$ and $g(i')=i$, $i=1,2$. Then we can consider the action of the cyclic group $G=\langle(g,\on{id})\rangle$ of order $2$ on $Q$. If we apply the construction of Section~\ref{sec:setup} choosing as a set of representatives of the vertices $\cal E = \{(i,j) \,|\, i,j=0,1,2\}$, then we obtain the quiver $Q_G$ of Figure~\ref{fig:sga_ausl}. We can take
$$\cal C = \{(\alpha,i-1)(0,\gamma_i)(\alpha,\gamma_i), (1,\gamma_i)(\alpha,i)(\alpha,\gamma_i), (\beta,i-1)(0,\gamma_i)(\beta,\gamma_i), (1,\gamma_i)(\beta,i)(\beta,\gamma_i) \,|\, i=1,2\}$$
and obtain the potential
\begin{align*}
  W_G  &= \sum_{i=1}^2 \widetilde{(\beta,i-1)}\widetilde{(0,\gamma_i)}\widetilde{(\beta,\gamma_i)} - \widetilde{(1,\gamma_i)}\widetilde{(\beta,i)}\widetilde{(\beta,\gamma_i)} + \\
  &+ \sum_{i=1}^2 \sum_{\mu=0}^1 \widetilde{(\beta,i-1)}\widetilde{(0,\gamma_i)}^\mu\widetilde{(\beta,\gamma_i)} - \widetilde{(1,\gamma_i)}^\mu\widetilde{(\beta,i)}\widetilde{(\beta,\gamma_i)}.
\end{align*}

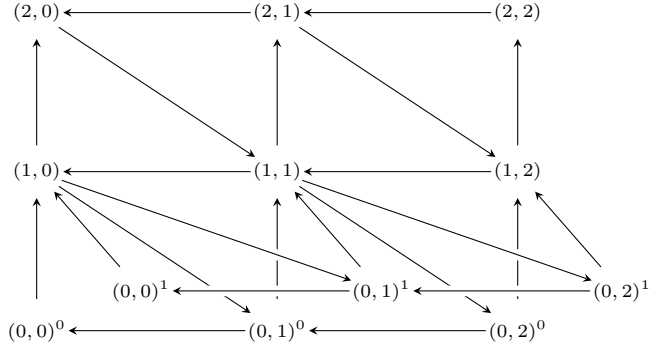
\begin{figure}[H]
\centering
\begin{tikzpicture}[->,scale=3,shape=circle,>=stealth,inner sep=0.1]
\begin{scriptsize}
\pgfmathtruncatemacro{\xscale}{30}
\foreach \i in {0,...,2}
{
\begin{scope}[xshift=\i*\xscale,node distance=2*\xscale]
  \node (0\i) {$(0,\i)^0$};
  \node[above of=0\i] (1\i) {$(1,\i)$};
  \node[above of=1\i] (2\i) {$(2,\i)$};
  \draw (0\i) -- (1\i);
  \draw (1\i) -- (2\i);
  \begin{scope}[xshift=13,yshift=5]
    \node (0'\i) {$(0,\i)^1$};
  \end{scope}
\end{scope}
}
\foreach \i in {0,...,1}
{
  \pgfmathtruncatemacro{\ip}{\i+1}
  \draw (2\ip) -- (2\i);
  \draw (1\ip) -- (1\i);
  \draw (0\ip) -- (0\i);
  \draw (2\i) -- (1\ip);
  \draw (1\i) -- (0\ip);
}
\foreach \i in {0,...,2}
{
    \draw (0'\i) edge [-,line width=4pt,draw=white, shorten >= 2mm, shorten <= 2mm] (1\i);
    \draw (0'\i) -- (1\i);
}
\foreach \i in {0,...,1}
{
  \pgfmathtruncatemacro{\ip}{\i+1}
  \draw (0'\ip) edge [-,line width=4pt,draw=white] (0'\i);
  \draw (0'\ip) -- (0'\i);
  \draw (1\i) edge [-,line width=4pt,draw=white] (0'\ip);
  \draw (1\i) -- (0'\ip);
}
\end{scriptsize}
\end{tikzpicture}
\caption{The quiver $(Q^1\tilde\otimes Q^2)_G$.} \label{fig:sga_ausl}
\end{figure}

\begin{rmk}
We may choose another basis for $\on{rad} \cal P(Q,W) / \on{rad}^2 \cal P(Q,W)$ by replacing $(\alpha',i)$ with $-(\alpha',i)$ and $(\beta',i)$ with $-(\beta',i)$, $i=0,1,2$.
In this way we get that $\cal P(Q,W) \cong \cal P(Q,W')$, where $W'$ is the potential defined as the sum of all the clockwise $3$-cycles minus the sum of all the anticlockwise ones.
We have an action of $G$ on $\cal P(Q,W')$ such that $\cal P(Q,W)G \cong \cal P(Q,W')G$, but note that in this case the assumption \ref{ass:arrows} is no longer satisfied.

Now let us consider the $G$-invariant cut $C=Q^0\times Q^1$ in $Q$. We may note that the truncated Jacobian algebra $\cal P(Q,W')_C$ is isomorphic to the Auslander algebra of $Q^1$.
Moreover, by Proposition~\ref{prop:cut_truncated} and what we observed above, we have that $(\cal P(Q,W)_C)G \cong (\cal P(Q,W')_C)G$ is Morita equivalent to $\cal P(Q_G,W_G)_{C_G}$. Notice that $\cal P(Q_G,W_G)_{C_G}$ is isomorphic to the Auslander algebra of a Dynkin quiver of type ${\rm D}_4$. This is no surprise, since we know by \cite[Theorem~1.3(c)(iv)]{RR85} that skew group algebras of Auslander algebras are again Auslander algebras.
\end{rmk}

\end{ex}

\newpage
\bibliographystyle{alpha}
\bibliography{Bibliography}

\end{document}